\documentclass[a4paper,10pt]{article}
\usepackage{eurosym}
\usepackage{makeidx}
\usepackage{amsfonts}
\usepackage{amsmath}
\usepackage{graphicx}
\usepackage{amssymb}
\usepackage{mathrsfs}
\usepackage{graphicx}
\usepackage[all,cmtip]{xy}
\usepackage{synttree}
\usepackage{tikz}
\usepackage{color}

\providecommand{\U}[1]{\protect\rule{.1in}{.1in}}
\newtheorem{theorem}{Theorem}[section]

\newtheorem{corollary}[theorem]{Corollary}

\newtheorem{definition}[theorem]{Definition}
\newtheorem{example}[theorem]{Example}
\newtheorem{lemma}[theorem]{Lemma}

\newtheorem{proposition}[theorem]{Proposition}

\newtheorem{remark}[theorem]{Remark}

\newenvironment{proof}[1][Proof]{\noindent \emph{#1.} }{\hfill \ \rule{0.5em}{0.5em}}
\newcommand{\BIGOP}[1]{\mathop{\mathchoice{\raise-0.22em\hbox{\huge $#1$}} {\raise-0.05em\hbox{\Large $#1$}}{\hbox{\large $#1$}}{#1}}}

\makeatletter\@addtoreset{equation}{section}\makeatother
\makeatletter\@addtoreset{figure}{section}\makeatother
\makeatletter\@addtoreset{table}{section}\makeatother
\begin{document}

\title{On the Dirac-Frenkel Variational Principle on Tensor Banach Spaces}
\author{Antonio Falc\'{o}\footnote{Corresponding author}$^{\, \, \,,1,3},$ Wolfgang Hackbusch$^{2}$ and Anthony Nouy$%
^{3}$ \\
%EndAName
$^{1}$ ESI International Chair@CEU-UCH, \\
Departamento de Matem\'aticas, F\'{\i}sica y Ciencias
Tecnol\'ogicas,\\
Universidad Cardenal Herrera-CEU, CEU Universities \\
San Bartolom\'e 55,
46115 Alfara del Patriarca (Valencia), Spain\\
e-mail: \texttt{afalco@uchceu.es}\\
$^{2}$ Max-Planck-Institut \emph{Mathematik in den Naturwissenschaften}\\
Inselstr. 22, D-04103 Leipzig, Germany \\
e-mail: \texttt{{wh@mis.mpg.de} }\\
$^{3}$ Centrale Nantes, \\
LMJL UMR CNRS 6629\\
1 rue de la No\"e,
44321 Nantes Cedex 3, France.\\
e-mail: \texttt{anthony.nouy@ec-nantes.fr}}
\date{}
\maketitle

\begin{abstract}
The main goal of this paper is to extend the so-called Dirac-Frenkel 
Variational Principle in the framework of tensor Banach spaces.
To this end we observe that a tensor product of
normed spaces can be described as a union of disjoint 
connected components. Then we show that 
each of these connected components, composed by tensors in Tucker format
with a fixed rank, is a Banach manifold modelled in a particular Banach space, 
for which we provide local charts.
The description of the local charts of these manifolds 
is crucial for an algorithmic
treatment of high-dimensional partial differential equations and minimization problems.
In order to describe the relationship between these manifolds and the
natural ambient space we prove under natural conditions that each
connected component can be immersed in a particular ambient Banach space. 
This fact allows us to finally extend the Dirac-Frenkel variational
principle in the framework of topological tensor spaces.
\end{abstract}



%\noindent Communicated by Joseph M. Landsberg

\section{Introduction}

Tensor approximation methods play a central role in the numerical solution 
of high-dimensional problems arising in a wide range of applications. 
Low-rank tensor formats based on subspaces are widely used for complexity reduction 
in the representation of high-order tensors.
The construction of these formats are usually based on a hierarchy of tensor product
subspaces spanned by orthonormal bases, because in most cases a hierarchical
representation fits with the structure of the mathematical model and facilitates
its computational implementation.
Two of the most popular formats are the
Tucker format and the Hierarchical Tucker format \cite{HaKuehn2009} (HT for
short). It is possible to show that the Tensor Train format \cite{Osedelets1}
(TT for short), introduced originally by Vidal \cite{Vidal}, is a particular
case of the HT format (see e.g. Chapter 12 in \cite{Hackbusch}). An
important feature of these formats, in the framework of topological tensor
spaces, is the existence of a best approximation in each fixed set of
tensors with bounded rank \cite{FALHACK}. In particular, it allows us to
construct, on a theoretical level, iterative minimisation methods for
nonlinear convex problems over reflexive tensor Banach spaces \cite%
{FalcoNouy}.

\bigskip

This paper is devoted to the use of the geometric structure of
 the Tucker format to construct reduced order
models of ordinary differential equations defined over tensor Banach spaces.
The Dirac-Frenkel variational principle is a well known tool in the 
numerical treatment of equations of quantum dynamics. 
It was originally proposed by Dirac and Frenkel in 1930 
to approximately solve the time-dependent Schr\"odinger equation.
It assumes the existence of a vector field (ordinary differential equation) over a 
configuration space represented by
a Hilbert space. This configuration space contains an immersed 
submanifold and the reduced order model is then obtained by 
projecting the vector field at each point of the submanifold 
onto its tangent space. Tucker tensors of fixed rank are used in the above framework
for the discretisation of differential 
equations arising in
quantum chemical problems or in the multireference Hartree and Hartree-Fock
methods (MR-HF) in quantum dynamics \cite{Lubish}. In particular, for finite-dimensional ambient tensor spaces, 
it can be shown that the set of Tucker tensors of fixed rank forms an immersed
finite-dimensional quotient manifold \cite{KoLu1}. A similar approach in a
complex Hilbert space setting for Tucker tensors of fixed rank is given in 
\cite{BCMT}. Then the numerical treatment of this class of problems follows
the general concepts of differential equations on manifolds \cite{HLW}.
Recently, similar results have been obtained for the TT format \cite{HRS}
and the HT format \cite{Uschmajew2012} (see also \cite{AJ}). The term
"matrix-product state" (MPS) was introduced in quantum physics (see, e.g., 
\cite{VC}). The related tensor representation can be found already in \cite%
{Vidal} without a special naming of the representation. The method has been
reinvented by Oseledets and Tyrtyshnikov (see \cite{Osedelets}, \cite%
{Osedelets1}, and \cite{OT}) and called "TT decomposition". For matrix
product states (MPS), the differential geometry in a finite-dimensional
complex Hilbert space setting is covered in \cite{HMOV}. 
Two commonly accepted facts are the following.
\begin{enumerate}
\item[(a)] Even if it can be shown in finite dimension that the set of Tucker
tensors with bounded rank is closed, the existence of a manifold structure for this set is
an open question. 
Thus the existence of minimisers over this
set can be shown, however, no first order optimality conditions are
available from a geometric point of view. 
\item[(b)] Even if either in finite dimension or in a Hilbert space
setting it can be shown that the set of Tucker (respectively, in finite
dimensions HT) tensors with fixed rank is a quotient manifold, 
an explicit parametrisation in order to provide a
manifold structure is not known.
\end{enumerate}

In our opinion, these two facts are due to the lack of a common mathematical
framework for developing a mathematical analysis of these
abstract objects. The main goal of this paper is to provide this common framework by
means of some of the tools developed in \cite{FALHACK} by some of the authors 
of this article in order to extend the Dirac-Frenkel variational method
to the framework of tensor Banach spaces.

Our starting point are the following natural questions that arise in 
the mathematical theory of tensor spaces. The first question is: It is 
possible to construct a parametrisation for the set of
tensors of fixed rank in order to show that it is a
true manifold even in the infinite-dimensional case? In a second step, if the answer
is positive, we would like to ask: Is the set of tensors
of fixed rank an immersed submanifold of the
topological tensor space, as ambient manifold, under consideration?
Finally, if the above two questions have positive answers, we would like to  extend
the Dirac-Frenkel variational principle on tensor Banach spaces.
\\

The paper is organised as follows. 
%From Sect.~\ref{Tensor_TBF} to Sect.~\ref%
%{Dirac_Frenkel}, we give the following contributions of this paper. More
%precisely,

\begin{itemize}
\item In Sect. \ref{preliminary}, we introduce some important 
definitions and results that we will use widely along this paper.
In particular we introduce Banach manifolds not modelled in a particular
Banach space and we give as example the Grassmann manifold of a Banach space introduced
by A. Douady \cite{Douady66} in 1966. Our main contribution is
to give a Banach manifold structure to the set of subspaces of a normed space 
with a fixed finite dimension.

\item In Sect. \ref{sec:banach_manifold_tucker_fixed_rank}, we introduce the set
of tensors in Tucker format with fixed rank over a tensor product space 
of normed spaces. We prove that if the tensor product space has a norm such that 
the tensor product is continuous, with respect to that norm, then the
set of tensors in Tucker format with fixed rank is a $\mathcal{C}^{\infty}$-Banach 
manifold modelled on a particular Banach space. We point out that the regularity of
the manifold depends on the regularity of the tensor product considered as
a multilinear map between normed spaces. Even if a continuous multilinear map between
complex Banach spaces is always analytic, under the authors' knowledge, for a continuous
multilinear map between normed spaces we can only obtain a 
$\mathcal{C}^{\infty}$--differentiability. 
An interesting remark is that the geometric structure
is independent of the choice of the norm on the tensor product space. We illustrate
this fact by means of an example using Sobolev spaces. Finally, we 
show under the above conditions that a tensor space of normed spaces 
is a $\mathcal{C}^{\infty}$-Banach manifold not modelled on a particular Banach space.

\item In Sect. \ref{embedded_manifold}, we discuss the choice of a norm in
the ambient tensor Banach space to prove that the set of tensors with fixed Tucker rank is 
an immersed submanifold of that space (considered as Banach manifold).
To this end we assume the existence of a norm over the tensor space not
weaker than the injective norm . The same assumption is  used
in \cite{FALHACK} to prove the existence of a best approximation in the
Tucker case. Then we  show that the set of tensors in Tucker format with fixed rank
is an immersed submanifold of the ambient tensor Banach space. This fact
is far from trivial. The main difficulty is to prove that the tangent space
is a closed and complemented subspace of the ambient tensor Banach space 
under consideration. In a Hilbert space, 
every closed subspace is complemented, but  this fact is not true in a Banach 
(non Hilbert) space.

\item In Sect. \ref{Dirac_Frenkel}, we give a formalisation in this
framework of the multi--configuration time--dependent Hartree (MCTDH) method
(see \cite{Lubish}) in tensor Banach spaces.
\end{itemize}

\section{The Grassmann-Banach manifold and its relatives}
\label{preliminary}

In this section we introduce some important definitions and results that
we will use (and elsewhere): throughout this paper.

\bigskip

In the following, $X$ is either a normed space or a Banach space with norm $\left\Vert
\cdot\right\Vert .$ We denote by $X^\ast$ the topological dual of $X$. 
The dual norm $\left\Vert \cdot\right\Vert _{X^{\ast}}$
on $X^{\ast}$ is%
\begin{equation}
\left\Vert \varphi\right\Vert _{X^{\ast}}=\sup\left\{ \left\vert
\varphi(x)\right\vert :x\in X\text{ with }\left\Vert x\right\Vert
_{X}\leq1\right\} =\sup\left\{ \left\vert \varphi(x)\right\vert /\left\Vert
x\right\Vert _{X}:0\neq x\in X\right\} .  \label{(Dualnorm}
\end{equation}
We recall that if $X$ is a normed space, then $X^*$ is always a Banach
space.

\bigskip

By $\mathcal{L}(X,Y)$ we denote the space of continuous linear mappings from 
$X$ into $Y.$ The corresponding operator norm is written as $\left\Vert
\cdot\right\Vert _{Y\leftarrow X}.$ If $X$ and $Y$ are normed spaces then 
$(\mathcal{L}(X,Y),\|\cdot\|_{Y\leftarrow X})$ is a normed space. It is well known
that if $Y$ is a Banach space then $(\mathcal{L}(X,Y),\|\cdot\|_{Y\leftarrow X})$
is also a Banach space. 

\bigskip

Let $X_1,\ldots,X_d$ and $Y$ be normed spaces and 
$M:\mathop{\mathchoice{\raise-0.22em\hbox{\huge $\times$}} {\raise-0.05em\hbox{\Large $\times$}}{\hbox{\large
$\times$}}{\times}}_{\alpha=1}^d X_\alpha \rightarrow Y.$ We will say that $M$ is a
multilinear map if for each fixed $\alpha \in \{1,2,\ldots,d\}$, 
$$
x_\alpha \mapsto M(x_1,\ldots,x_{\alpha-1},x_\alpha,x_{\alpha+1},\ldots,x_d)
$$
is a linear map from $X_{\alpha}$ to $Y$ for all $(x_1,\ldots,x_{\alpha-1},x_{\alpha+1},\ldots,x_d) \in 
\mathop{\mathchoice{\raise-0.22em\hbox{\huge $\times$}} {\raise-0.05em\hbox{\Large $\times$}}{\hbox{\large
$\times$}}{\times}}_{k \in \{1,2,\ldots,d\}\setminus \{\alpha\}} X_k.$
Recall that a multilinear map $M$ from 
$\mathop{\mathchoice{\raise-0.22em\hbox{\huge $\times$}} {\raise-0.05em\hbox{\Large $\times$}}{\hbox{\large
$\times$}}{\times}}_{\alpha=1}^d (X_\alpha,\|\cdot\|_\alpha)$
equipped with the product topology $\|\cdot\|$ to a normed space $(Y,\Vert\cdot\Vert_Y)$ 
is continuous if and only
if $\Vert M\Vert <\infty$, with 
\begin{equation*}
\Vert M\Vert :=\sup_{\substack{ (x_{1},\hdots,x_{d}) \\ \Vert x_{1}\Vert
_{1}\leq 1,\hdots,\Vert x_{d}\Vert _{d}\leq 1}}\Vert M(x_{1},\ldots
,x_{d})\Vert _{Y}=\sup_{{(x_{1},\hdots,x_{d})}}\frac{\Vert M(x_{1},\ldots
,x_{d})\Vert _{Y}}{\Vert x_{1}\Vert _{1}\hdots\Vert x_{d}\Vert _{d}}.
\end{equation*}%
A useful result is the following (see Proposition 79 in \cite{HJ}).

\begin{proposition}\label{multilinear}
Let $X_1,\ldots,X_d$ and $Y$ be normed spaces and 
$M:\mathop{\mathchoice{\raise-0.22em\hbox{\huge $\times$}} {\raise-0.05em\hbox{\Large $\times$}}{\hbox{\large
$\times$}}{\times}}_{\alpha=1}^d X_\alpha \rightarrow Y$ be a continuous multilinear map. Then
$M$ is $\mathcal{C}^{\infty}$-Fr\'echet differentiable and $D^{k}M(x_1,\ldots,x_d)=0$ for
all $(x_1,\ldots,x_d) \in \mathop{\mathchoice{\raise-0.22em\hbox{\huge $\times$}} {\raise-0.05em\hbox{\Large $\times$}}{\hbox{\large
$\times$}}{\times}}_{\alpha=1}^d X_\alpha$ and $k > d.$
\end{proposition}

Assume that $X$ and $Y$ are Banach spaces and let $U\subset X$ be an open connected set. Then 
a map $f:U \subset X \rightarrow X$ is an analytic map in $U$ if and only if for
 $x \in U$ and $\varphi \in Y^*$, there exists a neighbourhood of $0,$ namely $V(0)\subset \mathbb{K},$ 
where $\mathbb{K}$ is either $\mathbb{R}$ or $\mathbb{C},$ such that the map 
$$
V(0) \subset \mathbb{K} \rightarrow \mathbb{K}, \quad t \mapsto \varphi(f(x+th))
$$
is analytic. An immediate consequence of this definition is the fact that for $|t|$
sufficiently small and $x \in U,$
$$
\varphi(f(x+th)) = \sum_{n=0}^{\infty}a_n(x,h)\frac{t^n}{n!},
$$
where
$$
a_n(x,h) = \left.\frac{d^n}{dt^n}\varphi(f(x+th))\right|_{t=0}.
$$

The following result characterises an analytic function defined over complex Banach spaces (see Theorem 160 in \cite{
HJ}).

\begin{proposition}\label{analytic}
Let $X,Y$ be complex Banach spaces, $U \subset 􏰎X$ open, and $f:U\subset X \rightarrow Y.$ Then
$f$ is analytic if and only if $f$ is $\mathcal{C}^{1}$-Fr\'echet differentiable.
\end{proposition}

\begin{corollary}
Let $X_1,\ldots,X_d$ and $Y$ be complex Banach spaces and 
$M:\mathop{\mathchoice{\raise-0.22em\hbox{\huge $\times$}} {\raise-0.05em\hbox{\Large $\times$}}{\hbox{\large
$\times$}}{\times}}_{\alpha=1}^d X_\alpha \rightarrow Y$ be a continuous multilinear map. Then
$M$ is analytic.
\end{corollary}

\begin{definition}
Let $X$ be a Banach space and $P \in \mathcal{L}(X,X).$  
We say that $P$ is a projection if and only if 
$P\circ P = P$ holds. In this situation we also say that $P$ is a
projection from $X$ onto $P(X):=\mathrm{Im}\,P$ parallel to $\mathrm{Ker}\, P.$
\end{definition}

From now on, we will denote $P \circ P = P^2.$ Observe that if $P$ is a
projection then $id_X-P$ is also a projection. Moreover, $id_X-P$ is parallel
to $\mathrm{Im}\,P.$

Observe that each projection gives rise to a pair of subspaces,
namely $U=\mathrm{Im}\,P$ and $W=\mathrm{Ker}\,P$ such that $X=U\oplus W.$ It allows us to introduce the following definitions.

\begin{definition}
A subspace $U$ of a Banach space $X$ is said to be complemented in $X$ if there is a projection $P\in \mathcal{L}(X,X)$ from $X$ onto $U.$
\end{definition}

\begin{definition}
Let $U$ be a closed subspace of $X.$ 
We say that $U$ is a split subspace of $X$ 
if there exists $W,$ called (topological) complement of
$U$ in $X,$ such that 
$X=U \oplus W$ and $W$ is a closed subspace of
$X.$ Moreover, we will say that $(U,W)$ is a pair of
complementary subspaces of $X.$
\end{definition}

Corresponding to each pair $(U,W)$ of complementary subspaces, there is a
projection $P$ mapping $X$ onto $U$ along $W,$ defined as follows. Since for
each $x$ there exists a unique decomposition $x=u+w,$ where $u\in U$ and $%
w\in W,$ we can define a linear map $P(u+w):=u,$ where $\mathrm{Im}\,P=U$
and $\mathrm{Ker}\,P=W.$ Moreover, $P^{2}=P.$  In
Proposition \ref{characterize_P} it will follow that 
$P\in \mathcal{L}(X,X).$

\begin{definition}
The \emph{Grassmann manifold} of a Banach space $X,$ denoted by $\mathbb{G}(X),$ is the set of split subspaces of $X.$
\end{definition}

$U\in \mathbb{G}(X)$ holds if and only if $U$ is a closed subspace and there
exists a closed subspace $W$ in $X$ such that $X=U\oplus W.$ Observe that $X$
and $\{0\}$ are in $\mathbb{G}(X).$ Moreover, by the proof of Proposition 4.2 of \cite{FHHM}, the following result can be shown.

\begin{proposition}
\label{characterize_P} Let $X$ be a Banach space. The following conditions
are equivalent:

\begin{enumerate}
\item[(a)] $U \in \mathbb{G}(X).$

\item[(b)] $U$ is a closed subspace and there exists $P \in \mathcal{L}(X,X)$ such that $P^2=P$ and $%
\mathrm{Im}\, P = U.$

\item[(c)] There exists $Q \in \mathcal{L}(X,X)$ such that $Q^2=Q$ and $%
\mathrm{Ker}\, Q = U.$
\end{enumerate}
\end{proposition}

Moreover, from Theorem 4.5 in \cite{FHHM}, the following result can be shown.

\begin{proposition}
\label{finite_dim_subspaces} Let $X$ be a Banach space. Then every
finite-dimensional subspace $U$ belongs to $\mathbb{G}(X).$
\end{proposition}

Let $W$ and $U$ be closed subspaces of a Banach space $X$ such that $%
X=U\oplus W.$ From now on, we will denote by $P_{_{U\oplus W}}$ the
projection onto $U$ along $W.$ 
Then we have $P_{_{W\oplus U}}=id_{X}-P_{_{U\oplus W}}.$ Let $U,U^{\prime
}\in \mathbb{G}(X).$ We say that $U$ and $U^{\prime }$ have a common
complementary subspace in $X$ if $X=U\oplus W=U^{\prime }\oplus W$ for some 
$W\in \mathbb{G}(X).$ The following two results will 
be useful (for the first one see Lemma 2.1 in \cite{DRW}).

\begin{lemma}
\label{Char_Projections} Let $X$ be a Banach space and assume that $W$, $U$,
and $U^{\prime }$ are in $\mathbb{G}(X).$ Then the following statements are
equivalent:

\begin{enumerate}
\item[(a)] $X=U\oplus W=U^{\prime }\oplus W,$ i.e., $U$ and $U^{\prime }$
have a common complement in $X$.

\item[(b)] $P_{_{U\oplus W}}|_{U^{\prime }}:U^{\prime }\rightarrow U$ has an
inverse.
\end{enumerate}

Furthermore, if $Q = \left(P_{_{U \oplus W}}|_{_{U^{\prime }}}\right)^{-1},$
then $Q$ is bounded and $Q=P_{_{U^{\prime }\oplus W}}|_{_{U}}.$
\end{lemma}

We recall that an algebra is unital or unitary if it has an identity element with respect to the multipli-cation.

\begin{proposition}\label{niceinverse}
Let $X$ be a Banach space and $U,W \in \mathbb{G}(X)$ be such that 
$X=U\oplus W$ and consider the linear space
$$
\mathcal{L}_{(U,W)}(X,X):=\{ P_{_{W\oplus U}} \circ S \circ  P_{_{U\oplus W}}:S \in \mathcal{L}(X,X)\}.
$$
Then the bounded linear map
$$
\mathcal{L}(U,W) \longrightarrow \mathcal{L}_{(U,W)}(X,X),\quad  L \mapsto P_{_{W\oplus U}} \circ L \circ  P_{_{U\oplus W}}
$$
is an isometry. Moreover, for all $L,L' \in \mathcal{L}_{(U,W)}(X,X)$  
it holds that $L \circ L' = L' \circ L = 0$. Then $\mathcal{L}_{(U,W)}(X,X)$ is a sub-algebra of
the unital Banach algebra $\mathcal{L}(X,X)$ and
$$
\exp(L) = \sum_{n=0}^{\infty}\frac{L^n}{n!} = id_{X} + L \text{ and } \exp(-L)=id_{X} - L = (id_{X} + L)^{-1}. 
$$
\end{proposition}

\begin{proof}
Clearly, the map is a linear isomorphism and since
$$
\|L\|_{W \leftarrow U} = \|P_{_{W\oplus U}} \circ L \circ  P_{_{U\oplus W}}\|_{X \leftarrow X},
$$
it is an isometry. For $L = P_{_{W\oplus U}} \circ S \circ  P_{_{U\oplus W}} \in \mathcal{L}_{(U,W)}(X,X)$ and $L' = P_{_{W\oplus U}} \circ S' \circ  P_{_{U\oplus W}} \in \mathcal{L}_{(U,W)}(X,X)$
we have
$$
L \circ L' = P_{_{W\oplus U}} \circ S \circ  P_{_{U\oplus W}} \circ P_{_{W\oplus U}} \circ S' \circ  P_{_{U\oplus W}} = 0,
$$
because $ P_{_{U\oplus W}} \circ P_{_{W\oplus U}} = 0,$ then the second statement holds and the final statement follows in a straightforward way.
\end{proof}

\bigskip

Next, we recall the definition of a Banach manifold.

\begin{definition}
\label{banach_manifold_definition} Let $\mathbb{M}$ be a set. An atlas of
class $C^p\, (p \ge 0)$ or analytic on $\mathbb{M} $ is a family of charts with some
indexing set $A,$ namely $\{(M_{\alpha},u_{\alpha}): \alpha \in A\},$ having
the following properties:

\begin{enumerate}
\item[AT1] $\{M_{\alpha}\}_{\alpha \in A}$ is a covering\footnote{%
The condition of an \emph{open} covering is not needed, see \cite{Lang}.} of 
$\mathbb{M},$ that is, $M_{\alpha} \subset \mathbb{M}$ for all $\alpha \in A$
and $\cup_{\alpha \in A}M_{\alpha}=\mathbb{M}.$

\item[AT2] For each $\alpha \in A, (M_{\alpha},u_{\alpha})$ stands for a
bijection $u_{\alpha}:M_{\alpha} \rightarrow U_{\alpha}$ of $M_{\alpha}$
onto an open set $U_{\alpha}$ of a Banach space $X_{\alpha},$ and for any $%
\alpha$ and $\beta$ the set $u_{\alpha}(M_{\alpha} \cap M_{\beta})$ is open
in $X_{\alpha}.$

\item[AT3] Finally, if we let $M_{\alpha} \cap M_{\beta} = M_{\alpha\beta}$
and $u_{\alpha}(M_{\alpha\beta} )= U_{\alpha\beta},$ the transition mapping $%
u_{\beta}\circ u_{\alpha}^{-1}:U_{\alpha\beta} \rightarrow U_{\beta \alpha}$
is a diffeomorphism of class $C^p$ $(p \ge 0)$ or analytic.
\end{enumerate}
\end{definition}

Since different atlases can give the same manifold, we say that two atlases
are \emph{compatible} if each chart of one atlas is compatible with the
charts of the other atlas in the sense of AT3. One verifies that the
relation of compatibility between atlases is an equivalence relation.

\begin{definition}
An equivalence class of atlases of class $C^{p}$ on $\mathbb{M}$ is said to
define a structure of a $C^{p}$-Banach manifold on $\mathbb{M},$ and hence
we say that $\mathbb{M}$ is a Banach manifold. In a similar way, if an
equivalence class of atlases is given by analytic maps, then we say that $%
\mathbb{M}$ is an analytic Banach manifold. If $X_{\alpha }$ is a Hilbert
space for all $\alpha \in A,$ then we say that $\mathbb{M}$ is a Hilbert
manifold.
\end{definition}

In condition AT2 we do not require that the Banach spaces are the same for
all indices $\alpha,$ or even that they are isomorphic. If $X_{\alpha}$ is
linearly isomorphic to some Banach space $X$ for all $\alpha,$ we have the
following definition.

\begin{definition}
Let $\mathbb{M}$ be a set and $X$ be a Banach space. We say that $\mathbb{M}$
is a $C^{p}$ (respectively, analytic) Banach manifold modelled on $X$ if there exists an atlas of
class $C^{p}$ (respectively, analytic) over $\mathbb{M}$ with $X_{\alpha }$ linearly isomorphic to $X$
for all $\alpha \in A.$
\end{definition}

\begin{example}
Every Banach space is a Banach manifold modelled on itself (for a Banach
space $Y$, simply take $(Y,id_{Y})$ as atlas, where $id_{Y}$ is the identity
map on $Y$). We would point out that the trivial linear space 
$\{\mathbf{0}\}$ is also a (trivial) Banach manifold modelled on
itself. In particular, the set of all bounded linear maps $\mathcal{L}%
(X,X)$ of a Banach space $X$ is also a Banach manifold modelled on itself.
\end{example}

If $X$ is a Banach space, then the set of all bounded linear automorphisms
of $X$ will be denoted by 
\begin{equation*}
\mathrm{GL}(X):=\left\{ A\in \mathcal{L}(X,X):A\text{ invertible}\right\} .
\end{equation*}

Before giving the next examples, we introduce
the following definition.

\begin{definition}
Let $X$ and $Y$ be two Banach manifolds. Let $F:X \rightarrow Y$ be a map.
We shall say that $F$ is a $\mathcal{C}^r$ (respectively, analytic) \emph{%
morphism} if given $x \in X$ there exists a chart $(U,\varphi)$ at $x$ and a
chart $(W,\psi)$ at $F(x)$ such that $F(U) \subset W,$ and the map 
\begin{equation*}
\psi \circ F \circ \varphi^{-1}:\varphi(U) \rightarrow \psi(W)
\end{equation*}
is a $\mathcal{C}^r$-Fr\'echet differentiable (respectively, analytic) map.
\end{definition}

\begin{example}\label{inverse_a}
If $X$ is a Banach space, then $\mathrm{GL}(X)$ is a Banach manifold
modelled on $\mathcal{L}(X,X),$ because it is an open set in $\mathcal{L}%
(X,X).$ Moreover, the map $A\mapsto A^{-1}$ is analytic 
(see 2.7 in \cite{Upmeier}).
\end{example}

\begin{example}\label{exp_a}
If $X$ is a Banach space, then the exponential map 
$\exp: \mathcal{L}(X,X) \rightarrow \mathrm{GL}(X)$ defined by
$\exp(A) = \sum_{n=0}^{\infty} \frac{A^n}{n!}$ is an analytic
map (see 2.8 in \cite{Upmeier}).
\end{example}

\begin{example}\label{multiplication_a}
If $X$ is a Banach space, then $\mathrm{GL}(X) \times \mathrm{GL}(X)$
is a Banach manifold and the multiplication map
$m: \mathrm{GL}(X) \times \mathrm{GL}(X) \rightarrow \mathrm{GL}(X)$ defined by
$m(A,B) = A\circ B$ is an analytic map (see Theorem 2.42(ii) in \cite{B}).
\end{example}

\begin{example}\label{local_lee_group}
Let $X$ be a Banach space and $U,W \in \mathbb{G}(X)$ be such that $X=U\oplus W.$
From Proposition \ref{niceinverse} we know that $\mathcal{L}_{(U,W)}(X,X)$ is a 
sub-algebra of the Banach Algebra $\mathcal{L}(X,X).$ Then from
Theorem 3.5 of \cite{B} we have that
$$
\mathrm{GL}(\mathcal{L}_{(U,W)}(X,X)):=\{\exp(L): L \in \mathcal{L}_{(U,W)}(X,X)\} \subset \mathrm{GL}(X)
$$
is a closed Lie subgroup with associated Lie algebra $\mathcal{L}_{(U,W)}(X,X)$ and it is
also an analytic Banach manifold modelled into itself. Since
$\exp(L)=id_X+L$ then $\exp(L)$ is a linear isomorphism between the
linear subspace $U$ 
and $\exp(L)(U) = \left\{ (id_X + L)(u):u\in U\right\}.$ We remark that for
all $x \in X$ we have 
$$
\exp(L)(x) = \exp(L)(u+w) = \exp(L)(u)+w, \quad (x=u+w,\, u\in U \text{ and } w\in W  ),
$$ 
because $L(w)=0,$ hence $\exp(L)|_{W} = id_{W}$  and $\exp(L)|_{U} = id_{U}+L.$
Moreover,
the maps
$\exp: \mathcal{L}_{(U,W)}(X,X) \rightarrow \mathrm{GL}(\mathcal{L}_{(U,W)}(X,X)),$
$m:\mathrm{GL}(\mathcal{L}_{(U,W)}(X,X)) \times 
\mathrm{GL}(\mathcal{L}_{(U,W)}(X,X)) \rightarrow \mathrm{GL}(\mathcal{L}_{(U,W)}(X,X))$ and the 
map $\exp(L) \mapsto \exp(-L)$ are analytic.
\end{example}

The next example is a Banach manifold not modelled on a particular Banach
space.

\begin{example}[Grassmann--Banach manifold]
\label{Grassmann_Example} Let $X$ be a Banach space. Then, following \cite%
{Douady66} (see also \cite{Upmeier} and \cite{MRA}), it is possible to
construct an atlas for $\mathbb{G}(X).$ To do this, let us denote one of the
complements of $U\in \mathbb{G}(X)$ by $W, $ i.e., $X=U\oplus W$. Then we
define the \emph{Banach Grassmannian of $U$ relative to $W$} by 
\begin{equation*}
\mathbb{G}(W,X):=\left\{ V\in \mathbb{G}(X):X=V\oplus W\right\}.
\end{equation*}%
By using Lemma~\ref{Char_Projections} it is possible to introduce a
bijection 
\begin{equation*}
\Psi _{U\oplus W}:\mathbb{G}(W,X)\longrightarrow \mathcal{L}(U,W)
\end{equation*}%
defined by 
\begin{equation*}
\Psi _{U\oplus W}(U^{\prime }) = P_{W \oplus U}|_{U^{\prime }} \circ
P_{U^{\prime }\oplus W}|_{U} = P_{W \oplus U}|_{U^{\prime }} \circ
(P_{U\oplus W}|_{U^{\prime }})^{-1} .
\end{equation*}
It can be shown that the inverse 
\begin{equation*}
\Psi _{U\oplus W}^{-1}:\mathcal{L}(U,W)\longrightarrow \mathbb{G}(W,X)
\end{equation*}%
is given by 
\begin{equation*}
\Psi _{U\oplus W}^{-1}(L)=G(L):=\left\{ (id_X + L)(u):u\in U\right\}.
\end{equation*}%
From Proposition \ref{niceinverse} we can identify
$\mathcal{L}(U,W)$ with $\mathcal{L}_{(U,W)}(X,X).$ Hence can write $$(id_X+L)=\exp(L),$$ which following  
Example \ref{local_lee_group} can be proved to be a linear isomorphism from $U$ to $G(L).$
Observe that $G(0)=U$ and $G(L)\oplus W=X$ for all $L\in \mathcal{L}(U,W).$
Finally, to prove that this manifold is analytic we need to describe the overlap
maps. To explain the behaviour of one overlap map, assume that $X=U\oplus W
= U^{\prime }\oplus W^{\prime }$ and the existence of $U^{\prime \prime }\in 
\mathbb{G}(W,X) \cap \mathbb{G}(W^{\prime },X).$ Let $L \in \mathcal{L}(U,W)$
and $L' \in \mathcal{L}(U',W')$ be such that 
\begin{equation*}
\Psi_{U \oplus W}^{-1}(L) = G(L) = U^{\prime \prime} = G(L')=\Psi_{U' \oplus W'}^{-1}(L').
\end{equation*}
Then it follows
that
\begin{equation*}
X = U\oplus W = U^{\prime }\oplus W^{\prime }= G(L) \oplus W = G(L) \oplus
W^{\prime }.
\end{equation*}
Finally, it can be shown that the map 
$(\Psi_{U^{\prime }\oplus W^{\prime }} \circ \Psi_{U \oplus W}^{-1}): 
\mathcal{L}(U,W) \rightarrow \mathcal{L}(U^{\prime },W^{\prime })$ 
given by 
\begin{align*}
(\Psi_{U^{\prime }\oplus W^{\prime }} \circ \Psi_{U \oplus W}^{-1})(L)  = 
\Psi_{U^{\prime }\oplus W^{\prime }} (\exp(L)(U))=L'
\end{align*}
is analytic. 
Then we have that the collection $\{\mathbb{G}%
(W,X),\Psi _{U\oplus W}\}_{U\in \mathbb{G}(X)}$ is an analytic atlas, and therefore, $\mathbb{G%
}(X)$ is an analytic Banach manifold. In particular, for each $U \in \mathbb{%
G}(X)$ the set $\mathbb{G}(W,X)$ is a 
Banach manifold modelled on $\mathcal{L}(U,W).$ Observe that if $%
U$ and $U^{\prime }$ are finite-dimensional subspaces of $X$ such that $\dim
U \neq \dim U^{\prime }$ and $X=U \oplus W = U^{\prime }\oplus W^{\prime },$
then $\mathcal{L}(U,W)$ is not linearly isomorphic to $\mathcal{L}(U^{\prime
},W^{\prime }).$
\end{example}

\begin{example}
\label{fixed_rank_grassmann} Let $X$ be a Banach space. From Proposition~\ref%
{finite_dim_subspaces}, every finite-dimensional subspace belongs to $%
\mathbb{G}(X).$ It allows to introduce $\mathbb{G}_{n}(X),$ the space of all 
$n$-dimensional subspaces of $X$ $(n\geq 0).$ It can be shown (see \cite{MRA}%
) that $\mathbb{G}_{n}(X)$ is a connected component of $\mathbb{G}(X),$ and
hence it is also a Banach manifold modelled on $\mathcal{L}(U,W),$ here $%
U\in \mathbb{G}_{n}(X)$ and $X=U\oplus W.$ Moreover, 
\begin{equation*}
\mathbb{G}_{\leq r}(X):=\bigcup_{n\leq r}\mathbb{G}_{n}(X)
\end{equation*}%
is also a Banach manifold for each fixed $r<\infty .$
\end{example}

The next example introduces the Banach-Grassmannian manifold for a normed
(non-Banach) space. To the authors knowledge there is no reference in the
literature about this (non-trivial) Banach manifold structure. We need the
following lemma.

\begin{lemma}
\label{normed_G} Assume that $(X,\Vert \cdot \Vert )$ is a normed space and
let $\overline{X}$ be the Banach space obtained as the completion of $X.$
Let $U \in \mathbb{G}_n(\overline{X})$ be such that $U \subset X$ and $%
\overline{X}=U\oplus W$ for some $W \in \mathbb{G}(\overline{X}).$ Then
every subspace $U^{\prime }\in \mathbb{G}(W,\overline{X})$ is a subspace of $%
X,$ that is, $U^{\prime }\subset X.$
\end{lemma}

\begin{proof}
First of all observe that $X=U\oplus (W\cap X)$ where $W\cap X$ is a linear
subspace dense in $W=W\cap \overline{X}.$ Assume that the lemma is not true.
Then there exists $U^{\prime }\in \mathbb{G}(W,\overline{X})$ such that $%
U^{\prime }\oplus W=\overline{X}$ and $U^{\prime }\cap X\neq U^{\prime }.$
Clearly $U^{\prime }\cap X\neq \{0\},$ otherwise $W\cap X=X$ 
and hence $U=\{0\},$ a contradiction. We have $X=(U^{\prime }\cap X)\oplus (W\cap X),$ which
implies $\overline{X}=(U^{\prime }\cap X)\oplus W,$ that is, $U^{\prime }\cap X \in \mathbb{G}(W,\overline{X})$, a contradiction with $\dim (U^{\prime }\cap X) < \dim U^{\prime } = n.$ Thus 
the lemma follows.
\end{proof}

\begin{example}
\label{normed_grassmann} Assume that $(X,\Vert \cdot \Vert )$ is a normed
space and let $\overline{X}$ be the Banach space obtained as the completion
of $X.$ We define the set $\mathbb{G}_n(X)$ as follows. We say that $U\in 
\mathbb{G}_{n}(X)$ if and only if $U\in \mathbb{G}_{n}(\overline{X})$ and $%
U\subset X.$ Then $\mathbb{G}_{n}(X)$ is also a Banach manifold. To see this
observe that, by Lemma~\ref{normed_G}, for each $U \in \mathbb{G}_{n}(X)$ 
such that $\overline{X}=U\oplus W$ for some $W \in \mathbb{G}(\overline{X}),
$ we have $\mathbb{G}(W,\overline{X}) \subset \mathbb{G}_n(X).$ Then the
collection $\{\Psi _{U\oplus W},\mathbb{G}(W,\overline{X})\}_{U\in \mathbb{G}%
_{n}(X)}$ is an analytic atlas on $\mathbb{G}_{n}(X),$ and therefore, $%
\mathbb{G}_{n}(X)$ is an analytic Banach manifold modelled on $\mathcal{L}%
(U,W),$ here $U\in \mathbb{G}_{n}(X)$ and $\overline{X}=U\oplus W.$
Moreover, as in Example~\ref{fixed_rank_grassmann}, we can define a Banach
manifold $\mathbb{G}_{\leq r}(X)$ for each fixed $r<\infty .$
\end{example}

Let $\mathbb{M}$ be a Banach manifold of class $\mathcal{C}^{p}$ $(p\geq 1)$
or analytic. Let $m$ be a point of $\mathbb{M}$. We consider triples $(U,\varphi ,v)$
where $(U,\varphi )$ is a chart at $m$ and $v$ is an element of the vector
space in which $\varphi (U)$ lies. We say that two of such triples $%
(U,\varphi ,v)$ and $(V,\psi ,w)$ are \emph{equivalent} if the derivative of 
$\psi \circ \varphi ^{-1}$ at $\varphi (m)$ maps $v$ on $w.$ Thanks to the chain
rule it is an equivalence relation. An equivalence class of such triples is
called a \emph{tangent vector of $\mathbb{M}$ at $m$.}

\begin{definition}
\label{Def T}The set of such tangent vectors is called the \emph{tangent space
of $\mathbb{M}$ at $m$} and it is denoted by $\mathbb{T}_{m}(\mathbb{M}).$
\end{definition}

Each chart $(U,\varphi )$ determines a bijection of $\mathbb{T}_{m}(\mathbb{M%
})$ on a Banach space, namely the equivalence class of $(U,\varphi ,v)$
corresponds to the vector $v.$ By means of such a bijection it is possible
to equip $\mathbb{T}_{m}(\mathbb{M})$ with the structure of a topological
vector space given by the chart, and it is immediate that this structure is
independent of the selected chart.

\begin{example}
If $X$ is a Banach space, then $\mathbb{T}_{x}(X)=X$ for all $x\in X.$
\end{example}

\begin{example}
Let $X$ be a Banach space and take $A\in \mathrm{GL}(X).$ Then $\mathbb{T}%
_{A}(\mathrm{GL}(X))=\mathcal{L}(X,X).$
\end{example}

\begin{example}
For $U\in \mathbb{G}(X)$ such that $X=U\oplus W$ for some $W \in \mathbb{G}%
(X)$, we have $\mathbb{T}_{U}(\mathbb{G}(X))=\mathcal{L}(U,W).$
\end{example}

\begin{example}
\label{unit_sphere} For a Hilbert space $X$ with
associated inner product $\langle \cdot ,\cdot \rangle $ and norm $\Vert
\cdot \Vert ,$ its unit sphere denoted by 
\begin{equation*}
\mathbb{S}_{X}:=\{x\in X:\Vert x\Vert =1\}
\end{equation*}%
is a Hilbert manifold of co-dimension one. Moreover, for each $x\in \mathbb{S}%
_{X},$ its tangent space is 
\begin{equation*}
\mathbb{T}_{x}(\mathbb{S}_{X})=\mathrm{span}\,\{x\}^{\bot }=\{x^{\prime }\in
X:\langle x,x^{\prime }\rangle =0\}.
\end{equation*}
\end{example}

\section{The manifold of tensors in Tucker format with fixed rank}
\label{sec:banach_manifold_tucker_fixed_rank}

The MCTDH method is based on the construction of approximations of the wave function which,
at every time $t,$ lie in the algebraic tensor space
$\left._a \bigotimes_{\alpha = 1}^d V_{\alpha} \right.$ where 
$V_{\alpha} = L^2(\mathbb{R}^3)$ for $\alpha=1,2,\ldots,d$
(see \cite{Lubish}). Clearly, this set is a linear space. However it is not clear whether or not it is a (Hilbert/Banach) manifold, because it is a dense subspace of the Hilbert tensor space 
$L^2(\mathbb{R}^{3d}).$ In this section, we will show that every algebraic tensor product of
normed spaces can be seen as a Banach-Grassmann-like manifold.

\subsection{Tensor Spaces and the tensor product map}
\label{Tensor_TBF}

All along this paper we consider a finite index set $D:=\{1,2,\ldots,d\}$  of `spatial directions', with 
$d \ge 2$. Concerning the definition of the algebraic tensor space $_{a}\bigotimes
_{\alpha \in D} V_{\alpha}$ generated from vector spaces $V_{\alpha}$ $\left(\alpha \in D \right)$, we refer to Greub \cite{Greub}. As underlying field we choose $%
\mathbb{R},$ but the results hold also for $\mathbb{C}$. The suffix `$a$' in 
$_{a}\bigotimes_{\alpha \in D} V_{\alpha}$ refers to the `algebraic' nature. By
definition, all elements of 
\begin{equation*}
\mathbf{V}_D:=\left. _{a}\bigotimes_{\alpha \in D} V_{\alpha}\right.
\end{equation*}
are \emph{finite} linear combinations of elementary tensors $\mathbf{v}%
=\bigotimes_{\alpha \in D} v_{\alpha}$ $\left( v_{\alpha}\in V_{\alpha}\right) .$  
In the
sequel, the index sets $D\backslash \{\alpha\}$ will appear. Here, we use the
abbreviations%
\begin{equation*}
\mathbf{V}_{[\alpha]}:=\left. _{a}\bigotimes_{\beta \neq \alpha}V_{\beta}\right. \text{,\qquad
where }\bigotimes_{\beta \neq \alpha}\text{ means}\bigotimes_{\beta\in D\backslash \{\alpha\}}.
\label{(V[j]}
\end{equation*}%
Similarly, elementary tensors $\bigotimes_{\beta \neq \alpha}v_{\beta}$ are denoted by $%
\mathbf{v}_{[\alpha]}$. 
We notice that there exists a linear isomorphism 
$\Phi_{\alpha}:  \mathbf{V}_D \longrightarrow V_{\alpha} \left._a\otimes\right. \mathbf{V}_{[\alpha]}$ for each $\alpha \in D,$
and in order to simplify notation we will identify along the text a tensor
$\mathbf{v} \in \mathbf{V}_D$ with $\Phi_{\alpha}(\mathbf{v}) \in V_{\alpha} \left._a\otimes\right. \mathbf{V}_{[\alpha]}.$
This allows us to write $\mathbf{v} \in  \mathbf{V}_D$ as well as $\mathbf{v} \in V_{\alpha} \left._a\otimes\right. \mathbf{V}_{[\alpha]}$
for $\alpha \in D.$ Moreover, by the universal property of the tensor product,
there exists a unique multilinear map, also denoted by $\bigotimes$
\begin{equation*}
\bigotimes :\mathop{\mathchoice{\raise-0.22em\hbox{\huge $\times$}}
{\raise-0.05em\hbox{\Large $\times$}}{\hbox{\large $\times$}}{\times}}%
_{\alpha \in D} V_{\alpha} 
\longrightarrow%
\left. _{a}\bigotimes_{\alpha \in D} V_{\alpha}\right.,
\end{equation*}
defined by $\bigotimes\left( (v_{1},\ldots,v_{d})\right)
=\bigotimes_{\alpha \in D}v_{\alpha}$ and such that for each multilinear map
$M: \mathop{\mathchoice{\raise-0.22em\hbox{\huge $\times$}}
{\raise-0.05em\hbox{\Large $\times$}}{\hbox{\large $\times$}}{\times}}%
_{\alpha \in D} V_{\alpha} 
\longrightarrow Z,$ where $Z$ is a given vector space, 
there exists a unique map $\widehat{M}:\mathbf{V}_D \rightarrow Z$ such that
$ M = \widehat{M} \circ \bigotimes.$ 
The following notations, definitions and results will be useful.

\bigskip

Let $(V_{\alpha},\|\cdot\|_{\alpha})$ be normed spaces for $\alpha \in D$ 
and assume that $\|\cdot\|$ is a norm on the tensor space $\mathbf{V}_D = \left.
_{a}\bigotimes_{\alpha \in D}V_{\alpha}\right..$ Then consider the tensor product map
\begin{equation}
\bigotimes :\left(\mathop{\mathchoice{\raise-0.22em\hbox{\huge $\times$}}
{\raise-0.05em\hbox{\Large $\times$}}{\hbox{\large $\times$}}{\times}}%
_{\alpha \in D} V_{\alpha},\left\Vert \cdot\right\Vert_{\times} \right) 
\longrightarrow%
\bigg(
\left. _{a}\bigotimes_{\alpha \in D}V_{\alpha}\right. ,\left\Vert \cdot\right\Vert
\bigg)
,  \label{bigotimes}
\end{equation}
where the product space $\mathop{\mathchoice{\raise-0.22em\hbox{\huge $\times$}}
{\raise-0.05em\hbox{\Large $\times$}}{\hbox{\large $\times$}}{\times}}%
_{\alpha \in D} V_{\alpha}$ is equipped with the
product topology induced by the maximum norm $\|(v_1,\hdots,v_d)\|_{\times} =
\max_{\alpha \in D} \|v_{\alpha}\|_{\alpha}.$ Next, we discuss 
the conditions for having the  Fr\'echet differentiability  of 
the tensor product map (\ref{bigotimes}). 
The next result is a consequence of Proposition \ref{multilinear}.

\begin{proposition}
\label{Diff_otimes} Let $(V_{\alpha},\|\cdot\|_{\alpha})$ be normed spaces for $\alpha \in D.$ Assume that $\|\cdot\|$ is a norm on the tensor space $\mathbf{V}_D = \left.
_{a}\bigotimes_{\alpha \in D}V_{\alpha}\right.$ such that the tensor
product map  (\ref{bigotimes})
is continuous. Then it is also 
$\mathcal{C}^{\infty}$-Fr\'{e}chet differentiable and its differential is given by 
\begin{equation*}
D\left( \bigotimes (v_{1},\ldots ,v_{d})\right) (w_{1},\ldots
,w_{d})=\sum_{\alpha \in D}v_{1}\otimes \ldots \otimes v_{\alpha-1}\otimes
w_{\alpha}\otimes v_{\alpha+1}\otimes \cdots v_{d}.
\end{equation*}
\end{proposition}

Now, we recall the definition of some topological tensor spaces and we
will give some examples.

\begin{definition}
We say that $\mathbf{V}_{D_{\left\Vert
\cdot\right\Vert}}$ is a \emph{Banach tensor space} if there exists an
algebraic tensor space $\mathbf{V}_D$ and a norm $\left\Vert \cdot\right\Vert $
on $\mathbf{V}_D$ such that $\mathbf{V}_{D_{\left\Vert
\cdot\right\Vert}}$ is the
completion of $\mathbf{V}_D$ with respect to the norm $\left\Vert
\cdot\right\Vert $, i.e.
\begin{equation*}
\mathbf{V}_{D_{\left\Vert
\cdot\right\Vert}}:=\left. _{\left\Vert \cdot
\right\Vert }\bigotimes_{\alpha \in D}V_{\alpha}\right. =\overline{\left.
_{a}\bigotimes\nolimits_{\alpha \in D}V_{\alpha}\right. }^{\left\Vert \cdot\right\Vert
}.
\end{equation*}
If $\mathbf{V}_{D_{\left\Vert
\cdot\right\Vert}}$ is a Hilbert space, we say
that $\mathbf{V}_{D_{\left\Vert
\cdot\right\Vert}}$ is a \emph{Hilbert tensor
space}.
\end{definition}

Next, we give some examples of Banach and Hilbert tensor spaces.

\begin{example}
\label{Bsp HNp}For $I_{\alpha}\subset \mathbb{R}$ $\left( \alpha \in D\right) $
and $1\leq p<\infty ,$ the Sobolev space $H^{N,p}(I_{\alpha})$ consists of all
univariate functions $f$ in $L^{p}(I_{\alpha})$ with bounded norm\footnote{%
It suffices to have in \eqref{(SobolevNormp a} the terms $n=0$ and $n=N.$
The derivatives are to be understood as weak derivatives.}%
\begin{subequations}
\label{(SobolevNormp}%
\begin{equation*}
\left\Vert f\right\Vert _{N,p;I_{\alpha}}:=\bigg(\sum_{n=0}^{N}\int_{I_{\alpha}}\left%
\vert \partial ^{n}f\right\vert ^{p}\mathrm{d}x\bigg)^{1/p},
\label{(SobolevNormp a}
\end{equation*}%
whereas the space $H^{N,p}(\mathbf{I})$ of $d$-variate functions on $\mathbf{%
I}=I_{1}\times I_{2}\times \ldots \times I_{d}\subset \mathbb{R}^{d}$ consists of all functions $f$ in $L^{p}(\mathbf{I})$
with bounded norm 
\begin{equation*}
\left\Vert f\right\Vert _{N,p}:=\Big(\sum_{0\leq \left\vert \mathbf{n}%
\right\vert \leq N}\int_{\mathbf{I}}\left\vert \partial ^{\mathbf{n}%
}f\right\vert ^{p}\mathrm{d}\mathbf{x}\Big)^{1/p}  \label{(SobolevNormp b}
\end{equation*}%
\end{subequations}%
with $\mathbf{n}\in \mathbb{N}_{0}^{d}$ being a multi-index of length $%
\left\vert \mathbf{n}\right\vert :=\sum_{\alpha \in D}n_{\alpha}$. For $p>1$ it is
well known that $H^{N,p}(I_{\alpha})$ and $H^{N,p}(\mathbf{I})$ are reflexive and
separable Banach spaces. Moreover, for $p=2,$ the Sobolev spaces $%
H^{N}(I_{\alpha}):=H^{N,2}(I_{\alpha})$ and $H^{N}(\mathbf{I}):=H^{N,2}(\mathbf{I})$
are Hilbert spaces. As a first example,%
\begin{equation*}
H^{N,p}(\mathbf{I})=\left. _{\left\Vert \cdot \right\Vert
_{N,p}}\bigotimes_{\alpha \in D}H^{N,p}(I_{\alpha})\right.
\end{equation*}%
is a Banach tensor space. Examples of Hilbert tensor spaces are%
\begin{equation*}
L^{2}(\mathbf{I})=\left. _{\left\Vert \cdot \right\Vert
_{0,2}}\bigotimes_{\alpha \in D}L^{2}(I_{\alpha})\right. \text{\quad and\quad }H^{N}(%
\mathbf{I})=\left. _{\left\Vert \cdot \right\Vert
_{N,2}}\bigotimes_{\alpha \in D}H^{N}(I_{\alpha})\right. \text{ for }N\in \mathbb{N}.
\end{equation*}
\end{example}

The next result is a consequence of Corollary \ref{analytic}.

\begin{proposition}
\label{Analytic_otimes} Let $(V_{\alpha},\|\cdot\|_{\alpha})$ be complex Banach spaces for $\alpha \in D.$ Assume that $\|\cdot\|$ is a norm on the complex tensor space $\mathbf{V}_D = \left.
_{a}\bigotimes_{\alpha \in D}V_{\alpha}\right.$ such that the tensor
product map (\ref{bigotimes})
is continuous. Let $\mathfrak{i}:\mathbf{V}_D \rightarrow \left. _{\|\cdot\|}\bigotimes_{\alpha \in D}V_{\alpha}\right.$ be the
standard inclusion map, i.e. $\mathfrak{i}(\mathbf{v})=\mathbf{v}.$
Then 
\begin{equation*}
\left(\mathfrak{i} \circ \bigotimes\right) :\mathop{\mathchoice{\raise-0.22em\hbox{\huge $\times$}}
{\raise-0.05em\hbox{\Large $\times$}}{\hbox{\large $\times$}}{\times}}%
_{\alpha \in D}\left( V_{\alpha},\left\Vert \cdot\right\Vert _{\alpha}\right)
\longrightarrow%
\bigg(
\left. _{\|\cdot\|}\bigotimes_{\alpha \in D}V_{\alpha}\right. ,\left\Vert \cdot\right\Vert
\bigg), 
\end{equation*}
is an analytic map between complex Banach spaces.
\end{proposition}

For vector spaces $V_{\alpha}$ and $W_{\alpha}$ over $\mathbb{R},$ let linear mappings 
$A_{\alpha}:V_{\alpha}\rightarrow W_{\alpha}$ $\left( \alpha \in D\right) $ be given. Then
the definition of the elementary tensor%
\begin{equation*}
\mathbf{A}=\bigotimes_{\alpha \in D}A_{\alpha}:\;\mathbf{V}_D=\left. _{a}\bigotimes
_{\alpha \in D}V_{\alpha}\right. \longrightarrow\mathbf{W}_D=\left. _{a}\bigotimes
_{\alpha \in D}W_{\alpha}\right.
\end{equation*}
is given by%
\begin{equation}
\mathbf{A}\left( \bigotimes_{\alpha \in D}v_{\alpha}\right)
:=\bigotimes_{\alpha \in D}\left( A_{\alpha}v_{\alpha}\right) .
\label{(A als Tensorprodukt der Aj}
\end{equation}
Note that \eqref{(A als Tensorprodukt der Aj} uniquely defines the linear
mapping $\mathbf{A}:\mathbf{V}_D\rightarrow\mathbf{W}_D.$ We recall that $L(V,W)$
is the space of linear maps from $V$ into $W,$ while $V^{\prime}=L(V,\mathbb{%
R})$ is the algebraic dual of $V$. For metric spaces, $\mathcal{L}(V,W)$
denotes the continuous linear maps, while $V^{\ast}=\mathcal{L}(V,\mathbb{R}%
) $ is the topological dual of $V$. 

\begin{proposition}\label{linear_bigotimes}
Let $(V_{\alpha},\|\cdot\|_{\alpha})$ be normed spaces for $\alpha \in D$ and assume that $\|\cdot\|$ is a norm on the tensor space $\mathbf{V}_D = \left.
_{a}\bigotimes_{\alpha \in D}V_{\alpha}\right.$ such that the tensor product map (\ref{bigotimes})
is continuous. Let $U_{\alpha}$ be a finite-dimensional subspace of $V_{\alpha}$ for $\alpha \in D.$
Then 
\begin{align}\label{linear_big1}
\left. _{a}\bigotimes_{\alpha \in D}\mathcal{L}(U_{\alpha},V_{\alpha})\right. =\mathcal{L}\left( \left.
_{a}\bigotimes_{\alpha \in D}U_{\alpha}\right.
, \mathbf{V}_D \right)
\end{align}
and the tensor product map
\begin{align}\label{linear_big2}
\bigotimes:\mathop{\mathchoice{\raise-0.22em\hbox{\huge
 $\times$}} {\raise-0.05em\hbox{\Large $\times$}}{\hbox{\large
 $\times$}}{\times}}_{\alpha \in D} \mathcal{L}(U_{\alpha},V_{\alpha}) \rightarrow \mathcal{L}\left( \left.
_{a}\bigotimes_{\alpha \in D}U_{\alpha}\right.
, \mathbf{V}_D \right), 
 \quad (A_{\alpha})_{\alpha \in D} \mapsto \mathbf{A}:= \bigotimes_{\alpha \in D} A_{\alpha},
\end{align}
is continuous and hence $\mathcal{C} ^{\infty}$-Fr\'echet differentiable.
\end{proposition}

\begin{proof}
Recall that $L(U,X) = \mathcal{L}(U,X)$ holds  
for every finite-dimensional subspace $U$ of a normed space $X.$  
Then (\ref{linear_big1}) follows from Proposition 3.49 of \cite{Hackbusch}. To prove
the second statement we need to show that the tensor product map (\ref{linear_big2}) 
is bounded, that is,
\begin{align}\label{bounded}
\|\bigotimes\| = \sup \left\{
\|\bigotimes_{\alpha \in D} A_{\alpha}\|_{\mathbf{V}_D \leftarrow  \left.
_{a}\bigotimes_{\alpha \in D}U_{\alpha}\right. }: \|A_\alpha \|_{V_\alpha \leftarrow U_\alpha} \le 1 \text{ for } \, 1 \le \alpha \le d
\right\} < \infty.
\end{align}
For $\mathbf{A} = \bigotimes_{\alpha \in D} A_{\alpha}$, 
$$
\|\mathbf{A}(\otimes_{\alpha \in D} u_\alpha)\| = \|\otimes_{\alpha \in D} A_{\alpha}(u_{\alpha})\| \le C \prod_{\alpha \in D} \|A_{\alpha}(u_{\alpha})\|_{\alpha} \le
C \prod_{\alpha \in D} \|A_{\alpha}\|_{V_{\alpha} \leftarrow U_{\alpha}} \|u_{\alpha}\|_{\alpha}
$$
holds by the continuity of the tensor product map (\ref{bigotimes}). Therefore,
$$
\|\mathbf{A}\|_{\mathbf{V}_D \leftarrow  \left.
_{a}\bigotimes_{\alpha \in D}U_{\alpha}\right. } = \sup \left\{
	\|\mathbf{A}(\mathbf{u})\| : \mathbf{u}\in \bigotimes_{\alpha \in D}U_{\alpha}, \|\mathbf{u}\| \le 1 
\right\} \le C' \prod_{\alpha \in D} \|A_{\alpha}\|_{V_{\alpha} \leftarrow U_{\alpha}},
$$
for some constant $C'$ depending on the dimension of the spaces $U_{\alpha}$, $\alpha \in D,$ and (\ref{bounded}) follows. From Proposition \ref{multilinear} the second statement holds.
\end{proof}

\subsection{The set of tensors in Tucker format with fixed rank}

Before introducing the manifold of tensors in Tucker format with fixed rank 
in a Banach space framework, we need to define the minimal subspace 
of a tensor in an algebraic tensor space. The following statement summarises 
the results given in Section 2.2 in \cite{FALHACK}.

\begin{proposition}\label{minimal}
Given a finite index set $D=\{1,2,\ldots,d\}$, let $V_{\alpha}$ be a vector space
for each $\alpha \in D$ and let $\mathbf{v}\in \left._a \bigotimes_{\alpha\in D} V_{\alpha} \right..$ 
Then for each $\alpha \in D$ there exists a unique subspace $U_{\alpha}^{\min}(\mathbf{v})$ with  $\dim 
U_{\alpha}^{\min}(\mathbf{v}) = r_{\alpha}$ for some $r_{\alpha} < \infty,$ and  such that the following statements hold.
\begin{enumerate}
\item[(a)] If $\mathbf{v}\in \left._a \bigotimes_{\alpha \in D} U_{\alpha} \right.$ then
$U_{\alpha}^{\min}(\mathbf{v})\subset U_{\alpha}$ ($\alpha \in D$), while $\mathbf{v}\in \left._a \bigotimes_{\alpha\in D} U_{\alpha}^{\min}(\mathbf{v}) \right..$
\item[(b)]  For each $\alpha \in D$ there exists a unique subspace $U_{D\setminus \{\alpha\}}^{\min}(\mathbf{v}) \subset
\mathbf{V}_{[\alpha]}$ such that $\mathbf{v} \in U_{\alpha}^{\min}(\mathbf{v}) \otimes_{a}
U_{D\setminus \{\alpha\}}^{\min}(\mathbf{v})$ and $\dim U_{D\setminus\{\alpha\}}^{\min}(\mathbf{v})=r_{\alpha}.$
\end{enumerate}
\end{proposition}

\bigskip

For a tensor $\mathbf{v}\in \mathbf{V}_D = \left._a \bigotimes_{\alpha \in D} V_{\alpha} \right.$ the linear subspaces $U_{\alpha}^{\min}(\mathbf{v})$ ($\alpha \in D$) 
are called \emph{minimal subspaces} and
$r_{\alpha}= \dim U_{\alpha}^{\min}(\mathbf{v})$ is called the $\alpha$-rank of $\mathbf{v}.$

\bigskip

Let $\mathbb{Z}_+$ be the set of non-negative integers.
We will say that $\mathfrak{r}=(r_1,\ldots,r_d) \in \mathbb{Z}_+^d$ is an \emph{admissible rank} 
for $\mathbf{V}_D:= \left._a \bigotimes_{\alpha \in D} V_{\alpha} \right.$
if and only if there exists $\mathbf{v} \in \mathbf{V}_D$ 
such that $r_{\alpha}=\dim U_{\alpha}^{\min}(\mathbf{v})$
for $\alpha \in D.$ We will denote the set of all admissible ranks of a tensor space $\mathbf{V}_D$ by 
$\mathcal{AD}(\mathbf{V}_D),$ and hence
$$
\mathcal{AD}(\mathbf{V}_D) = \left\{(\dim U_{\alpha}^{\min}(\mathbf{v}))_{\alpha \in D}\in \mathbb{Z}_+^d: \mathbf{v} \in \mathbf{V}_D\right\}.
$$
It is not difficult to see that $\mathbf{0}=(0,\ldots,0) \in \mathcal{AD}(\mathbf{V}_D)$ and 
$\mathbf{1}=(1,\ldots,1) \in \mathcal{AD}(\mathbf{V}_D)$
if and only if $\dim V_{\alpha} \ge 1$ for all $\alpha \in D.$

\bigskip

Now, we define in an algebraic tensor space 
$\mathbf{V}_D = \left._a \bigotimes_{\alpha \in D} V_{\alpha} \right.$ \emph{the set of 
tensors in Tucker format with fixed rank 
$\mathfrak{r}=(r_1,\ldots,r_d) \in \mathcal{AD}(\mathbf{V}_D)$} 
by 
$$
\mathfrak{M}_{\mathfrak{r}}(\mathbf{V}_D):=
\left\{\mathbf{v} \in \mathbf{V}_D: \dim U_{\alpha}^{\min}(\mathbf{v}) = r_{\alpha}, \, \alpha \in D \right\}.
$$
Then
$$
\mathbf{V}_D = \bigcup_{\mathfrak{r} \in \mathcal{AD}(\mathbf{V}_D)}\mathfrak{M}_{\mathfrak{r}}(\mathbf{V}_D).
$$

Before introducing the representation of a tensor with a fixed rank $\mathfrak{r}$ 
we need to define the set of coefficients of that tensors. To this end, we recall
the definition of the `matricisation' (or `unfolding') of a tensor in a
finite-dimensional setting.

\begin{definition}
\label{Def Malpha}For a finite index set $D=\{1,2,\ldots,d\},$ $d \ge 2,$ and each $\mu \in D$ the
map $\mathcal{M}_{\mu}$ is defined as the isomorphism%
\begin{equation*}
\begin{tabular}{llll}
$\mathcal{M}_{\mu}:$ & $\mathbb{R}^{%
\mathop{\mathchoice{\raise-0.22em\hbox{\huge $\times$}} {\raise-0.05em\hbox{\Large $\times$}}{\hbox{\large
$\times$}}{\times}}_{\beta \in D}r_{\beta}} $ & $\rightarrow $ & $\mathbb{R}%
^{r_{\mu} \times \left( \prod_{\delta \in
D \setminus \{\mu\}}r_{\delta}\right) },$ \\ 
& $C_{(i_{\beta})_{\beta \in D}}$ & $\mapsto $ & $C_{i_{\mu},(i_{\delta})_{\delta \in D \setminus \{\mu\}}}$%
\end{tabular}%
\ \ \ \ \ 
\end{equation*}
\end{definition}

It allows us to introduce the following definition.

\begin{definition}
For a finite index set 
$D=	\{1,2,\ldots,d\},$ $d \ge 2,$ let $C^{(D)}\in \mathbb{R}^{%
\mathop{\mathchoice{\raise-0.22em\hbox{\huge $\times$}}
{\raise-0.05em\hbox{\Large $\times$}}{\hbox{\large
$\times$}}{\times}}_{\mu \in D}r_{\mu}}.$  
We say that $C^{(D)}\in 
\mathbb{R}_*^{%
\mathop{\mathchoice{\raise-0.22em\hbox{\huge $\times$}}
{\raise-0.05em\hbox{\Large $\times$}}{\hbox{\large
$\times$}}{\times}}_{\mu \in D}r_{\mu}}$ if and only if $\mathrm{rank}\, \mathcal{M}_{\mu}(C^{(D)}) = r_{\mu},$ where $\mathcal{M}_{\mu}(C^{(D)}) \in \mathbb{R}^{r_{\mu} \times \left(
\prod_{\beta \in D \setminus \{\mu\}}r_{\beta}\right) },$ for each $%
\mu \in D.$ 
\end{definition}

\begin{remark}\label{open_set}
We have that 
$C^{(D)} \in \mathbb{R}_*^{%
\mathop{\mathchoice{\raise-0.22em\hbox{\huge $\times$}}
{\raise-0.05em\hbox{\Large $\times$}}{\hbox{\large
$\times$}}{\times}}_{\mu \in D}r_{\mu}}$ if and only if $%
\mathcal{M}_{\mu}(C^{(D)})\mathcal{M}_{\mu}(C^{(D)})^T \in \mathrm{GL}(\mathbb{R}^{r_{\mu}})$ for $\mu \in D.$  Since the determinant is a continuous function, $\mathbb{R}_*^{%
\mathop{\mathchoice{\raise-0.22em\hbox{\huge $\times$}}
{\raise-0.05em\hbox{\Large $\times$}}{\hbox{\large
$\times$}}{\times}}_{\mu \in D}r_{\mu }}$ is an open set in $\mathbb{R}%
^{%
\mathop{\mathchoice{\raise-0.22em\hbox{\huge $\times$}}
{\raise-0.05em\hbox{\Large $\times$}}{\hbox{\large
$\times$}}{\times}}_{\mu \in D}r_{\mu }}$ and hence a finite-dimensional manifold. 
We point out that if $r_{\mu}=1$ for all $\mu \in D$ then
$\mathbb{R}^{%
\mathop{\mathchoice{\raise-0.22em\hbox{\huge $\times$}}
{\raise-0.05em\hbox{\Large $\times$}}{\hbox{\large
$\times$}}{\times}}_{\mu \in D}r_{\mu}}_* = \mathbb{R}_* = \mathbb{R} \setminus \{0\},$
which coincides with the Lie group $\mathrm{GL}(\mathbb{R}).$
\end{remark}

In the next lemma we give a characterisation of the representation 
of tensors in $\mathfrak{M}_{\mathfrak{r}}(\mathbf{V}_D).$

\begin{lemma}\label{characterization_FT}
Let $\mathbf{V}_D = \left._a \bigotimes_{\alpha \in D} V_{\alpha} \right.$ be an algebraic tensor space. Then
the following statements are equivalent.
\begin{enumerate}
\item[(a)] $\mathbf{v} \in \mathfrak{M}_{\mathfrak{r}}(\mathbf{V}_D).$
\item[(b)] For each $\alpha \in D$ there exists a set
$\mathcal{B}_{\alpha}=\{u_{i_{\alpha}}^{(\alpha)}: 1 \le i_{\alpha} \le r_{\alpha}\}$ of
linearly independent vectors and a unique $C^{(D)} \in \mathbb{R}_*^{%
\mathop{\mathchoice{\raise-0.22em\hbox{\huge $\times$}}
{\raise-0.05em\hbox{\Large $\times$}}{\hbox{\large
$\times$}}{\times}}_{\alpha \in D}r_{\alpha}},$ once $\mathcal{B}_{\alpha}$ is fixed $(\alpha \in D)$, such that
\begin{align}\label{Tucker_representation}
\mathbf{v} = \sum_{\substack{1 \le i_{\alpha} \le r_{\alpha} \\ \alpha \in D}} C^{(D)}_{(i_{\alpha})_{\alpha \in D}} \bigotimes_{\alpha \in D} u^{(\alpha)}_{i_{\alpha}}.
\end{align} 
\item[(c)] For each $\alpha \in D$ there exist 
 linearly independent vectors $\{u_{i_{\alpha}}^{(\alpha)}: 1 \le i_{\alpha} \le r_{\alpha}\}$
in $V_{\alpha}$ and  
 
linearly independent vectors $\{\mathbf{U}_{i_{\alpha }}^{(\alpha )}: 1 \le  i_{\alpha} \le r_{\alpha}\}$ in $\mathbf{V}_{[\alpha]} = 
\left._a \bigotimes_{\beta \in D \setminus \{\alpha\}}V_{\beta} \right.$
such that
\begin{equation}\label{F1}
\mathbf{v} = \sum_{1 \le i_{\alpha} \le r_{\alpha}} 
u_{i_{\alpha}}^{(\alpha)} \otimes 
\mathbf{U}_{i_{\alpha}}^{(\alpha)}.
\end{equation}
Furthermore, if (\ref{Tucker_representation}) holds, then
\begin{equation}\label{F2}
\mathbf{U}_{i_{\alpha }}^{(\alpha )}=\sum_{\substack{ 1\leq i_{\beta }\leq
r_{\beta }  \\ \beta \in D \setminus \{\alpha\}}}C_{i_{\alpha},(i_{\beta })_{\beta
\in D\setminus \{\alpha\}}}^{(D)}\bigotimes_{\beta \in D}u_{i_{\beta }}^{(\beta
)}
\end{equation}
for $1 \le i_{\alpha} \le r_{\alpha}$ and $\alpha \in D.$
\end{enumerate}
\end{lemma}

\begin{proof}
First, we prove that (a) and (c) are equivalent. If (a) holds, then from Proposition~\ref{minimal}(b) we know
that 
$$
\mathbf{v} \in U_{\alpha}^{\min}(\mathbf{v}) \otimes_a U_{D \setminus \{\alpha\}}^{\min}(\mathbf{v})
$$
where $\dim  U_{\alpha}^{\min}(\mathbf{v}) = \dim U_{D \setminus \{\alpha\}}^{\min}(\mathbf{v}) = r_{\alpha}$ for each $\alpha \in D.$ Then there exists 
 linearly independent vectors $\{u_{i_{\alpha}}^{(\alpha)}: 1 \le i_{\alpha} \le r_{\alpha}\}$ 
in $V_{\alpha}$ and  
linearly independent vectors $\{\mathbf{U}_{i_{\alpha }}^{(\alpha )}: 1 \le  i_{\alpha} \le r_{\alpha}\}$  in $\mathbf{V}_{[\alpha]} = 
\left._a \bigotimes_{\beta \in D \setminus \{\alpha\}}V_{\beta} \right.$
such that (\ref{F1}) holds and hence (c) is true. Conversely, if (c) holds then
clearly $\dim  U_{\alpha}^{\min}(\mathbf{v}) = r_{\alpha}$ for each $\alpha \in D,$
and hence (a) is also true. 

Now, we prove that (b) and (c) are equivalent. Clearly (b) implies (c). To prove
that (c) implies (b) assume that (c) holds. By the definition of minimal subspace
we have that $$U_{\alpha}^{\min}(\mathbf{v}) = \mathrm{span}\, \{u_{i_{\alpha}}^{(\alpha)}: 1 \le i_{\alpha} \le r_{\alpha}\}$$ for each $\alpha \in D.$ Since $\mathbf{v} \in \left._a \bigotimes_{\alpha \in D} 
U_{\alpha}^{\min}(\mathbf{v})\right.$ there exists $C^{(D)} \in \mathbb{R}^{%
\mathop{\mathchoice{\raise-0.22em\hbox{\huge $\times$}}
{\raise-0.05em\hbox{\Large $\times$}}{\hbox{\large
$\times$}}{\times}}_{\alpha \in D}r_{\alpha}}$ such that (\ref{Tucker_representation}) holds.
To conclude the proof  we only need to show that  $C^{(D)} \in\mathbb{R}_*^{%
\mathop{\mathchoice{\raise-0.22em\hbox{\huge $\times$}}
{\raise-0.05em\hbox{\Large $\times$}}{\hbox{\large
$\times$}}{\times}}_{\alpha \in D}r_{\alpha}}.$ To this end observe that (\ref{F2})
must hold for $1 \le i_{\alpha} \le r_{\alpha}$ and each $\alpha \in D,$
and hence $\mathrm{rank}\,\mathcal{M}_{\alpha}(C^{(D)}) =r_{\alpha}$ for
each $\alpha \in D.$ In consequence, $C^{(D)} \in\mathbb{R}_*^{%
\mathop{\mathchoice{\raise-0.22em\hbox{\huge $\times$}}
{\raise-0.05em\hbox{\Large $\times$}}{\hbox{\large
$\times$}}{\times}}_{\alpha \in D}r_{\alpha}}$ and (b) is true.
\end{proof}

\begin{remark}\label{identification}
From the proof of Lemma~\ref{characterization_FT} we have that 
$U_{\alpha}^{\min}(\mathbf{v}) = \mathrm{span}\, \{u_{i_{\alpha}}^{(\alpha)}: 1 \le i_{\alpha} \le r_{\alpha}\}$ and
$U_{D \setminus \{\alpha\}}^{\min}(\mathbf{v}) = \mathrm{span}\, \{\mathbf{U}_{i_{\alpha}}^{(\alpha)}: 1 \le i_{\alpha} \le r_{\alpha}\}$
for each $\alpha \in D.$ Furthermore, for 
$$
\mathbf{v} = \sum_{\substack{1 \le i_{\alpha} \le r_{\alpha} \\ \alpha \in D}} C^{(D)}_{(i_{\alpha})_{\alpha \in D}} \bigotimes_{\alpha \in D} u^{(\alpha)}_{i_{\alpha}} \in \mathfrak{M}_{\mathfrak{r}}(\mathbf{V}_D),
$$
there exists a natural diffeomorphism
$$
\mathfrak{M}_{\mathfrak{r}}\left( \left._a \bigotimes_{\alpha \in D} U_{\alpha}^{\min}(\mathbf{v}) \right. \right) \rightarrow \mathbb{R}_*^{%
\mathop{\mathchoice{\raise-0.22em\hbox{\huge $\times$}}
{\raise-0.05em\hbox{\Large $\times$}}{\hbox{\large
$\times$}}{\times}}_{\alpha \in D}r_{\alpha}}, \quad
\sum_{\substack{1 \le i_{\alpha} \le r_{\alpha} \\ \alpha \in D}} E^{(D)}_{(i_{\alpha})_{\alpha \in D}} \bigotimes_{\alpha \in D} u^{(\alpha)}_{i_{\alpha}}
\mapsto E^{(D)}.
$$
Thus we will identify each $u \in \mathfrak{M}_{\mathfrak{r}}\left( \left._a \bigotimes_{\alpha \in D} U_{\alpha}^{\min}(\mathbf{v}) \right. \right)$ with an element 
$E^{(D)} \in \mathbb{R}_*^{%
\mathop{\mathchoice{\raise-0.22em\hbox{\huge $\times$}}
{\raise-0.05em\hbox{\Large $\times$}}{\hbox{\large
$\times$}}{\times}}_{\alpha \in D}r_{\alpha}},$ once a basis
$\{u_{i_{\alpha}}^{(\alpha)}:1\leq i_{\alpha}\leq r_{\alpha}\}$ of $U_{\alpha}^{\min }(%
\mathbf{v})$ is fixed for each $\alpha\in D,$ by means of the equality
$$
\mathbf{u}=\mathbf{u}(E^{(D)}) = \sum_{\substack{1 \le i_{\alpha} \le r_{\alpha} \\ \alpha \in D}} E^{D}_{(i_{\alpha})_{\alpha \in D}} \bigotimes_{\alpha \in D} u^{(\alpha)}_{i_{\alpha}}.
$$
\end{remark}

\subsection{The manifold of tensors in Tucker format with fixed rank}\label{manifold_structure}

Assume that $(V_{\alpha},\|\cdot\|_{\alpha})$ is a normed space and  denote by 
$V_{\alpha_{\|\cdot\|_{\alpha}}}$ the Banach 
space obtained by the completion of $V_{\alpha}$ for each $\alpha \in D.$ 
Moreover, we also assume that $\|\cdot\|_D$ is a norm 
on the tensor space $\mathbf{V}_D = \left.
_{a}\bigotimes_{\alpha \in D}V_{\alpha}\right.$ such that the tensor
product map (\ref{bigotimes}) is continuous and hence, by Proposition \ref{Diff_otimes}, it is also 
$\mathcal{C}^{\infty}$-Fr\'echet differentiable.

\bigskip

Now, we proceed to provide a geometric structure for the set 
$\mathfrak{M}_{\mathfrak{r}}(\mathbf{V}_D).$
By Proposition~\ref{minimal} and Example~\ref{normed_grassmann} we have that
for each $\mathbf{v} \in \mathbf{V}_D$ the set $U_{\alpha}^{\min}(\mathbf{v}) 
\in \mathbb{G}_{r_{\alpha}}(V_{\alpha})$ for some $r_{\alpha} < \infty$ and $\alpha \in D.$
Since 
$
\mathbf{V}_D = \bigcup_{
\substack{\mathfrak{r} \in \mathcal{AD}(\mathbf{V}_D)}
} 
\mathfrak{M}_{\mathfrak{r}}(\mathbf{V}_D)$, 
thanks to Proposition~\ref{minimal} we can define a surjective map from a tensor space to an analytic Banach manifold:
$$
\varrho: \mathbf{V}_D \longrightarrow \bigcup_{(r_1,\ldots,r_d) \in \mathcal{AD}(\mathbf{V}_D)}\left(
\mathop{\mathchoice{\raise-0.22em\hbox{\huge $\times$}}
{\raise-0.05em\hbox{\Large $\times$}}{\hbox{\large
$\times$}}{\times}}_{\alpha \in D}\mathbb{G}_{_{r_{\alpha}}}(V_{\alpha})\right), \quad \mathbf{v} \mapsto 
(U_{\alpha}^{\min}(\mathbf{v}))_{\alpha \in D}.
$$
It allows us to consider for a fixed $\mathfrak{r} \in \mathcal{AD}(\mathbf{V}_D),$ 
$\mathfrak{r} \neq \mathbf{0},$ the restricted map
$$
\varrho_{\mathfrak{r}} = \varrho|_{\mathfrak{M}_{\mathfrak{r}}(\mathbf{V}_D)}: \mathfrak{M}_{\mathfrak{r}}(\mathbf{V}_D) = \varrho^{-1}\left( \mathop{\mathchoice{\raise-0.22em\hbox{\huge $\times$}}
{\raise-0.05em\hbox{\Large $\times$}}{\hbox{\large
$\times$}}{\times}}_{\alpha \in D}\mathbb{G}_{r_{\alpha}}(V_{\alpha})\right)
\longrightarrow \mathop{\mathchoice{\raise-0.22em\hbox{\huge $\times$}}
{\raise-0.05em\hbox{\Large $\times$}}{\hbox{\large
$\times$}}{\times}}_{\alpha \in D}\mathbb{G}_{r_{\alpha}}(V_{\alpha}), \quad \mathbf{v} \mapsto (U_{\alpha}^{\min}(\mathbf{v}))_{\alpha \in D},
$$
which is also surjective. For each $\alpha\in D$  the linear subspace $U_{\alpha}^{\min }(\mathbf{v}%
)\subset V_{\alpha}\subset V_{{\alpha}_{\Vert \cdot \Vert _{\alpha}}}$ belongs to the Banach
manifold $\mathbb{G}_{r_{\alpha}}(V_{\alpha})$ and hence there exists a closed subspace $W_{\alpha}^{\min }(%
\mathbf{v})$ such that $V_{{\alpha}_{\|\cdot\|_{\alpha}}} = U_{\alpha}^{\min }(%
\mathbf{v}) \oplus W_{\alpha}^{\min }(%
\mathbf{v})$  and a bijection (local chart) 
\begin{equation*}
\Psi _{\mathbf{v}}^{(\alpha)}:\mathbb{G}%
(W_{\alpha}^{\min }(\mathbf{v}),\mathbf{V}_{{\alpha}_{\Vert \cdot \Vert
_{\alpha}}})\rightarrow \mathcal{L}(U_{\alpha}^{\min }(\mathbf{v}),W_{\alpha}^{\min }(%
\mathbf{v}))
\end{equation*}%
given by 
\begin{equation*}
\Psi _{\mathbf{v}}^{(\alpha)}(U_{\alpha})=L_{\alpha}:=P_{W_{\alpha}^{\min }(\mathbf{v})\oplus U_{\alpha}^{\min }(\mathbf{v}%
)}|_{U_{\alpha}}\circ (P_{U_{\alpha}^{\min }(\mathbf{v})\oplus W_{\alpha}^{\min }(\mathbf{v}%
)}|_{U_{\alpha}})^{-1}.
\end{equation*}%
Moreover, $U_{\alpha}= (\Psi _{\mathbf{v}}^{(\alpha)})^{-1}(L_{\alpha}) = G(L_{\alpha})=\mathrm{span}\{(id_{\alpha}+L_{\alpha})(u_{\alpha}):u_{\alpha}\in
U_{\alpha}^{\min }(\mathbf{v})\}.$ Clearly, the map 
\begin{equation*}
\boldsymbol{\Psi }_{\mathbf{v}}:%
\mathop{\mathchoice{\raise-0.22em\hbox{\huge
$\times$}} {\raise-0.05em\hbox{\Large $\times$}}{\hbox{\large
$\times$}}{\times}}_{\alpha\in D}\mathbb{G}(W_{\alpha}^{\min }(%
\mathbf{v}),\mathbf{V}_{{\alpha}_{\Vert \cdot \Vert _{\alpha}}})\rightarrow %
\mathop{\mathchoice{\raise-0.22em\hbox{\huge $\times$}}
{\raise-0.05em\hbox{\Large $\times$}}{\hbox{\large $\times$}}{\times}}_{\alpha\in 
D}\mathcal{L}(U_{\alpha}^{\min }(\mathbf{v}),W_{\alpha}^{\min }(%
\mathbf{v})),
\end{equation*}%
defined as $\boldsymbol{\Psi }_{\mathbf{v}}:=%
\mathop{\mathchoice{\raise-0.22em\hbox{\huge $\times$}} {\raise-0.05em\hbox{\Large
$\times$}}{\hbox{\large $\times$}}{\times}}_{\alpha\in D}\Psi _{\mathbf{v}}^{(\alpha)}$ is also
bijective. Furthermore, it is a local chart for an element $\varrho_{\mathfrak{r}}(
\mathbf{v})$ in the
product manifold $\mathop{\mathchoice{\raise-0.22em\hbox{\huge $\times$}}
{\raise-0.05em\hbox{\Large $\times$}}{\hbox{\large
$\times$}}{\times}}_{\alpha \in D}\mathbb{G}_{r_{\alpha}}(V_{\alpha})$ such that $\boldsymbol{\Psi }_{\mathbf{v}}(\varrho_{\mathfrak{r}}(
\mathbf{v}))=\mathfrak{0}:=(0)_{k\in D}.$  

\bigskip

Now, for each $%
\mathbf{v}\in \mathfrak{M}_{\mathfrak{r}}(\mathbf{V}_{D})$ introduce the set 
\begin{equation*}
\mathcal{U}(\mathbf{v}):=\varrho _{\mathfrak{r}}^{-1}\left( 
\mathop{\mathchoice{\raise-0.22em\hbox{\huge
$\times$}} {\raise-0.05em\hbox{\Large $\times$}}{\hbox{\large
$\times$}}{\times}}_{\alpha \in D}\mathbb{G}(W_{\alpha}^{\min }(%
\mathbf{v}),V_{\alpha})\right) =\left\{ \mathbf{w}\in \mathfrak{M}_{\mathfrak{r}}(%
\mathbf{V}_{D}):U_{\alpha}^{\min }(\mathbf{w})\in \mathbb{G}(W_{\alpha}^{\min }(%
\mathbf{v}),V_{\alpha}),\, \alpha \in D\right\} .
\end{equation*}%
Recall that from Proposition  \ref{niceinverse} we can identify 
the linear space  $\mathcal{L}(U_{\alpha}^{\min }(%
\mathbf{v}),W_{\alpha}^{\min }(\mathbf{v}))$ with a sub-algebra of 
$\mathcal{L}(V_{\alpha_{\|\cdot\|_{\alpha}}},V_{\alpha_{\|\cdot\|_{\alpha}}})$ for $\alpha \in D.$
Then, from Example \ref{local_lee_group}, the map 
$$
\mathop{\mathchoice{\raise-0.22em\hbox{\huge
 $\times$}} {\raise-0.05em\hbox{\Large $\times$}}{\hbox{\large
 $\times$}}{\times}}_{\alpha\in D}\mathcal{L}(U_{\alpha}^{\min }(%
 \mathbf{v}),W_{\alpha}^{\min }(\mathbf{v})) \rightarrow
 \mathop{\mathchoice{\raise-0.22em\hbox{\huge
 $\times$}} {\raise-0.05em\hbox{\Large $\times$}}{\hbox{\large
 $\times$}}{\times}}_{\alpha\in D}\mathcal{L}(U_{\alpha}^{\min }(%
 \mathbf{v}),V_{\alpha}) 
$$
between normed spaces given by
$$
\mathfrak{L}=(L_{\alpha})_{\alpha \in D} \mapsto ((id_{\alpha} + L_{\alpha})|_{U_{\alpha}^{\min}(\mathbf{v})})_{\alpha \in D} = 
(\exp(L_{\alpha})|_{U_{\alpha}^{\min}(\mathbf{v})})_{\alpha \in D}
$$
is clearly $\mathcal{C}^{\infty}$-Fr\'echet differentiable.
Finally, from Proposition \ref{linear_bigotimes}, the map
$$
\mathop{\mathchoice{\raise-0.22em\hbox{\huge
 $\times$}} {\raise-0.05em\hbox{\Large $\times$}}{\hbox{\large
 $\times$}}{\times}}_{\alpha\in D}\mathcal{L}(U_{\alpha}^{\min }(%
 \mathbf{v}),W_{\alpha}^{\min }(\mathbf{v})) \rightarrow 
 \mathcal{L}\left( \left.
_{a}\bigotimes_{\alpha \in D}U_{\alpha}^{\min }(\mathbf{v})\right.
, \mathbf{V}_D \right)
$$
given by
$$
\mathfrak{L}=(L_{\alpha})_{\alpha \in D} \mapsto \bigotimes_{\alpha \in D}\exp(L_{\alpha})|_{U_{\alpha}^{\min}(\mathbf{v})}
$$
is also $\mathcal{C}^{\infty}$-Fr\'echet differentiable.

\bigskip

Our next step is to characterise the representation of tensors that 
belong to $\mathcal{U}(\mathbf{v})$ by using the following lemma.

\begin{lemma}\label{element}
Assume that
$(V_{\alpha},\|\cdot\|_{\alpha})$ is a normed space for each $\alpha \in D$ 
and that $\|\cdot\|_D$ is a norm 
on the tensor space $\mathbf{V}_D = \left.
_{a}\bigotimes_{\alpha \in D}V_{\alpha}\right.$ such that the tensor
product map (\ref{bigotimes}) is continuous.
For $\mathbf{v} \in  \mathfrak{M}_{\mathfrak{r}}(\mathbf{V}_D)$ the following statements are equivalent.
\begin{enumerate}
\item[(a)] $\mathbf{w} \in \mathcal{U}(\mathbf{v}).$ 
\item[(b)] There exists a unique
$$
(\mathfrak{L}, \mathbf{u}(E^{(D)})) =( (L_{\alpha})_{\alpha \in D},\mathbf{u}(E^{(D)})) \in \mathop{\mathchoice{\raise-0.22em\hbox{\huge
 $\times$}} {\raise-0.05em\hbox{\Large $\times$}}{\hbox{\large
 $\times$}}{\times}}_{\alpha\in D}\mathcal{L}(U_{\alpha}^{\min }(%
 \mathbf{v}),W_{\alpha}^{\min }(\mathbf{v}))
\times 
\mathfrak{M}_{\mathfrak{r}}\left( \left._a \bigotimes_{\alpha \in D} U_{\alpha}^{\min}(\mathbf{v}) \right. \right)
 $$ 
 such that
\begin{align*}
\mathbf{w} =  \left( \bigotimes_{
\alpha \in D}\exp(L_{\alpha})\right) (\mathbf{u}(E^{(D)})).
\end{align*}
\end{enumerate}
\end{lemma}
\begin{proof}
Assume that $\mathbf{w}\in \mathcal{U}(\mathbf{v}).$ Then 
we have the following facts:
\begin{enumerate}
\item[(i)] From Lemma \ref{characterization_FT}(b) there exist bases
$\mathcal{B}_\alpha=\{u_{i_{\alpha}}^{(\alpha)}: 1 \le i_{\alpha} \le r_{\alpha}\}$, $\alpha\in D$, 
and a unique $C^{(D)} \in \mathbb{R}_*^{%
\mathop{\mathchoice{\raise-0.22em\hbox{\huge $\times$}}
{\raise-0.05em\hbox{\Large $\times$}}{\hbox{\large
$\times$}}{\times}}_{\alpha \in D}r_{\alpha}},$ once the bases are fixed, such that
$\mathbf{v} = \sum_{\substack{1 \le i_{\alpha} \le r_{\alpha} \\ \alpha \in D}} C^{(D)}_{(i_{\alpha})_{\alpha \in D}} \bigotimes_{\alpha \in D} u^{(\alpha)}_{i_{\alpha}}  \in  \mathfrak{M}_{\mathfrak{r}}(\mathbf{v}).$   
From Remark \ref{identification}, we know that 
$\mathcal{B}_\alpha$ is a basis  of $U_{\alpha}^{\min }(\mathbf{v})$ for $\alpha\in D.$  Now, we will
consider that the bases $\mathcal{B}_\alpha$, $\alpha \in D$, are fixed. 

\item[(ii)] Since $U_{\alpha}^{\min}(\mathbf{w}) \in \mathbb{G}%
(W_{\alpha}^{\min}(\mathbf{v}),V_{{\alpha}_{\|\cdot\|_{\alpha}}}),$ for $\alpha \in D,$ 
there exists a unique
\begin{equation*}
\mathfrak{L}=(L_{\alpha})_{\alpha \in D} \in 
\mathop{\mathchoice{\raise-0.22em\hbox{\huge
$\times$}} {\raise-0.05em\hbox{\Large $\times$}}{\hbox{\large
$\times$}}{\times}}_{\alpha \in D}\mathcal{L}(U_{\alpha}^{\min}(\mathbf{v}%
),W_{\alpha}^{\min}(\mathbf{v}))
\end{equation*}
such that $\boldsymbol{\Psi }_{\mathbf{v}}(\varrho_{\mathfrak{r}}(\mathbf{w}%
)) = \mathfrak{L},$ that is, $U_{\alpha}^{\min}(\mathbf{w})=G(L_{\alpha}) = 
\mathrm{span}\,\{(id_{\alpha}+L_{\alpha})(u_{i_{\alpha}}^{(\alpha)}): 1 \le i_{\alpha} \le r_{\alpha} \}$ for all $\alpha
\in D.$ Then 
$$(U_{\alpha}^{\min}(\mathbf{w}))_{\alpha \in D} = \boldsymbol{%
\Psi }_{\mathbf{v}}^{-1}(\mathfrak{L})
$$ and we can construct from $\mathcal{B}_{\alpha},$ a basis of
$U_{\alpha}^{\min}(\mathbf{w}).$
In particular we have $%
\boldsymbol{\Psi }_{\mathbf{v}}(\varrho_{\mathfrak{r}}(\mathbf{v})) =(0)_{\alpha
\in D}.$

\item[(iii)] Now by Lemma~\ref{characterization_FT}(b), since a basis  
of $U_{\alpha}^{\min}(\mathbf{w})=G(L_{\alpha}) = 
\mathrm{span}\,\{(id_{\alpha}+L_{\alpha})(u_{i_{\alpha}}^{(\alpha)}): 1 \le i_{\alpha} \le r_{\alpha} \}$
for $\alpha \in D$ is fixed, there exists a unique 
$ E^{(D)} \in \mathbb{R%
}_*^{\mathfrak{r}}$ such that 
\begin{equation}  \label{w_equation}
\mathbf{w}=\sum_{\substack{ 1\leq i_{\alpha }\leq r_{\alpha }  \\ \alpha \in
D}}E_{(i_{\alpha })_{\alpha \in D}}^{(D)}\bigotimes_{\alpha \in D}
(id_{\alpha}+L_{\alpha})(u_{i_{\alpha}}^{(\alpha)}) = \bigotimes_{\alpha \in D}
(id_{\alpha}+L_{\alpha}) \left( \mathbf{u}(E^{(D)})\right),
\end{equation}
where
$$
\mathbf{u}(E^{(D)}) := \sum_{\substack{ 1\leq i_{\alpha }\leq r_{\alpha }  \\ \alpha \in
D}}E_{(i_{\alpha })_{\alpha \in D}}^{(D)}\bigotimes_{\alpha \in D}
u_{i_{\alpha}}^{(\alpha)} \in \mathfrak{M}_{\mathfrak{r}}\left(
\left._a \bigotimes_{\alpha \in D} U_{\alpha}^{\min}(\mathbf{v}) \right.
\right) = \mathbb{R}_*^{
\mathop{\mathchoice{\raise-0.22em\hbox{\huge $\times$}}
{\raise-0.05em\hbox{\Large $\times$}}{\hbox{\large
$\times$}}{\times}}_{\alpha \in D}r_{\alpha}}.
$$
\end{enumerate}
It follows (b). From what was said above, (b) clearly implies (a).
\end{proof}

\bigskip

\begin{remark}
We can interpret Lemma~\ref{element} as follows. $\mathbf{w}\in \mathcal{U}(%
\mathbf{v})$ holds if and only if 
\begin{equation*}
\mathbf{w}\in \left( \bigotimes_{\alpha\in 
D}\exp(L_{\alpha})\right) \left(\mathfrak{M}_{\mathfrak{r}}\left( \left. _{a}\bigotimes_{\alpha\in 
D}U_{\alpha}^{\min }(\mathbf{v})\right. \right) \right)
\end{equation*}%
for some $\mathfrak{L} = (L_\alpha)_{\alpha \in D}\in 
\mathop{\mathchoice{\raise-0.22em\hbox{\huge
$\times$}} {\raise-0.05em\hbox{\Large $\times$}}{\hbox{\large
$\times$}}{\times}}_{\alpha\in D}\mathcal{L}(U_{\alpha}^{\min }(%
\mathbf{v}),W_{\alpha}^{\min }(\mathbf{v})).$ In consequence, each neighbourhood
of $\mathbf{v}$ in $\mathfrak{M}_{\mathfrak{r}}(\mathbf{V}_{D})$ can be
written as 
\begin{equation}\label{openset}
\mathcal{U}(\mathbf{v})=\bigcup_{\mathfrak{L}\in 
\mathop{\mathchoice{\raise-0.22em\hbox{\huge
$\times$}} {\raise-0.05em\hbox{\Large $\times$}}{\hbox{\large
$\times$}}{\times}}_{\alpha \in D}\mathcal{L}(U_{\alpha}^{\min }(%
\mathbf{v}),W_{\alpha}^{\min }(\mathbf{v}))}
\left( \bigotimes_{\alpha\in D}\exp(L_{\alpha})\right)  \left(\mathfrak{M}_{\mathfrak{r}}\left(
\left. _{a}\bigotimes_{\alpha\in D}U_{\alpha}^{\min }(\mathbf{v}%
)\right. \right) \right) ,
\end{equation}%
that is, a union of manifolds (each of them diffeomeorphic to $\mathbb{R}_*^{
\mathop{\mathchoice{\raise-0.22em\hbox{\huge $\times$}}
{\raise-0.05em\hbox{\Large $\times$}}{\hbox{\large
$\times$}}{\times}}_{\alpha \in D}r_{\alpha}}$) indexed by a Banach manifold. 
\end{remark}

\bigskip

Now, also by using Lemma~\ref{element}, 
we construct an explicit manifold structure for $\mathfrak{M}_{\mathfrak{r}}
(\mathbf{V}_D).$ Indeed, Lemma~\ref{element} allows 
us to define for each $\mathbf{v} \in \mathfrak{M}_{\mathfrak{r}}(\mathbf{v}),$ 
once a basis of $U_{\alpha}^{\min}(\mathbf{v})$ for each $\alpha \in D$ is fixed, a bijective
map 
\begin{equation*}
\xi_{\mathbf{v}}:\mathcal{U}(%
\mathbf{v})\rightarrow \left( 
\mathop{\mathchoice{\raise-0.22em\hbox{\huge
$\times$}} {\raise-0.05em\hbox{\Large $\times$}}{\hbox{\large
$\times$}}{\times}}_{\alpha \in D}\mathcal{L}(U_{\alpha}^{\min }(%
\mathbf{v}),W_{\alpha}^{\min }(\mathbf{v}))\right) \times \mathbb{R}^{
\mathop{\mathchoice{\raise-0.22em\hbox{\huge $\times$}}
{\raise-0.05em\hbox{\Large $\times$}}{\hbox{\large
$\times$}}{\times}}_{\alpha \in D}r_{\alpha}}_*,
\end{equation*}%
by
$$
\xi_{\mathbf{v}} \left(\left( \bigotimes_{\alpha\in D}\exp(L_{\alpha})\right) (\mathbf{u}(C^{(D)}))\right) := (\mathfrak{L},C^{(D)}),
$$
where $\mathfrak{L}:=(L_{\alpha})_{\alpha \in D}.$
Clearly, $\xi_{\mathbf{v}}$ is a bijective map and hence $\mathcal{U}(\mathbf{v})$
can be identified with the Banach manifold
$$\left( 
\mathop{\mathchoice{\raise-0.22em\hbox{\huge
$\times$}} {\raise-0.05em\hbox{\Large $\times$}}{\hbox{\large
$\times$}}{\times}}_{\alpha \in D}\mathbb{G}(W_{\alpha}^{\min }(%
\mathbf{v}),V_{\alpha})\right) \times \mathbb{R}_*^{
\mathop{\mathchoice{\raise-0.22em\hbox{\huge $\times$}}
{\raise-0.05em\hbox{\Large $\times$}}{\hbox{\large
$\times$}}{\times}}_{\alpha \in D}r_{\alpha}},
$$  
which is modelled on the Banach space
$$
\left( 
\mathop{\mathchoice{\raise-0.22em\hbox{\huge
$\times$}} {\raise-0.05em\hbox{\Large $\times$}}{\hbox{\large
$\times$}}{\times}}_{\alpha \in D}\mathcal{L}(U_{\alpha}^{\min }(%
\mathbf{v}),W_{\alpha}^{\min }(%
\mathbf{v}))\right) \times \mathbb{R}^{
\mathop{\mathchoice{\raise-0.22em\hbox{\huge $\times$}}
{\raise-0.05em\hbox{\Large $\times$}}{\hbox{\large
$\times$}}{\times}}_{\alpha \in D}r_{\alpha}}.
$$

\bigskip

The next lemma allows us to prove 
that $\{(\mathcal{U}(\mathbf{v}),\xi_{\mathbf{v}})\}_{\mathbf{v} \in \mathfrak{M}_{\mathfrak{r}}
(\mathbf{V}_D)}$ is a local chart system for the set of tensors in Tucker format with fixed
rank $\mathfrak{r}.$

\begin{lemma}
\label{premanifold} Assume that
$(V_{\alpha},\|\cdot\|_{\alpha})$ is a normed space for each $\alpha \in D$ 
and that $\|\cdot\|_D$ is a norm 
on the tensor space $\mathbf{V}_D = \left.
_{a}\bigotimes_{\alpha \in D}V_{\alpha}\right.$ such that the tensor
product map (\ref{bigotimes}) is continuous. Let $\mathbf{v},\mathbf{v}^{\prime }\in \mathfrak{M}_{%
\mathfrak{r}}(\mathbf{V}_{D})$ be such that $\mathcal{U}(\mathbf{v}) \cap 
\mathcal{U}(\mathbf{v}^{\prime }) \neq \emptyset.$ Then the bijective map 
\begin{equation*}
\xi_{\mathbf{v}^{\prime}} \circ\xi_{\mathbf{v}}^{-1}:\xi_{\mathbf{v}}\left(\mathcal{U}(\mathbf{v}) \cap 
\mathcal{U}(\mathbf{v}^{\prime })\right)
\rightarrow \xi_{\mathbf{v}^{\prime}}\left(\mathcal{U}(\mathbf{v}) \cap 
\mathcal{U}(\mathbf{v}^{\prime })\right)
\end{equation*}
is $\mathcal{C}^{\infty}$-Fr\'{e}chet differentiable.
\end{lemma}

\begin{proof}
Let $\mathbf{w}\in \mathcal{U}(\mathbf{v})\cap 
\mathcal{U}(\mathbf{v}^{\prime })$ be such that $\xi _{\mathbf{v}}(\mathbf{w})=(\mathfrak{L},\mathbf{u}(C^{(D)})%
)$ and $\xi_{\mathbf{v}^{\prime }}(\mathbf{w})=(\mathfrak{L}^{\prime },\mathbf{u}^{\prime}(E^{(D)}),$ that is,
$$
(\xi_{\mathbf{v}^{\prime
}} \circ \xi^{-1}_{\mathbf{v}})(\mathfrak{L},\mathbf{u}(C^{(D)})) = (\mathfrak{L}^{\prime },
\mathbf{u}^{\prime}(E^{(D)})).
$$
Since $%
\mathbf{w}\in \mathcal{U}(\mathbf{v})\cap \mathcal{U}(\mathbf{v}^{\prime })$
then 
\begin{equation*}
\varrho _{\mathfrak{r}}(\mathbf{w})=(U_{\alpha}^{\min }(\mathbf{w}))_{\alpha\in 
D}\in \left( 
\mathop{\mathchoice{\raise-0.22em\hbox{\huge
$\times$}} {\raise-0.05em\hbox{\Large $\times$}}{\hbox{\large
$\times$}}{\times}}_{\alpha \in D}\mathbb{G}(W_{\alpha}^{\min }(%
\mathbf{v}),V_{\alpha})\right) \cap \left( 
\mathop{\mathchoice{\raise-0.22em\hbox{\huge
$\times$}} {\raise-0.05em\hbox{\Large $\times$}}{\hbox{\large
$\times$}}{\times}}_{\alpha \in D}\mathbb{G}(W_{\alpha}^{\min }(%
\mathbf{v}^{\prime }),V_{\alpha})\right)
\end{equation*}
and 
\begin{equation*}
(\Psi _{\mathbf{v}^{\prime }}\circ \Psi _{\mathbf{v}}^{-1})(\Psi _{\mathbf{v}%
}((U_{\alpha}^{\min }(\mathbf{w}))_{\alpha\in D}))=\Psi _{\mathbf{v}%
^{\prime }}(U_{\alpha}^{\min }(\mathbf{w}))_{\alpha \in D}),
\end{equation*}%
that is, 
\begin{equation*}
(\Psi _{\mathbf{v}^{\prime }}\circ \Psi _{\mathbf{v}}^{-1})(\mathfrak{L})=%
\mathfrak{L}^{\prime }.
\end{equation*}%
Hence 
\begin{equation*}
\xi_{\mathbf{v}%
^{\prime }} (\mathbf{w})=((\Psi _{\mathbf{v}^{\prime }}\circ \Psi _{%
\mathbf{v}}^{-1})(\mathfrak{L}),\mathbf{u}^{\prime}(E^{(D)})),
\end{equation*}%
where $\Psi _{\mathbf{v}^{\prime }}\circ \Psi _{\mathbf{v}}^{-1}$ is an
analytic map. On the other hand, since
\begin{align*}
\mathbf{w} & = \xi^{-1}_{\mathbf{v}}(\mathfrak{L},C^{(D)}) =\left( \bigotimes_{\alpha\in D}\exp(L_{\alpha})\right) (%
\mathbf{u}(C^{(D)})) =\xi^{-1}_{\mathbf{v}%
^{\prime }}(\mathfrak{L}',E^{(D)}) = \left( \bigotimes_{\alpha\in D}\exp(L_{\alpha}^{\prime
})\right) (\mathbf{u}^{\prime }(E^{(D)})),
\end{align*}%
we have
\begin{align*}
\mathbf{u}'(E^{(D)}) & =\left(  \bigotimes_{\alpha\in D} \exp(-L_{\alpha}')\circ \exp(L_{\alpha})\right) 
(\mathbf{u}(C^{(D)})) = \left(  \bigotimes_{\alpha\in D} \exp(L_{\alpha}-L_{\alpha}^{\prime})\right) 
(\mathbf{u}(C^{(D)})),
\end{align*}%
because from Proposition \ref{niceinverse} 
$L_{\alpha} \circ L_{\alpha}^{\prime} = L_{\alpha}^{\prime} \circ L_{\alpha} = 0$ holds.
In consequence, 
\begin{equation*}
\mathbf{u}^{\prime}(E^{(D)})=f(\mathfrak{L},\mathbf{u}(C^{(D)})):= \left(\bigotimes_{\alpha\in D} 
 \exp(L_{\alpha} - (\Psi _{\mathbf{v}^{\prime}}^{(\alpha)} \circ (\Psi _{\mathbf{v}}^{(\alpha)})^{-1})(L_{\alpha})\right) 
(\mathbf{u}(C^{(D)}))
\end{equation*}%
where
\begin{equation*}
f:\mathop{\mathchoice{\raise-0.22em\hbox{\huge $\times$}}
{\raise-0.05em\hbox{\Large $\times$}}{\hbox{\large $\times$}}{\times}}_{\alpha\in 
D}\mathcal{L}(U_{\alpha}^{\min }(\mathbf{v}),W_{\alpha}^{\min }(%
\mathbf{v}))\times \mathfrak{M}_{\mathfrak{r}}\left(\left. _{a}\bigotimes_{\alpha \in D}U_{\alpha}^{\min }(\mathbf{v}%
)\right. \right)\rightarrow \mathfrak{M}_{\mathfrak{r}}\left(\left. _{a}\bigotimes_{\alpha \in D}U_{\alpha}^{\min }(\mathbf{v}^{\prime}
)\right. \right).
\end{equation*}
To prove the lemma we claim that the map $f$ is $\mathcal{C}^{\infty}$-Fr\'echet differentiable.

\bigskip

Recall that for each $\alpha \in D$ the map given by
$$
L_{\alpha}' = \left(\Psi _{\mathbf{v}^{\prime}}^{(\alpha)} \circ (\Psi _{\mathbf{v}}^{(\alpha)})^{-1}\right)(L_{\alpha})
$$ 
is analytic because $\mathbb{G}_{r_{\alpha}}(V_{\alpha})$ is an analytic Banach manifold.
Since we can identify the linear space  $\mathcal{L}(U_{\alpha}^{\min }(%
\mathbf{v}),W_{\alpha}^{\min }(\mathbf{v}))$ with a sub-algebra of 
$\mathcal{L}(V_{\alpha_{\|\cdot\|_{\alpha}}},V_{\alpha_{\|\cdot\|_{\alpha}}}),$
from Example \ref{local_lee_group}, we know that
$$\exp:\mathcal{L}(U_{\alpha}^{\min }(%
 \mathbf{v}),W_{\alpha}^{\min }(\mathbf{v}))
 \rightarrow \mathrm{GL}(\mathcal{L}(U_{\alpha}^{\min }(%
 \mathbf{v}),W_{\alpha}^{\min }(\mathbf{v})))$$  
 is analytic for each $\alpha \in D.$
In consequence, the map 
$$
L_{\alpha} \mapsto \exp\left(L_{\alpha}-\left(\Psi _{\mathbf{v}^{\prime}}^{(\alpha)} \circ (\Psi _{\mathbf{v}}^{(\alpha)})^{-1}\right)(L_{\alpha})\right)
$$
is also analytic for each $\alpha \in D.$ Finally, we conclude by using 
Proposition \ref{linear_bigotimes} that the map
$$
\mathop{\mathchoice{\raise-0.22em\hbox{\huge
 $\times$}} {\raise-0.05em\hbox{\Large $\times$}}{\hbox{\large
 $\times$}}{\times}}_{\alpha\in D}\mathcal{L}(U_{\alpha}^{\min }(%
 \mathbf{v}),W_{\alpha}^{\min }(\mathbf{v})) \rightarrow 
 \mathcal{L}\left( \left.
_{a}\bigotimes_{\alpha \in D}U_{\alpha}^{\min }(\mathbf{v})\right.
, \mathbf{V}_D \right)
$$
given by
$$
(L_{\alpha})_{\alpha \in D} \mapsto \bigotimes_{\alpha \in D}\left.\exp\left(L_{\alpha}-\left(\Psi _{\mathbf{v}^{\prime}}^{(\alpha)} \circ (\Psi _{\mathbf{v}}^{(\alpha)})^{-1}\right)(L_{\alpha})\right)\right|_{U_{\alpha}^{\min}(\mathbf{v})}
$$
is $\mathcal{C}^{\infty}$-Fr\'echet differentiable. Observe that $f$ can be written by using the evaluation map
$$
\mathrm{eval}: 
\mathcal{L}\left( \left.
_{a}\bigotimes_{\alpha \in D}U_{\alpha}^{\min }(\mathbf{v})\right.
,\left. _{a}\bigotimes_{\alpha \in D}U_{\alpha}^{\min }(\mathbf{v}^{\prime}
)\right. \right)
\times \left. _{a}\bigotimes_{k\in D}U_{k}^{\min }(\mathbf{v}
)\right.  \rightarrow \left. _{a}\bigotimes_{k\in D}U_{k}^{\min }(\mathbf{v}^{\prime}
)\right.
$$
given by
$$
\mathrm{eval}\left(F, \sum_{\substack{ 1\leq i_{\alpha }\leq r_{\alpha }  \\ \alpha \in
D}}E_{(i_{\alpha })_{\alpha \in D}}^{(D)}\bigotimes_{\alpha \in D}
u_{i_{\alpha}}^{(\alpha)}\right) = F \left(\sum_{\substack{ 1\leq i_{\alpha }\leq r_{\alpha }  \\ \alpha \in
D}}E_{(i_{\alpha })_{\alpha \in D}}^{(D)}u_{i_{\alpha}}^{(\alpha)} \right),
$$
which is multilinear and continuous. From Proposition \ref{multilinear}, 
it is also $\mathcal{C}^{\infty}$-Fr\'{e}chet differentiable.
Since
$$
 f(\mathfrak{L},\mathbf{u}(C^{(D)})) = \mathrm{eval}\left(\left( \bigotimes_{\alpha\in D}\exp\left(L_{\alpha} - (\Psi _{\mathbf{v}^{\prime}}^{(\alpha)} \circ (\Psi _{\mathbf{v}}^{(\alpha)})^{-1})(L_{\alpha})\right)\right) , \mathbf{u}(C^{(D)}) \right),
$$
the claim follows. We recall that 
$$
 \mathfrak{M}_{\mathfrak{r}}\left(\left. _{a}\bigotimes_{k\in D}U_{k}^{\min }(\mathbf{v}^{\prime}
)\right. \right) = \mathfrak{M}_{\mathfrak{r}}\left(\left. _{a}\bigotimes_{k\in D}U_{k}^{\min }(\mathbf{v}
)\right. \right) = \mathbb{R}_*^{
\mathop{\mathchoice{\raise-0.22em\hbox{\huge $\times$}}
{\raise-0.05em\hbox{\Large $\times$}}{\hbox{\large
$\times$}}{\times}}_{\alpha \in D}r_{\alpha}}.
$$
Thus the lemma is proved.
\end{proof}

\begin{remark}\label{analytic1}
Observe that if we assume that $(V_{\alpha},\|\cdot\|_{\alpha})$ is a complex Banach space 
for each $\alpha \in D$ and  $\|\cdot\|_D$ is a norm 
on the complex tensor space $\mathbf{V}_D = \left.
_{a}\bigotimes_{\alpha \in D}V_{\alpha}\right.$ such that the tensor
product map (\ref{bigotimes}) is continuous, from Proposition \ref{Analytic_otimes}, we have that
the extension of the tensor product map (\ref{bigotimes}) is analytic. Moreover, the map
$$
\mathop{\mathchoice{\raise-0.22em\hbox{\huge
 $\times$}} {\raise-0.05em\hbox{\Large $\times$}}{\hbox{\large
 $\times$}}{\times}}_{\alpha\in D}\mathcal{L}(U_{\alpha}^{\min }(%
 \mathbf{v}),W_{\alpha}^{\min }(\mathbf{v})) \rightarrow
 \mathop{\mathchoice{\raise-0.22em\hbox{\huge
 $\times$}} {\raise-0.05em\hbox{\Large $\times$}}{\hbox{\large
 $\times$}}{\times}}_{\alpha\in D}\mathcal{L}(U_{\alpha}^{\min }(%
 \mathbf{v}),V_{\alpha}), \quad
 \mathfrak{L}=(L_{\alpha})_{\alpha \in D} \mapsto (id_{\alpha} + L_{\alpha})_{\alpha \in D} = 
 (\exp(L_{\alpha}))_{\alpha \in D}
$$
between the product of complex Banach spaces is clearly analytic.
In consequence, under the above assumptions it can be shown that 
the bijective map $\xi_{\mathbf{v}^{\prime}} \circ\xi_{\mathbf{v}}^{-1}$ is analytic.
\end{remark}

Before stating the next result we recall the definition of  a fibre bundle.

\begin{definition}
A $\mathcal{C}^k$-fibre bundle $(E,B,\pi),$ where $k \ge 0,$ 
with typical fibre $F$ (a given manifold) is a $\mathcal{C}^k$-surjective morphism 
of $\mathcal{C}^k$ 
manifolds $\pi : E \rightarrow B$ 
which is locally a product, that is, the  $\mathcal{C}^k$-manifold $B$ 
has an open atlas $\{(U_{\alpha},\xi_{\alpha})\}_{\alpha \in A}$ 
such that for each $\alpha \in A$ there is a $\mathcal{C}^k$ diffeomorphism 
$\chi_{\alpha}: \pi^{-1}(U_{\alpha}) \rightarrow U_{\alpha} \times F$ 
such that $p_{\alpha} \circ \chi_{\alpha} = \pi,$ 
where $p_{\alpha} : U_{\alpha} \times F \rightarrow U_{\alpha}$ is the projection. 
The $\mathcal{C}^k$ manifolds $E$ and $B$ 
are called the \emph{total space} and \emph{base} of the fibre bundle, respectively. 
For each $b \in B,$ $\pi^{-1}(b) = E_b$ is called the \emph{fibre} 
over $b.$ The $\mathcal{C}^k$ diffeomorphisms $\chi_{\alpha}$ 
are called \emph{fibre bundle charts}.
\end{definition}

\begin{theorem}
\label{Tucker_Banach_Manifold} Assume that
$(V_{\alpha},\|\cdot\|_{\alpha})$ is a normed space for each $\alpha \in D$ 
and that $\|\cdot\|_D$ is a norm 
on the tensor space $\mathbf{V}_D = \left.
_{a}\bigotimes_{\alpha \in D}V_{\alpha}\right.$ such that the tensor
product map (\ref{bigotimes}) is continuous.
Then the collection $\{\mathcal{U}(\mathbf{v}),\xi_{\mathbf{v}}\}_{\mathbf{v}\in \mathfrak{M}_{\mathfrak{%
r}}(\mathbf{V}_{D})}$ is a  $\mathcal{C}^{\infty}$-atlas for $\mathfrak{M}_{\mathfrak{r}}(%
\mathbf{V}_{D})$ and hence it is a $\mathcal{C}^{\infty}$-Banach manifold modelled on a Banach space
$$
\left( 
\mathop{\mathchoice{\raise-0.22em\hbox{\huge
$\times$}} {\raise-0.05em\hbox{\Large $\times$}}{\hbox{\large
$\times$}}{\times}}_{\alpha \in D}\mathcal{L}(U_{\alpha},W_{\alpha})\right) \times \mathbb{R}^{
\mathop{\mathchoice{\raise-0.22em\hbox{\huge $\times$}}
{\raise-0.05em\hbox{\Large $\times$}}{\hbox{\large
$\times$}}{\times}}_{\alpha \in D}r_{\alpha}},
$$
here $U_{\alpha} \in \mathbb{G}_{r_{\alpha}}(V_{\alpha})$ and $V_{\alpha_{\|\cdot\|_{\alpha}}} = U_{\alpha} \oplus W_{\alpha},$ where $V_{{\alpha}_{\|\cdot\|_{\alpha}}}$ is the completion of
$V_{\alpha}$ for $\alpha \in D.$ Moreover, 
$$
\left(\mathfrak{M}_{\mathfrak{r}}(%
\mathbf{V}_{D}),\mathop{\mathchoice{\raise-0.22em\hbox{\huge
 $\times$}} {\raise-0.05em\hbox{\Large $\times$}}{\hbox{\large
 $\times$}}{\times}}_{\alpha\in D}\mathbb{G}_{r_{\alpha}}(V_{\alpha}),\varrho_{\mathfrak{r}}\right)
$$
is a $\mathcal{C}^{\infty}$-fibre bundle with typical
 fibre $\mathbb{R}_{*}^{\mathop{\mathchoice{\raise-0.22em\hbox{\huge
 $\times$}} {\raise-0.05em\hbox{\Large $\times$}}{\hbox{\large
 $\times$}}{\times}}_{\alpha\in D}r_{\alpha}}.$
\end{theorem}
\begin{proof}
Since  $\{(\mathcal{U}(\mathbf{v}),\xi_{\mathbf{v}})\}_{\mathbf{v} \in \mathfrak{M}_{\mathfrak{r}}
(\mathbf{V}_D)}$ satisfies AT1, Lemma ~\ref{element} implies AT2 and AT3 follows from Lemma  \ref{premanifold},  we obtain the first statement. To prove the second one we observe that the
local chart system $\{(\mathcal{U}(\mathbf{v}),\xi_{\mathbf{v}})\}$ for the manifold
$\mathfrak{M}_{\mathfrak{r}}(\mathbf{V}_{D})$ allows us to write the morphism
$$
\varrho_{\mathfrak{r}}: \mathfrak{M}_{\mathfrak{r}}(%
 \mathbf{V}_{D}) \rightarrow \mathop{\mathchoice{\raise-0.22em\hbox{\huge
  $\times$}} {\raise-0.05em\hbox{\Large $\times$}}{\hbox{\large
  $\times$}}{\times}}_{\alpha\in D}\mathbb{G}_{r_{\alpha}}(V_{\alpha}), \quad
  \mathbf{v} \mapsto (U_{\alpha}^{\min}(\mathbf{v}))_{\alpha \in D},
  $$ 
locally as a map
$$
\left( 
\mathop{\mathchoice{\raise-0.22em\hbox{\huge
$\times$}} {\raise-0.05em\hbox{\Large $\times$}}{\hbox{\large
$\times$}}{\times}}_{\alpha \in D}\mathcal{L}(U_{\alpha}^{\min}(\mathbf{v}),W_{\alpha}^{\min}(\mathbf{v}))\right) \times \mathbb{R}_*^{
\mathop{\mathchoice{\raise-0.22em\hbox{\huge $\times$}}
{\raise-0.05em\hbox{\Large $\times$}}{\hbox{\large
$\times$}}{\times}}_{\alpha \in D}r_{\alpha}} \rightarrow \mathop{\mathchoice{\raise-0.22em\hbox{\huge
$\times$}} {\raise-0.05em\hbox{\Large $\times$}}{\hbox{\large
$\times$}}{\times}}_{\alpha \in D}\mathcal{L}(U_{\alpha}^{\min}(\mathbf{v}),W_{\alpha}^{\min}(\mathbf{v})), 
$$
given by
$$
((L_{\alpha})_{\alpha \in D},E^{(D)}) \mapsto (L_{\alpha})_{\alpha \in D}.
$$
Thus, $\varrho_{\mathfrak{r}}$ is a $C^{\infty}$-surjective morphism. Moreover, by construction of the atlases, for each $\mathbf{v} \in \mathfrak{M}_{\mathfrak{r}}(\mathbf{V}_{D})$ the map
$\chi_{\mathbf{v}}:=(\Psi_{\mathbf{v}} \times id_{\mathbb{R}_{*}^{\mathop{\mathchoice{\raise-0.22em\hbox{\huge
 $\times$}} {\raise-0.05em\hbox{\Large $\times$}}{\hbox{\large
 $\times$}}{\times}}_{\alpha\in D}r_{\alpha}}}
) \circ \xi_{\mathbf{v}}$ where
$$
\chi_{\mathbf{v}}:\mathcal{U}(\mathbf{v})=\varrho _{\mathfrak{r}}^{-1}\left( 
\mathop{\mathchoice{\raise-0.22em\hbox{\huge
$\times$}} {\raise-0.05em\hbox{\Large $\times$}}{\hbox{\large
$\times$}}{\times}}_{\alpha \in D}\mathbb{G}(W_{\alpha}^{\min }(%
\mathbf{v}),V_{\alpha})\right) \rightarrow \left( 
\mathop{\mathchoice{\raise-0.22em\hbox{\huge
$\times$}} {\raise-0.05em\hbox{\Large $\times$}}{\hbox{\large
$\times$}}{\times}}_{\alpha \in D}\mathbb{G}(W_{\alpha}^{\min }(%
\mathbf{v}),V_{\alpha})\right) \times \mathbb{R}_{*}^{\mathop{\mathchoice{\raise-0.22em\hbox{\huge
 $\times$}} {\raise-0.05em\hbox{\Large $\times$}}{\hbox{\large
 $\times$}}{\times}}_{\alpha\in D}r_{\alpha}}
$$
is a $C^{\infty}$-diffeomorphism satisfying
$\pi_{\mathbf{v}} \circ \chi_{\mathbf{v}} = \varrho_{\mathfrak{r}}$ where
$$
\pi_{\mathbf{v}}:\left( 
\mathop{\mathchoice{\raise-0.22em\hbox{\huge
$\times$}} {\raise-0.05em\hbox{\Large $\times$}}{\hbox{\large
$\times$}}{\times}}_{\alpha \in D}\mathbb{G}(W_{\alpha}^{\min }(%
\mathbf{v}),V_{\alpha})\right) \times \mathbb{R}_{*}^{\mathop{\mathchoice{\raise-0.22em\hbox{\huge
 $\times$}} {\raise-0.05em\hbox{\Large $\times$}}{\hbox{\large
 $\times$}}{\times}}_{\alpha\in D}r_{\alpha}} \rightarrow 
 \mathop{\mathchoice{\raise-0.22em\hbox{\huge
$\times$}} {\raise-0.05em\hbox{\Large $\times$}}{\hbox{\large
$\times$}}{\times}}_{\alpha \in D}\mathbb{G}(W_{\alpha}^{\min }(%
\mathbf{v}),V_{\alpha}), \quad ((U_{\alpha})_{\alpha \in D},E^{(D)}) 
\mapsto (U_{\alpha})_{\alpha \in D}.
$$ 
In consequence, the second statement is proved.
\end{proof}

\begin{remark}
We point out that for $d=2$ the typical fibre is the Lie group $\mathrm{GL}(\mathbb{R}^r)$
for some $r \ge 1$ and for $\mathfrak{r}=\mathbf{1}$ (and any $d\ge 2$) the typical fibre is the
Lie group $\mathrm{GL}(\mathbb{R})=\mathbb{R}\setminus\{0\}.$ Then in both cases we have that 
the fibre bundle is a principal bundle, that is, a fibre bundle which has as a typical fibre
a Lie group.
\end{remark}

\begin{remark}
Assume that $(V_{\alpha},\|\cdot\|_{\alpha})$ is a complex Banach space 
for each $\alpha \in D$ and  $\|\cdot\|_D$ is a norm 
on the complex tensor space $\mathbf{V}_D = \left.
_{a}\bigotimes_{\alpha \in D}V_{\alpha}\right.$ such that the tensor
product map (\ref{bigotimes}) is continuous. From Remark \ref{analytic1} we have that
the collection $\{\mathcal{U}(\mathbf{v}),\xi_{\mathbf{v}}\}_{\mathbf{v}\in \mathfrak{M}_{\mathfrak{%
r}}(\mathbf{V}_{D})}$ is an analytic atlas for $\mathfrak{M}_{\mathfrak{r}}(%
\mathbf{V}_{D})$ and hence it is an analytic Banach manifold modelled on a Banach space
$$
\left( 
\mathop{\mathchoice{\raise-0.22em\hbox{\huge
$\times$}} {\raise-0.05em\hbox{\Large $\times$}}{\hbox{\large
$\times$}}{\times}}_{\alpha \in D}\mathcal{L}(U_{\alpha},W_{\alpha})\right) \times \mathbb{C}^{
\mathop{\mathchoice{\raise-0.22em\hbox{\huge $\times$}}
{\raise-0.05em\hbox{\Large $\times$}}{\hbox{\large
$\times$}}{\times}}_{\alpha \in D}r_{\alpha}},
$$
here $U_{\alpha} \in \mathbb{G}_{r_{\alpha}}(V_{\alpha})$ and $V_{\alpha} = U_{\alpha} \oplus W_{\alpha}$
for $\alpha \in D.$
\end{remark}

We point out that the norm $\|\cdot\|_D$ in $\mathbf{V}_D$ is only used in
the proof of Lemma \ref{premanifold} in order to endow 
the finite-dimensional tensor space $\left.
_{a}\bigotimes_{\alpha \in D}U_{\alpha}^{\min }(\mathbf{v})\right.$
with a structure of finite-dimensional Banach space 
for each $\mathbf{v} \in \mathfrak{M}_{\mathfrak{r}}(\mathbf{V}_D)$. Thus,
the geometric structure of manifold is independent of the
choice of the norm $\|\cdot\|_D$ over the tensor space $\mathbf{V}_D.$ We 
illustrate this assertion with the following example.

\begin{example}
\label{example_BM1} Let $V_{1_{\|\cdot\|_1}} := H^{1,p}(I_1)$ and $%
V_{2_{\|\cdot\|_2}}= H^{1,p}(I_2),$ with $\|\cdot\|_\alpha = \|\cdot\|_{1,p,I_\alpha}$, and $1 \le p < \infty.$ Take $\mathbf{V}_D:= H^{1,p}(I_1)
\otimes_a H^{1,p}(I_2).$ Now, we can consider as ambient Banach space either 
$$\overline{\mathbf{V}_{D}%
}^{\|\cdot\|_{D,1}} := H^{1,p}(I_1 \times I_2),
$$ with $\|\cdot\|_{D,1} = \|\cdot\|_{1,p}$,
or $$\overline{\mathbf{V}_{D}}%
^{\|\cdot\|_{D,2}} = H^{1,p}(I_1) \otimes_{\|\cdot\|_{D,2}}
H^{1,p}(I_2), $$ where $\|\cdot\|_{D,2} := \|\cdot\|_{(0,1),p}$ is the norm given by 
\begin{equation*}
\|f\|_{(0,1),p}:= \left(\|f\|_p^p + \left\|\frac{\partial f}{\partial x_2}%
\right\|_p^p\right)^{1/p}.
\end{equation*}
 The tensor product map (\ref{bigotimes}) is continuous for both norms (see Examples 4.41 and 4.42 in \cite{Hackbusch})
and hence from Theorem~\ref{Tucker_Banach_Manifold} we obtain that for each $r \ge 1$
the set $\mathfrak{M}_{(r,r)}(\mathbf{V}_D)$ is a $\mathcal{C}^{\infty}$-Banach manifold modelled on 
$$
\mathcal{L}(U_{1},W_{1}) \times \mathcal{L}(U_{2},W_{2})  \times \mathrm{GL}(\mathbb{R}^r),
$$
here $U_i \in \mathbb{G}_{r}(H^{1,p}(I_i))$ and $H^{1,p}(I_i)=U_i \oplus W_i$ for $i=1,2.$
\end{example}

The next result gives us the conditions to have a Hilbert manifold.

\begin{corollary}
\label{hilbert_manifold_tensor} Assume that
$(V_{\alpha},\|\cdot\|_{\alpha})$ is a normed space such that $V_{\alpha_{\|\cdot\|_{\alpha}}}$ is a Hilbert space 
for each $\alpha \in D$ 
and let $\|\cdot\|_D$ be a norm 
on the tensor space $\mathbf{V}_D = \left.
_{a}\bigotimes_{\alpha \in D}V_{\alpha}\right.$ such that the tensor
product map (\ref{bigotimes}) is continuous. Then $\mathfrak{M}_{\mathfrak{r}}(\mathbf{V}_{D})$ is a 
$\mathcal{C}^{\infty}$-Hilbert manifold modelled on a Hilbert space
$$
\mathop{\mathchoice{\raise-0.22em\hbox{\huge
$\times$}} {\raise-0.05em\hbox{\Large $\times$}}{\hbox{\large
$\times$}}{\times}}_{\alpha \in D} W_{\alpha}^{r_{\alpha}} \times \mathbb{R}^{
\mathop{\mathchoice{\raise-0.22em\hbox{\huge $\times$}}
{\raise-0.05em\hbox{\Large $\times$}}{\hbox{\large
$\times$}}{\times}}_{\alpha \in D}r_{\alpha}},
$$
here  $V_{{\alpha}_{\|\cdot\|_{\alpha}}} = U_{\alpha}\oplus W_{\alpha},$ 
for some $U_{\alpha} \in \mathbb{G}_{r_{\alpha}}(V_{\alpha})$
for $\alpha \in D.$ 
\end{corollary}

\begin{proof}
We can identify each $L_{\alpha}\in \mathcal{L}\left( U_{\alpha}^{\min }(\mathbf{v}%
),W_{\alpha}^{\min }(\mathbf{v})\right) $ with a set of vectors $(w_{s_{\alpha}}^{(\alpha
)})_{s_{\alpha}=1}^{s_{\alpha}=r_{\alpha}}\in W_{\alpha}^{\min }(\mathbf{v})^{r_{\alpha}},$
where $w_{s_{\alpha}}^{(\alpha
)}=L_{\alpha}(u_{s_{\alpha}}^{(\alpha)})$ and $%
U_{\alpha}^{\min }(\mathbf{v})=\mathrm{span}\,\{u_{1}^{(\alpha)},\ldots ,%
u_{r_{\alpha}}^{(\alpha)}\}$ for $\alpha \in D.$ Thus we can
identify each $\left(\mathfrak{L},C^{(D)}\right) \in \xi_{\mathbf{v}}(\mathcal{U}(%
\mathbf{v}))$ with a pair%
\begin{equation*}
\left(\mathcal{W},C^{(D)}\right) \in 
\mathop{\mathchoice{\raise-0.22em\hbox{\huge $\times$}} {\raise-0.05em\hbox{\Large $\times$}}{\hbox{\large
$\times$}}{\times}}_{\alpha \in D}W_{\alpha}^{\min }(\mathbf{v})^{r_{\alpha
}} \times \mathbb{R}_*^{
\mathop{\mathchoice{\raise-0.22em\hbox{\huge $\times$}}
{\raise-0.05em\hbox{\Large $\times$}}{\hbox{\large
$\times$}}{\times}}_{\alpha \in D}r_{\alpha}},
\end{equation*}%
where $\mathcal{W}:=((w_{i_{\alpha}}^{(\alpha )})_{i_{\alpha}=1}^{r_{\alpha
}})_{\alpha \in D}.$ Take $%
\mathop{\mathchoice{\raise-0.22em\hbox{\huge $\times$}}
{\raise-0.05em\hbox{\Large $\times$}}{\hbox{\large
$\times$}}{\times}}_{\alpha \in D}W_{\alpha}^{\min }(\mathbf{v})^{r_{\alpha
}} \times \mathbb{R}_*^{
\mathop{\mathchoice{\raise-0.22em\hbox{\huge $\times$}}
{\raise-0.05em\hbox{\Large $\times$}}{\hbox{\large
$\times$}}{\times}}_{\alpha \in D}r_{\alpha}}$ an open subset of the
Hilbert space $%
\mathop{\mathchoice{\raise-0.22em\hbox{\huge $\times$}}
{\raise-0.05em\hbox{\Large $\times$}}{\hbox{\large
$\times$}}{\times}}_{\alpha \in D}W_{\alpha }^{\min }(\mathbf{v}%
)^{r_{\alpha}} \times \mathbb{R}^{
\mathop{\mathchoice{\raise-0.22em\hbox{\huge $\times$}}
{\raise-0.05em\hbox{\Large $\times$}}{\hbox{\large
$\times$}}{\times}}_{\alpha \in D}r_{\alpha}}$ endowed with the inner product norm
\begin{equation*}
\Vert \left(\mathcal{W},C^{(D)}\right) \Vert _{\times,\mathbf{v}}^2:=\Vert C^{(D)}\Vert _{F}^2+\sum_{\alpha
\in D}\sum_{i_{\alpha}=1}^{r_{\alpha}}\Vert w_{i_{\alpha}}^{(\alpha
)}\Vert _{\alpha}^2,
\end{equation*}%
with $\Vert \cdot \Vert _{F}$ the Frobenius norm. 
It allows us to define local charts, also denoted by $\xi_{\mathbf{v}},$ by 
\begin{equation*}
\xi_{\mathbf{v}}^{-1}: 
\mathop{\mathchoice{\raise-0.22em\hbox{\huge $\times$}} {\raise-0.05em\hbox{\Large $\times$}}{\hbox{\large
$\times$}}{\times}}_{\alpha \in D}W_{\alpha }^{\min }(\mathbf{v}%
)^{r_{\alpha}} \times \mathbb{R}_*^{
\mathop{\mathchoice{\raise-0.22em\hbox{\huge $\times$}}
{\raise-0.05em\hbox{\Large $\times$}}{\hbox{\large
$\times$}}{\times}}_{\alpha \in D}r_{\alpha}}\longrightarrow \mathcal{U%
}(\mathbf{v}),
\end{equation*}%
where $\xi_{\mathbf{v}}^{-1}\left( \mathcal{W},C^{(D)}\right) =\mathbf{w}$, putting $L_{\alpha}(%
u_{i_{\alpha}}^{(\alpha )})=w_{i_{\alpha}}^{(\alpha )},$ $1\leq i_{\alpha}\leq
r_{\alpha} $ and $\alpha \in D.$ Since each local chart is defined
over an open subset of the Hilbert space $%
\mathop{\mathchoice{\raise-0.22em\hbox{\huge $\times$}}
{\raise-0.05em\hbox{\Large $\times$}}{\hbox{\large
$\times$}}{\times}}_{\alpha \in D}W_{\alpha}^{\min }(\mathbf{v})^{r_{\alpha
}} \times \mathbb{R}^{
\mathop{\mathchoice{\raise-0.22em\hbox{\huge $\times$}}
{\raise-0.05em\hbox{\Large $\times$}}{\hbox{\large
$\times$}}{\times}}_{\alpha \in D}r_{\alpha}},$ the corollary follows.
\end{proof}

\bigskip

Using the definition of the local charts for the manifold $\mathfrak{M}%
_{\mathfrak{r}}(\mathbf{V}_{D}),$ we can identify its tangent space at $%
\mathbf{v}$ with $\mathbb{T}_{\mathbf{v}}(\mathfrak{M}_{\mathfrak{r}}(%
\mathbf{V}_{D})):=%
\mathop{\mathchoice{\raise-0.22em\hbox{\huge $\times$}}
{\raise-0.05em\hbox{\Large $\times$}}{\hbox{\large $\times$}}{\times}}_{\alpha\in 
D}\mathcal{L}(U_{\alpha}^{\min }(\mathbf{v}),W_{\alpha}^{\min }(%
\mathbf{v}))\times \mathbb{R}^{
\mathop{\mathchoice{\raise-0.22em\hbox{\huge $\times$}}
{\raise-0.05em\hbox{\Large $\times$}}{\hbox{\large
$\times$}}{\times}}_{\alpha \in D}r_{\alpha}}.$ We will consider $\mathbb{T}_{%
\mathbf{v}}(\mathfrak{M}_{\mathfrak{r}}(\mathbf{V}_{D}))$ endowed with the
product norm 
\begin{equation*}
||| (\mathfrak{L},C^{(D)})|||_{\mathbf{v}} := \Vert C^{(D )}\Vert _{F}+\sum_{\alpha \in 
D}\Vert L_{\alpha}\Vert _{W_{\alpha}^{\min }(\mathbf{v})\leftarrow
U_{\alpha}^{\min }(\mathbf{v})}.
\end{equation*}%

Finally, the fact that $\mathbf{V}_D = \bigcup_{
\substack{\mathfrak{r} \in \mathcal{AD}(\mathbf{V}_D)}
} 
\mathfrak{M}_{\mathfrak{r}}(\mathbf{V}_D)$ 
allows us to state the following.

\begin{corollary}
\label{Bounded_Banach_Manifold} Assume that
$(V_{\alpha},\|\cdot\|_{\alpha})$ is a normed space for each $\alpha \in D$ 
and that $\|\cdot\|_D$ is a norm 
on the tensor space $\mathbf{V}_D = \left.
_{a}\bigotimes_{\alpha \in D}V_{\alpha}\right.$ such that the tensor
product map (\ref{bigotimes}) is continuous. 
Then the algebraic tensor space $\mathbf{V}_{D}$ is a $\mathcal{C}^{\infty}$-Banach 
manifold not modelled on a particular Banach space.
\end{corollary}

\section{The manifold of tensors in Tucker format with fixed rank 
and its natural ambient tensor Banach space}
\label{embedded_manifold}

Consider the tensor space $\mathbf{V}_{D}=\left.
_{a}\bigotimes_{\alpha \in D}V_{\alpha }\right. $ and assume 
that for each $ \alpha \in D$ the vector
space $V_{\alpha}$ is a normed space with a norm $\Vert \cdot \Vert _{\alpha}.$ We
start with a brief discussion about the choice of the ambient manifold for $%
\mathfrak{M}_{\mathfrak{r}}(\mathbf{V}_{D}).$ Recall that in Example \ref{example_BM1} we
have two norms $\Vert \cdot \Vert _{D,1}$ and $\Vert \cdot \Vert
_{D,2}$ on $\mathbf{V}_{D}$ such that the tensor
product map (\ref{bigotimes}) is continuous for both norms. Then we have two natural embeddings 
$\mathbf{V}_{D}\subset \overline{%
\mathbf{V}_{D}}^{\Vert \cdot \Vert _{D,1}}$ and $\mathbf{V}_{D}\subset 
\overline{\mathbf{V}_{D}}^{\Vert \cdot \Vert _{D,2}}.$ 
In this context a natural question about the choice of a norm $\|\cdot\|_{D}$
for the algebraic tensor space $\mathbf{V}_{D}$ appears: What is the good choice 
for this norm to show that $\mathfrak{M}_{%
\mathfrak{r}}(\mathbf{V}_D)$ is an immersed submanifold? 

\bigskip

More precisely, assume that $(V_{\alpha},\|\cdot\|_{\alpha})$ is a normed space 
for each $\alpha \in D$ 
and let $\|\cdot\|_D$ be a norm 
on the tensor space $\mathbf{V}_D = \left.
_{a}\bigotimes_{\alpha \in D}V_{\alpha}\right.$ such that the tensor
product map (\ref{bigotimes}) is continuous. Then
we have a natural ambient space for $\mathfrak{M}_{\mathfrak{r}}(\mathbf{V}%
_{D})$ given by a Banach tensor space $\overline{\mathbf{V}_D}^{\|\cdot\|_D}
= \mathbf{V}_{D_{\|\cdot\|_D}}.$ Since the natural inclusion 
\begin{equation*}
\mathfrak{i}:\mathfrak{M}_{\mathfrak{r}}(\mathbf{V}_{D})\longrightarrow 
\mathbf{V}_{D_{\Vert \cdot \Vert _{D}}},
\end{equation*}
given by $\mathfrak{i}(\mathbf{v})=\mathbf{v},$ is an injective map we will
study $\mathfrak{i}$ as a function between Banach manifolds. To this end we
recall the definition of an immersion between manifolds.

\begin{definition}
Let $F:X\rightarrow Y$ be a morphism between Banach manifolds and let $x\in
X.$ We shall say that $F$ is an \emph{immersion at $x$} if there exists an
open neighbourhood $X_{x}$ of $x$ in $X$ such that the restriction of $F$ to 
$X_{x}$ induces an isomorphism from $X_{x}$ onto a submanifold of $Y.$ We
say that $F$ is an \emph{immersion} if it is an immersion at each point of $%
X.$
\end{definition}

Our next step is to recall the definition of the differential as a morphism
which gives a linear map between the tangent spaces of the manifolds
involved with the morphism.

\begin{definition}
Let $X$ and $Y$ be two Banach manifolds. Let $F:X\rightarrow Y$ be a $%
\mathcal{C}^{r}$ morphism, i.e., 
\begin{equation*}
\psi \circ F\circ \varphi ^{-1}:\varphi (U)\rightarrow \psi (W)
\end{equation*}%
is a $\mathcal{C}^{r}$-Fr\'{e}chet differentiable map, where $(U,\varphi )$
is a chart in $X$ at $x$ and $(W,\psi )$ is a chart in $Y$ at $F(x)$. For $%
x\in X,$ we define 
\begin{equation*}
\mathrm{T}_{x}F:\mathbb{T}_{x}(X)\longrightarrow \mathbb{T}_{F(x)}(Y),\quad
v\mapsto \lbrack (\psi \circ F\circ \varphi ^{-1})^{\prime }(\varphi (x))]v.
\end{equation*}
\end{definition}

For Banach manifolds we have the following criterion for immersions (see
Theorem 3.5.7 in \cite{MRA}).

\begin{proposition}
\label{prop_inmersion} Let $X,Y$ be Banach manifolds of class $\mathcal{C}%
^{p}$ $(p\geq 1).$ Let $F:X\rightarrow Y$ be a $\mathcal{C}^{p}$ morphism
and $x\in X.$ Then $F$ is an immersion at $x$ if and only if $\mathrm{T}%
_{x}F $ is injective and $\mathrm{T}_{x}F(\mathbb{T}_{x}(X)) \in \mathbb{G}(%
\mathbb{T}_{F(x)}(Y)).$
\end{proposition}

A concept related to an immersion between Banach manifolds is introduced in the
following definition.

\begin{definition}
Assume that $X$ and $Y$ are Banach manifolds and let $f:X\longrightarrow Y$
be a $\mathcal{C}^{r}$ morphism. If $f$ is an injective immersion, then $%
f(X) $ is called an \emph{immersed submanifold of $Y$.}
\end{definition}

In consequence, to prove that the standard inclusion map $\mathfrak{i}$ is
an immersion we shall prove, under the appropriate conditions, that if $%
\mathfrak{i}$ is a differentiable morphism then for each $\mathbf{v} \in 
\mathfrak{M}_{\mathfrak{r}}(\mathbf{V}_{D})$ the linear map $\mathrm{T}_{%
\mathbf{v}}\mathfrak{i}$ is injective and $\mathrm{T}_{\mathbf{v}}\mathfrak{i%
}(\mathbb{T}_{\mathbf{v}}(\mathfrak{M}_{\mathfrak{r}}(\mathbf{V}_{D})))$
belongs to $\mathbb{G}(\mathbf{V}_{D_{\|\cdot\|_D}}).$

\subsection{The linear map $\mathrm{T}_{%
\mathbf{v}}\mathfrak{i}$ is injective}

To describe $\mathfrak{i}$ as a morphism, we proceed as follows. Given $%
\mathbf{v} = \sum_{\substack{ 1\leq i_{\alpha }\leq r_{\alpha }  \\ \alpha \in
D}}C_{(i_{\alpha })_{\alpha \in D}}^{(D)}\bigotimes_{\alpha \in D}
u_{i_{\alpha}}^{(\alpha)}\in \mathfrak{M}_{\mathfrak{r}}(\mathbf{V}_{D}),$ we consider $%
\mathcal{U}(\mathbf{v}),$ a neighbourhood of $\mathbf{v},$ and
\begin{equation*}
(\mathfrak{i}\circ \xi_{\mathbf{v}}^{-1}):%
\mathop{\mathchoice{\raise-0.22em\hbox{\huge $\times$}} {\raise-0.05em\hbox{\Large $\times$}}{\hbox{\large
$\times$}}{\times}}_{\alpha \in D}\mathcal{L}(U_{\alpha
}^{\min }(\mathbf{v}),W_{\alpha }^{\min }(\mathbf{v}))\times \mathbb{R}%
_{\ast }^{\mathfrak{r}}\rightarrow \mathbf{V}_{\Vert \cdot \Vert _{D}}.
\end{equation*}%
From the proof of Lemma \ref{premanifold} the map $(\mathfrak{i}\circ \xi_{\mathbf{v}}^{-1})$ 
is given by 
\begin{equation*}
(\mathfrak{i}\circ \xi_{\mathbf{v}}^{-1})\left( \mathfrak{L},E^{(D)}\right) = 
\mathrm{eval}\left(\bigotimes_{\alpha \in D}(id_{\alpha}+L_{\alpha}), \mathbf{u}(E^{(D)})\right)=
\sum_{\substack{ 1\leq i_{\alpha }\leq r_{\alpha }  \\ \alpha \in
D}}E_{(i_{\alpha })_{\alpha \in D}}^{(D)}\bigotimes_{\alpha \in D}(id_{\alpha}+L_{\alpha})(
u_{i_{\alpha}}^{(\alpha)}).
\end{equation*}%

\begin{remark}
Observe that it allows us to define a left local action of the Lie group $\mathop{\mathchoice{\raise-0.22em\hbox{\huge $\times$}} {\raise-0.05em\hbox{\Large $\times$}}{\hbox{\large
$\times$}}{\times}}_{\alpha \in D}\mathrm{GL}(\mathcal{L}(U_{\alpha
}^{\min }(\mathbf{v}),W_{\alpha }^{\min }(\mathbf{v})))$ onto the local manifold
$\mathcal{U}(\mathbf{v})$ as follows:
$$
\mathop{\mathchoice{\raise-0.22em\hbox{\huge $\times$}} {\raise-0.05em\hbox{\Large $\times$}}{\hbox{\large
$\times$}}{\times}}_{\alpha \in D}\mathrm{GL}(\mathcal{L}(U_{\alpha
}^{\min }(\mathbf{v}),W_{\alpha }^{\min }(\mathbf{v}))) \times \mathcal{U}(\mathbf{v})
\rightarrow \mathcal{U}(\mathbf{v}), \quad \left( (\exp(L_{\alpha}))_{\alpha \in D}, \mathbf{w} 
\right) \mapsto \bigotimes_{\alpha \in D}\exp(L_{\alpha})(\mathbf{w}).
$$
Moreover, we can also define a right local action using the Lie group
$
\mathop{\mathchoice{\raise-0.22em\hbox{\huge $\times$}} {\raise-0.05em\hbox{\Large $\times$}}{\hbox{\large$\times$}}{\times}}_{\alpha \in D}\mathrm{GL}(U_{\alpha
}^{\min }(\mathbf{v}))
$
by  
$$
\mathcal{U}(\mathbf{v}) \times \left( \mathop{\mathchoice{\raise-0.22em\hbox{\huge $\times$}} {\raise-0.05em\hbox{\Large $\times$}}{\hbox{\large$\times$}}{\times}}_{\alpha \in D}\mathrm{GL}(U_{\alpha
}^{\min }(\mathbf{v}))  \right) \rightarrow  \mathcal{U}(\mathbf{v}),
\quad \left(\bigotimes_{\alpha \in D}\exp(L_{\alpha})(\mathbf{u}(E^{(D)})), (G_{\alpha})_{\alpha \in D}
\right) \mapsto \bigotimes_{\alpha \in D}(\exp(L_{\alpha}) \circ G_{\alpha})(\mathbf{u}(E^{(D)})).
$$
\end{remark}

\bigskip

The next lemma describes the tangent map $\mathrm{T}_{\mathbf{v}}\mathfrak{i}.$

\begin{proposition}
\label{characterization_tangent_map} Assume that
$(V_{\alpha},\|\cdot\|_{\alpha})$ is a normed space 
for each $\alpha \in D$ 
and let $\|\cdot\|_D$ be a norm 
on the tensor space $\mathbf{V}_D = \left.
_{a}\bigotimes_{\alpha \in D}V_{\alpha}\right.$ such that the tensor
product map (\ref{bigotimes}) is continuous. For $\mathbf{v} = \sum_{\substack{1 \le i_{\alpha} \le r_{\alpha} \\ \alpha \in D}} C^{(D)}_{(i_{\alpha})_{\alpha \in D}} \bigotimes_{\alpha \in D} u^{(\alpha)}_{i_{\alpha}}  \in \mathfrak{M}_{\mathfrak{r%
}}(\mathbf{V}_{D})$ the following statements hold.

\begin{itemize}
\item[(a)] The map $(\mathfrak{i} \circ \xi_{\mathbf{v}}^{-1})$ from $%
\mathop{\mathchoice{\raise-0.22em\hbox{\huge $\times$}} {\raise-0.05em\hbox{\Large $\times$}}{\hbox{\large
$\times$}}{\times}}_{\alpha \in D}\mathcal{L}(U_{\alpha
}^{\min }(\mathbf{v}),W_{\alpha }^{\min }(\mathbf{v})) \times \mathbb{R}^{
\mathop{\mathchoice{\raise-0.22em\hbox{\huge $\times$}}
{\raise-0.05em\hbox{\Large $\times$}}{\hbox{\large
$\times$}}{\times}}_{\alpha \in D}r_{\alpha}}$ to $\mathbf{V}_{D_{\|\cdot\|_D}}$ is Fr\'echet
differentiable, and hence 
\begin{equation*}
\mathrm{T}_{\mathbf{v}}\mathfrak{i} \in \mathcal{L}\left(\mathbb{T}_{\mathbf{%
v}}(\mathfrak{M}_{\mathfrak{r}}(\mathbf{V}_{D})), \mathbf{V}%
_{D_{\|\cdot\|_D}} \right).
\end{equation*}

\item[(b)] Assume $(\dot{\mathfrak{L}},\dot{C}^{(D)})\in \mathbb{T}_{%
\mathbf{v}}(\mathfrak{M}_{\mathfrak{r}}(\mathbf{V}_{D})),$ where $\dot{%
C}^{(D)}\in \mathbb{R}^{
\mathop{\mathchoice{\raise-0.22em\hbox{\huge $\times$}}
{\raise-0.05em\hbox{\Large $\times$}}{\hbox{\large
$\times$}}{\times}}_{\alpha \in D}r_{\alpha}}$ and $\dot{\mathfrak{L}}=(\dot{L}%
_{\alpha })_{\alpha \in D}$ is in $ 
\mathop{\mathchoice{\raise-0.22em\hbox{\huge $\times$}} {\raise-0.05em\hbox{\Large $\times$}}{\hbox{\large
$\times$}}{\times}}_{\alpha \in D}\mathcal{L}(U_{\alpha
}^{\min }(\mathbf{v}),W_{\alpha }^{\min }(\mathbf{v})).$ Then $\dot{\mathbf{w%
}}=\mathrm{T}_{\mathbf{v}}\mathfrak{i}(\dot{\mathfrak{L}},\dot{C}^{(D)}%
) $ if and only if
\begin{equation}
\dot{\mathbf{w}}=\sum_{\substack{ 1\leq i_{\alpha }\leq r_{\alpha }  \\ %
\alpha \in D}}\dot{C}_{(i_{\alpha })_{\alpha \in
D}}^{(D)}\bigotimes_{\alpha \in D}u_{i_{\alpha }}^{(\alpha
)}+\sum_{\substack{1 \le i_{\alpha} \le r_{\alpha} \\ \alpha \in D}}
\left( \dot{u}_{i_{\alpha }}^{(\alpha )}\otimes \mathbf{U}%
_{i_{\alpha }}^{(\alpha )}\right) ,  \label{kinematic1}
\end{equation}%
where 
$$
\mathbf{U}_{i_{\alpha }}^{(\alpha )}=\sum_{\substack{ 1\leq i_{\beta }\leq
r_{\beta }  \\ \beta \in D \setminus \{\alpha\}}}C_{i_{\alpha},(i_{\beta })_{\beta
\in D\setminus \{\alpha\}}}^{(D)}\bigotimes_{\beta \in D}u_{i_{\beta }}^{(\beta
)}.
$$
\end{itemize}
\end{proposition}

\begin{proof}
To prove statement (a), from the results of section \ref{manifold_structure} we know that 
$(\mathfrak{i}\circ \xi_{\mathbf{v}}^{-1})$ is $\mathcal{C}^{\infty}$-Fr\'{e}chet differentiable and 
that $\xi_{\mathbf{v}}(\mathbf{v%
}) = (\mathfrak{0},C^{(D)}).$ Now, to prove (b) observe that  
\begin{equation*}
\mathrm{T}_{\mathbf{v}}\mathfrak{i}:%
\mathop{\mathchoice{\raise-0.22em\hbox{\huge $\times$}} {\raise-0.05em\hbox{\Large $\times$}}{\hbox{\large
$\times$}}{\times}}_{\alpha \in D}\mathcal{L}(U_{\alpha
}^{\min }(\mathbf{v}),W_{\alpha }^{\min }(\mathbf{v})) \times \mathbb{R}^{
\mathop{\mathchoice{\raise-0.22em\hbox{\huge $\times$}}
{\raise-0.05em\hbox{\Large $\times$}}{\hbox{\large
$\times$}}{\times}}_{\alpha \in D}r_{\alpha}}\longrightarrow \mathbf{V}_{\Vert \cdot \Vert _{D}}
\end{equation*}
is given by the chain rule:
\begin{align*}
\mathrm{T}_{\mathbf{v}}\mathfrak{i}(\dot{\mathfrak{L}},\dot{C}^{(D)})& =[(%
\mathfrak{i}\circ \xi_{\mathbf{v}}^{-1})^{\prime }((\mathfrak{i}\circ \xi_{\mathbf{v}}^{-1})(\mathbf{v}))](\dot{\mathfrak{L}},\dot{C}^{(D)}) \\
& = [(%
\mathfrak{i}\circ \xi_{\mathbf{v}}^{-1})^{\prime }(\mathfrak{0},C^{(D)})](\dot{\mathfrak{L}},\dot{%
C}^{(D)}) \\
& =  \sum_{\substack{ 1\leq i_{\alpha }\leq r_{\alpha }  \\ \alpha \in
D}}\dot{C}_{(i_{\alpha })_{\alpha \in D}}^{(D)}\bigotimes_{\alpha \in D}
u_{i_{\alpha}}^{(\alpha)} 
+ 
\sum_{\substack{1 \le i_{\alpha} \le r_{\alpha} \\ \alpha \in D}}\sum_{\substack{ \substack{1\leq i_{\beta }\leq
r_{\beta } \\ \beta \in D \setminus \{\alpha\}}}}C^{(D)}_{i_{\alpha},(i_{\beta})_{\beta \in D \setminus \{\alpha\}}}
\left( \dot{L}_{\alpha}(u_{i_{\alpha}}^{\alpha}) \otimes 
\bigotimes_{\substack{\beta \in D\\ \beta \neq \alpha }}u_{i_{\beta }}^{(\beta
)}
\right) \\
& = \sum_{\substack{ 1\leq i_{\alpha }\leq r_{\alpha }  \\ \alpha \in
D}}\dot{C}_{(i_{\alpha })_{\alpha \in D}}^{(D)}\bigotimes_{\alpha \in D}
u_{i_{\alpha}}^{(\alpha)} + \sum_{\substack{1 \le i_{\alpha} \le r_{\alpha} \\ \alpha \in D}}
\left( \dot{L}_{\alpha}(u_{i_{\alpha}}^{\alpha}) \otimes \mathbf{U}%
_{i_{\alpha }}^{(\alpha )}\right).
\end{align*}%
This implies statement (b).
\end{proof}

\bigskip

In the next proposition we prove that $\mathrm{T}_{\mathbf{v}}\mathfrak{i}$ is
injective when we consider $\mathbf{v}$ in the manifold $\mathfrak{M}_{%
\mathbf{r}}(\mathbf{V}_{D}).$ It allows us to characterise the tangent space
of $\mathfrak{M}_{%
\mathbf{r}}(\mathbf{V}_{D})$ inside the tensor space $\mathbf{V}_{D_{\Vert \cdot \Vert
_{D}}}.$  We recall that from Remark \ref{identification} we have 
\begin{equation*}
U_{D\setminus \{\alpha \}}^{\min }(\mathbf{v})=%
\mathrm{span}\,\{\mathbf{U}_{i_{\alpha }}^{(\alpha )}:1\leq i_{\alpha }\leq
r_{\alpha }\},
\end{equation*}
for $\alpha \in D.$ In order to simplify notations we introduce the following definition.
For each $\mathbf{v} \in \mathfrak{M}_{\mathfrak{r}}(\mathbf{V}_{D})$ we 
denote by $\mathbf{Z}^{(D)}(\mathbf{v})$ the linear
subspace in $\mathbf{V}_{D_{\|\cdot\|_D}}$ defined by
\begin{equation*}
\mathbf{Z}^{(D)}(\mathbf{v}):=\left. _{a}\bigotimes_{\alpha \in
D}U_{\alpha }^{\min }(\mathbf{v})\right. \oplus \left( \bigoplus_{\alpha
\in D}W_{\alpha }^{\min }(\mathbf{v})\otimes _{a}U_{D\setminus
\{\alpha \}}^{\min }(\mathbf{v})\right) .
\end{equation*}%

\begin{proposition}
\label{rank_one_tangent_space} Assume that
$(V_{\alpha},\|\cdot\|_{\alpha})$ is a normed space 
for each $\alpha \in D$ 
and let $\|\cdot\|_D$ be a norm 
on the tensor space $\mathbf{V}_D = \left.
_{a}\bigotimes_{\alpha \in D}V_{\alpha}\right.$ such that the tensor
product map (\ref{bigotimes}) is continuous. Let $\mathbf{v}\in 
\mathfrak{M}_{\mathfrak{r}}(\mathbf{V}_{D}),$ then the linear map $\mathrm{T}_{%
\mathbf{v}}\mathfrak{i}$ is injective and 
$
\mathrm{T}_{\mathbf{v}}\mathfrak{i}(\mathbb{T}_{\mathbf{v}}(\mathfrak{M}_{%
\mathbf{r}}(\mathbf{V}_{D}))) = \mathbf{Z}^{(D)}(\mathbf{v})
$
is linearly isomorphic to $\mathbb{T}_{\mathbf{v}}(\mathfrak{M}_{\mathfrak{r}}(%
\mathbf{V}_{D})).$
\end{proposition}

\begin{proof}
First, observe that if $\mathbf{v}\in \mathfrak{M}_{\mathfrak{r}}(\mathbf{V}%
_{D})$ and $\dot{\mathbf{w}}=\mathrm{T}_{\mathbf{v}}\mathfrak{i}(\dot{\mathfrak{L}},\dot{%
C}^{(D)}),$ then by Proposition~\ref%
{characterization_tangent_map}(b)%
\begin{equation*}
\dot{\mathbf{w}}=\sum_{\substack{ 1\leq i_{\alpha }\leq r_{\alpha }  \\ %
\alpha \in D}}\dot{C}_{(i_{\alpha })_{\alpha \in
D}}^{(D)}\bigotimes_{\alpha \in D}u_{i_{\alpha }}^{(\alpha
)}+\sum_{\substack{1 \le i_{\alpha} \le r_{\alpha} \\ \alpha \in D}}
\left( \dot{u}_{i_{\alpha }}^{(\alpha )}\otimes \mathbf{U}%
_{i_{\alpha }}^{(\alpha )}\right) ,
\end{equation*}%
where 
\begin{equation*}
\mathbf{U}_{i_{\alpha }}^{(\alpha )}=\sum_{\substack{ 1\leq i_{\beta }\leq
r_{\beta }  \\ \beta \in D\setminus \{\alpha\}}}C_{(i_{\beta })_{\beta
\in D}}^{(D)}\bigotimes_{\beta \in D}u_{i_{\beta }}^{(\beta
)}\in U_{D\setminus \{\alpha \}}^{\min }(\mathbf{v}),
\end{equation*}%
and $\dot{u}_{i_{\alpha }}^{(\alpha )}=\dot{L}_{\alpha}(u%
_{i_{\alpha }}^{(\alpha )})\in W_{\alpha }^{\min }(\mathbf{v})$ for all $%
\alpha \in D.$ Hence $\mathrm{T}_{\mathbf{v}}\mathfrak{i}(%
\mathbb{T}_{\mathbf{v}}(\mathfrak{M}_{\mathfrak{r}}(\mathbf{V}_{D})))\subset 
\mathbf{Z}^{(D)}(\mathbf{v}).$ 
Next, we claim that $\mathbf{Z}^{(D)}(\mathbf{v})\subset \mathrm{T}_{\mathbf{%
v}}\mathfrak{i}(\mathbb{T}_{\mathbf{v}}(\mathfrak{M}_{\mathfrak{r}}(\mathbf{V}%
_{D}))).$ To prove the claim take $\mathbf{w}\in \mathbf{Z}^{(D)}(\mathbf{v}%
).$ Then we can write 
\begin{equation*}
\mathbf{w}=\sum_{\substack{ 1\leq i_{\alpha }\leq r_{\alpha }  \\ \alpha \in
D}}(\dot{C}^{(D)})_{(i_{\alpha })_{\alpha \in D}}\bigotimes_{\alpha
\in D}u_{i_{\alpha }}^{(\alpha )}+\sum_{\substack{1 \le i_{\alpha} \le r_{\alpha} \\ \alpha \in D}}\left( w_{i_{\alpha
}}^{(\alpha )}\otimes \mathbf{U}_{i_{\alpha }}^{(\alpha )}\right) ,
\end{equation*}%
where $w_{i_{\alpha }}^{(\alpha )} \in W_{\alpha }^{\min }(\mathbf{v})$
for $1\leq i_{\alpha }\leq r_{\alpha }$ and $\alpha \in D.$ 
Now, define $\dot{L}_{\alpha }\in \mathcal{L}(U_{\alpha }^{\min }(\mathbf{v}%
),W_{\alpha }^{\min }(\mathbf{v}))$ by $\dot{L}_{\alpha }(u%
_{i_{\alpha }}^{(\alpha )}):=w_{i_{\alpha }}^{(\alpha )}$ for $%
1\leq i_{\alpha }\leq r_{\alpha }$ and $\alpha \in D.$ Then the claim
follows from $\mathbf{w}=\mathrm{T}_{\mathbf{v}}\mathfrak{i}((\dot{L}%
_{\alpha })_{\alpha \in D},\dot{C}^{(D)}).$ To conclude the proof of the
proposition we need to show that the map $\mathrm{T}_{\mathbf{v}}\mathfrak{i}
$ is an injective linear operator. To prove this consider that 
\begin{equation*}
\mathrm{T}_{\mathbf{v}}\mathfrak{i}\left( (\dot{L}_{\beta })_{\beta \in 
D},\dot{C}^{(D)}\right) =\mathbf{0},
\end{equation*}%
that is, 
\begin{equation*}
\mathbf{0}=\sum_{\substack{ 1\leq i_{\alpha }\leq r_{\alpha }  \\ \alpha \in
D}}(\dot{C}^{(D)})_{(i_{\alpha })_{\alpha \in D}}\bigotimes_{\alpha
\in D}u_{i_{\alpha }}^{(\alpha )}+\sum_{\substack{1 \le i_{\alpha} \le r_{\alpha} \\ \alpha \in D}}\sum 
_{\substack{ 1\leq i_{\alpha }\leq r_{\alpha }}}\left( \dot{u}%
_{i_{\alpha }}^{(\alpha )}\otimes \mathbf{U}_{i_{\alpha }}^{(\alpha
)}\right)
\end{equation*}%
with $\dot{u}_{i_{\alpha }}^{(\alpha )}=\dot{L}_{\alpha}(u%
_{i_{\alpha }}^{(\alpha )}).$
Thus, 
\begin{align*}
\sum_{\substack{ 1\leq i_{\alpha }\leq r_{\alpha }  \\ \alpha \in D}}(%
\dot{C}^{(D)})_{(i_{\alpha })_{\alpha \in D}}\bigotimes_{\alpha \in D}%
u_{i_{\alpha }}^{(\alpha )}& =\mathbf{0}, \\
\sum_{\substack{ 1\leq i_{\alpha }\leq r_{\alpha }}}\left( \dot{u}%
_{i_{\alpha }}^{(\alpha )}\otimes \mathbf{U}_{i_{\alpha }}^{(\alpha
)}\right) & =\mathbf{0}\text{ for }\alpha \in D,
\end{align*}%
and hence $\dot{C}^{(D)}=\mathfrak{0},$ because $\left\{ \bigotimes_{\alpha
\in D}u_{i_{\alpha }}^{(\alpha )}\right\} $ is a basis of $%
\left. _{a}\bigotimes_{\alpha \in D}U_{\alpha }^{\min }(\mathbf{v}%
)\right. ,$ and $\dot{u}_{i_{\alpha }}^{(\alpha )}=0$ for $1\leq
i_{\alpha }\leq r_{\alpha },$ because the $\{\mathbf{U}_{i_{\alpha
}}^{(\alpha )}:1\leq i_{\alpha }\leq r_{\alpha }\}$ are linearly independent
for $\alpha \in D.$ Then $\dot{L}_{\alpha }=0$ for all $\alpha \in D$.
We conclude that

\begin{equation*}
\left( (\dot{L}_{\beta })_{\beta \in D},\dot{C}%
^{(D)}\right) =((0)_{\beta \in D},\mathfrak{0})
\end{equation*}%
and, in consequence, $\mathrm{T}_{\mathbf{v}}\mathfrak{i}$ is injective.
\end{proof}

\subsection{The linear subspace $\mathrm{T}_{\mathbf{v}}\mathfrak{i%
}(\mathbb{T}_{\mathbf{v}}(\mathfrak{M}_{\mathfrak{r}}(\mathbf{V}_{D})))$
belongs to $\mathbb{G}(\mathbf{V}_{D_{\|\cdot\|_D}})$}

Finally, to show that $\mathfrak{i}$ is an immersion, and hence $\mathfrak{M}%
_{\mathfrak{r}}(\mathbf{V}_{D})$ is an immersed submanifold of $\mathbf{V}%
_{D_{\Vert \cdot \Vert _{D}}},$ by Proposition~\ref{rank_one_tangent_space}, 
we need to prove that $\mathbf{Z}^{(D)}(\mathbf{v}) 
\in \mathbb{G}(\mathbf{V}_{\Vert \cdot \Vert _{D}}).$
To reach it we need a slightly stronger condition than the continuity of 
the tensor product map. To this end we introduce the crossnorms.

\subsubsection{Crossnorms}

Let $\left\Vert \cdot\right\Vert _{\alpha},$ $\alpha \in D,$ be the norms of the
vector spaces $V_{\alpha}$ appearing in $\mathbf{V}_D=\left. _{a}\bigotimes
\nolimits_{\alpha \in D}V_{\alpha}\right. .$ By $\left\Vert \cdot\right\Vert $ we
denote the norm on the tensor space $\mathbf{V}_D$. Note that $%
\left\Vert\cdot\right\Vert $ is not determined by $\left\Vert
\cdot\right\Vert _{\alpha},$ for $\alpha \in D,$ but there are relations which are
`reasonable'. Any norm $\left\Vert \cdot\right\Vert $ on $\left.
_{a}\bigotimes_{\alpha \in D}V_{\alpha}\right. $ satisfying%
\begin{equation}
\Big{\|}%
\bigotimes\nolimits_{\alpha \in D}v_{\alpha}%
\Big{\|}%
=\prod\nolimits_{\alpha \in D}\Vert v_{\alpha}\Vert_{\alpha}\qquad\text{for all }v_{\alpha}\in V_{\alpha}%
\text{ }\left( \alpha \in D\right)  \label{(rcn a}
\end{equation}
is called a \emph{crossnorm}. As usual, the dual norm of $\left\Vert
\cdot\right\Vert $ is denoted by $\left\Vert \cdot\right\Vert ^{\ast}$. If $%
\left\Vert \cdot\right\Vert $ is a crossnorm and also $\left\Vert
\cdot\right\Vert ^{\ast}$ is a crossnorm on $\left.
_{a}\bigotimes_{\alpha \in D}V_{\alpha}^{\ast}\right. $, i.e.,%
\begin{equation}
\Big{\|}%
\bigotimes\nolimits_{\alpha \in D}\varphi^{(\alpha)}%
\Big{\|}%
^{\ast}=\prod\nolimits_{\alpha \in D}\Vert\varphi^{(\alpha)}\Vert_{\alpha}^{\ast}\qquad%
\text{for all }\varphi^{(\alpha)}\in V_{\alpha}^{\ast}\text{ }\left( \alpha \in D\right) ,  \label{(rcn b}
\end{equation}
then $\left\Vert \cdot\right\Vert $ is called a \emph{reasonable crossnorm}.

\begin{remark}
\label{tensor product continuity}Eq. \eqref{(rcn a} implies the inequality $%
\Vert\bigotimes\nolimits_{\alpha \in D}v_{\alpha}\Vert\lesssim\prod\nolimits_{\alpha \in D}%
\Vert v_{\alpha}\Vert_{\alpha}$ which is equivalent to the continuity of the multilinear
tensor product map (\ref{bigotimes}).
\end{remark}

%To prove that $\mathfrak{M}_{\mathfrak{r}}(\mathbf{V}_D)$ can be immersed in a tensor
%Banach space we need a stronger condition than the continuity of the tensor product. 
Grothendieck \cite{Grothendiek1953} named the following norm $\left\Vert \cdot
\right\Vert _{\vee }$ the \emph{injective norm}.

\begin{definition}
Let $V_{\alpha}$ be a Banach space with norm $\left\Vert \cdot\right\Vert _{\alpha}$
for $\alpha\in D.$ Then for $\mathbf{v}\in\mathbf{V}=\left.
_{a}\bigotimes_{\alpha \in D}V_{\alpha}\right. $ define $\left\Vert \cdot\right\Vert
_{\vee(V_1,\ldots,V_d)}$ by%
\begin{equation}
\left\Vert \mathbf{v}\right\Vert _{\vee(V_1,\ldots,V_d)}:=\sup\left\{ \frac{%
\left\vert \left(
\varphi_{1}\otimes\varphi_{2}\otimes\ldots\otimes\varphi_{d}\right) (\mathbf{%
v})\right\vert }{\prod_{\alpha \in D}\Vert\varphi_{\alpha}\Vert_{\alpha}^{\ast}}%
:0\neq\varphi_{\alpha}\in V_{\alpha}^{\ast},\alpha \in D\right\} .
\label{(Norm ind*(V1,...,Vd)}
\end{equation}
\end{definition}

It is well known that the injective norm is a reasonable crossnorm (see
Lemma 1.6 in \cite{Light} and \eqref{(rcn a}-\eqref{(rcn b}). Further
properties are given by the next proposition (see Lemma 4.96 and Section 4.2.4 in 
\cite{Hackbusch}).

\begin{proposition}
\label{injective bounded below} Let $V_{\alpha}$ be a Banach space with norm $%
\left\Vert \cdot\right\Vert _{\alpha}$ for $\alpha\in D,$ and $\|\cdot\|$ be a
norm on $\mathbf{V}_D= \left._a \bigotimes_{\alpha \in D} V_{\alpha}\right..$ The following
statements hold.

\begin{itemize}
\item[(a)] For each $\alpha\in D$ introduce the tensor Banach space $%
\mathbf{X}_{\alpha}:=\left._{\|\cdot\|_{\vee(V_1,\ldots,V_{\alpha-1},V_{\alpha+1},%
\ldots,V_d)}} \bigotimes_{\beta \neq \alpha}V_\beta \right..$ Then 
\begin{equation}
\|\cdot\|_{\vee(V_1,\ldots,V_d)} = \|\cdot\|_{\vee\left(V_{\alpha}, \mathbf{X}_{\alpha}
\right)}
\end{equation}
holds for $\alpha \in D.$

\item[(b)] The injective norm is the weakest reasonable crossnorm on $%
\mathbf{V},$ i.e., if $\left\Vert \cdot\right\Vert $ is a reasonable
crossnorm on $\mathbf{V},$ then 
\begin{equation}
\left\Vert \cdot\right\Vert \;\gtrsim\ \left\Vert \cdot\right\Vert
_{\vee(V_1,\ldots,V_d)}.  \label{(Norm staerker als ind*}
\end{equation}

\item[(c)] For any norm $\left\Vert \cdot\right\Vert $ on $\mathbf{V}$
satisfying $\left\Vert \cdot\right\Vert
_{\vee(V_1,\ldots,V_d)}\lesssim\left\Vert \cdot\right\Vert ,$ the map (\ref%
{bigotimes}) is continuous, and hence Fr\'echet differentiable.
\end{itemize}
\end{proposition}

The following result shows an interesting use of the injective norm. 

\begin{corollary}
Assume that
$(V_{\alpha},\|\cdot\|_{\alpha})$ is a normed space for each $\alpha \in D.$ 
Then the algebraic tensor space $\mathbf{V} =  \left._a \bigotimes_{\alpha \in D} V_{\alpha } \right.$ is a $\mathcal{C}^{\infty}$-Banach
manifold not modelled on a particular Banach space.
\end{corollary}

\begin{proof}
Let $V_{\alpha_{\|\cdot\|_{\alpha}}}$ be the Banach space obtained by the completion of
$V_{\alpha}$ by using the norm $\|\cdot\|_{\alpha}$ for $\alpha \in D.$
Then we have
$$
\mathbf{V} = \left._a \bigotimes_{\alpha \in D} V_{\alpha } \right. \subset \mathbf{V}^{\star} = \left._a \bigotimes_{\alpha \in D} V_{\alpha_{\|\cdot\|_{\alpha}}} \right..
$$
From Proposition  \ref{injective bounded below}(c) the map  
$$
\bigotimes :\left(\mathop{\mathchoice{\raise-0.22em\hbox{\huge $\times$}}
{\raise-0.05em\hbox{\Large $\times$}}{\hbox{\large $\times$}}{\times}}%
_{\alpha \in D} V_{\alpha_{\|\cdot\|_{\alpha}}},\left\Vert \cdot\right\Vert \right) 
\longrightarrow%
\bigg(
\mathbf{V}^{\star}, \|\cdot\|_{\vee(V_{1_{\|\cdot\|_{1}}},\ldots, V_{d_{\|\cdot\|_{d}}})}
\bigg)
$$
is continuous and hence  
$$
\bigotimes :\left(\mathop{\mathchoice{\raise-0.22em\hbox{\huge $\times$}}
{\raise-0.05em\hbox{\Large $\times$}}{\hbox{\large $\times$}}{\times}}%
_{\alpha \in D} V_{\alpha},\left\Vert \cdot\right\Vert \right) 
\longrightarrow%
\bigg(
\mathbf{V}, \|\cdot\|_{\vee(V_{1_{\|\cdot\|_{1}}},\ldots, V_{d_{\|\cdot\|_{d}}})}
\bigg)
$$
is also continuous. Then Corollary \ref{Bounded_Banach_Manifold}
proves the desired conclusion.
\end{proof}

\begin{remark}
Observe that from the proof of the above corollary, we can conclude that $\mathbf{V}^{\star} = 
\left._a \bigotimes_{\alpha \in D} V_{\alpha_{\|\cdot\|_{\alpha}}} \right.$ is also 
 a $\mathcal{C}^{\infty}$-Banach
manifold not modelled on a particular Banach space.
\end{remark}

\subsubsection{$\mathrm{T}_{\mathbf{v}}\mathfrak{i%
}(\mathbb{T}_{\mathbf{v}}(\mathfrak{M}_{\mathfrak{r}}(\mathbf{V}_{D})))$
belongs to $\mathbb{G}(\mathbf{V}_{D_{\|\cdot\|_D}})$}

We will assume that the norm $\left\Vert \cdot\right\Vert_D$ on $\mathbf{V}_D$
satisfies
\begin{align}\label{tree_injective_norm}
\left\Vert \cdot\right\Vert_{\vee(V_1,\ldots,V_d)}\lesssim\left\Vert \cdot\right\Vert_D, 
\end{align} 
and hence, by Proposition~\ref{injective bounded below}(c), under this condition, 
Proposition~\ref{rank_one_tangent_space} also holds.
A first useful result is the following lemma.

\bigskip

\begin{lemma}
\label{two_faces} Assume that
$(V_{\alpha},\|\cdot\|_{\alpha})$ is a normed space 
for each $\alpha \in D$ 
and let $\|\cdot\|_D$ be a norm 
on the tensor space $\mathbf{V}_D = \left.
_{a}\bigotimes_{\alpha \in D}V_{\alpha}\right.$ such that 
\eqref{tree_injective_norm} holds. Let $\beta \in D.$ If $%
W_{\beta }\in \mathbb{G}(V_{\beta _{\Vert \cdot \Vert _{\beta }}})$
satisfies $V_{\beta _{\Vert \cdot \Vert _{\beta }}}=U_{\beta
}\oplus W_{\beta }$ for some finite-dimensional subspace $U_{\beta }$ in $%
V_{\beta _{\Vert \cdot \Vert _{\beta }}},$ then $W_{\beta }\otimes
_{a}U_{[\beta ]}\in \mathbb{G}(\mathbf{V}_{D _{\Vert \cdot \Vert
_{D}}})$ for every finite-dimensional subspace $U_{[\beta ]}\subset \mathbf{V}_{[\beta]} =
\left. _{a}\bigotimes_{\delta \in D\setminus \{\beta\}}V%
_{\delta _{\Vert \cdot \Vert _{\delta }}}\right. .$
\end{lemma}

\begin{proof}
First, observe that if $W_{\beta }$ is a finite-dimensional subspace, then $%
W_{\beta }\otimes _{a}U_{[\beta ]}$ is also finite-dimensional, and hence
the lemma follows. Thus, assume that $W_{\beta }$ is an infinite-dimensional
closed subspace of $V_{\beta _{\Vert \cdot \Vert _{\beta }}},$ and
to simplify the notation write 
\begin{equation*}
\mathbf{X}_{\beta }:=\left. _{\Vert \cdot \Vert _{\vee(
	V_1,\ldots,V_{\beta-1},V_{\beta+1},\ldots,V_{d})
}}\bigotimes_{\delta \in D\setminus \{\beta \}}\mathbf{V}_{\delta
_{\Vert \cdot \Vert _{\delta }}}\right. .
\end{equation*}%
If $U_{[\beta ]}\subset \mathbf{X}_{\beta }$ is a finite-dimensional
subspace, then there exists $W_{[\beta ]}\in \mathbb{G}(\mathbf{X}_{\beta })$
such that $\mathbf{X}_{\beta }=U_{[\beta ]}\oplus W_{[\beta ]}.$ Since the
tensor product map 
\begin{equation*}
\bigotimes :(V_{\beta _{\Vert \cdot \Vert _{\beta }}},\Vert \cdot
\Vert _{\beta })\times \left( \mathbf{X}_{\beta },\Vert \cdot \Vert _{\vee(
	V_1,\ldots,V_{\beta-1},V_{\beta+1},\ldots,V_{d})
}\right) \rightarrow (\mathbf{V}_{D _{\Vert
\cdot \Vert _{D}}},\Vert \cdot \Vert _{D })
\end{equation*}%
is continuous and by Lemma~3.18 in \cite{FALHACK}, for each elementary
tensor $\mathbf{v}_{\beta }\otimes \mathbf{v}_{[\beta ]}\in \mathbf{V}%
_{\beta _{\Vert \cdot \Vert _{\beta }}}\otimes _{a}\mathbf{X}_{\beta }$ we
have 
\begin{align*}
\Vert (id_{\beta }\otimes P_{_{U_{[\beta ]}\oplus W_{[\beta ]}}})(\mathbf{v}%
_{\beta }\otimes \mathbf{v}_{[\beta ]})\Vert _{\alpha }& \leq C\sqrt{\dim
U_{[\beta ]}}\,\Vert \mathbf{v}_{\beta }\Vert _{\beta }\Vert \mathbf{v}%
_{[\beta ]}\Vert _{\vee(
	V_1,\ldots,V_{\beta-1},V_{\beta+1},\ldots,V_{d})} \\
& =C\,\sqrt{\dim U_{[\beta ]}}\,\Vert \mathbf{v}_{\beta }\otimes \mathbf{v}%
_{[\beta ]}\Vert _{\vee(
	V_1,\ldots,V_{\beta-1},V_{\beta},V_{\beta+1},\ldots,V_{d})} \\
& \leq C^{\prime }\,\sqrt{\dim U_{[\beta ]}}\,\Vert \mathbf{v}_{\beta
}\otimes \mathbf{v}_{[\beta ]}\Vert _{D}.
\end{align*}%
Thus, $(id_{\beta }\otimes P_{_{U_{[\beta ]}\oplus W_{[\beta ]}}})$ is
continuous over $\mathbf{V}_{\beta _{\Vert \cdot \Vert _{\beta }}}\otimes
_{a}\mathbf{X}_{\beta },$ and hence in $\mathbf{V}_{D _{\Vert \cdot
\Vert _{D}}}.$ Now, take into account the fact that 
\begin{equation*}
id_{\beta }=P_{_{U_{\beta }\oplus W_{\beta }}}+P_{_{W_{\beta }\oplus
U_{\beta }}},
\end{equation*}%
so that
\begin{equation*}
id_{\beta }\otimes P_{_{U_{[\beta ]}\oplus W_{[\beta ]}}}=P_{_{U_{\beta
}\oplus W_{\beta }}}\otimes P_{_{U_{[\beta ]}\oplus W_{[\beta
]}}}+P_{_{W_{\beta }\oplus U_{\beta }}}\otimes P_{_{U_{[\beta ]}\oplus
W_{[\beta ]}}}.
\end{equation*}%
Observe that $id_{\beta }\otimes P_{_{U_{[\beta ]}\oplus W_{[\beta ]}}}$ and 
$P_{_{U_{\beta }\oplus W_{\beta }}}\otimes P_{_{U_{[\beta ]}\oplus W_{[\beta
]}}}$ are continuous linear maps over $\mathbf{V}_{\beta _{\Vert \cdot \Vert
_{\beta }}}\otimes _{a}\mathbf{X}_{\beta },$ and then $P_{_{W_{\beta }\oplus
U_{\beta }}}\otimes P_{_{U_{[\beta ]}\oplus W_{[\beta ]}}}$ is a continuous
linear map over $\mathbf{V}_{\beta _{\Vert \cdot \Vert _{\beta }}}\otimes
_{a}\mathbf{X}_{\beta }.$ Thus, 
\begin{equation*}
\mathcal{P}:=\overline{P_{_{W_{\beta }\oplus U_{\beta }}}\otimes
P_{_{U_{[\beta ]}\oplus W_{[\beta ]}}}}\in \mathcal{L}(\mathbf{V}_{D
_{\Vert \cdot \Vert _{D}}},\mathbf{V}_{D _{\Vert \cdot \Vert
_{D }}})
\end{equation*}%
and $\mathcal{P}\circ \mathcal{P}=\mathcal{P}.$
Since $\mathcal{P}(\mathbf{V}_{D_{\Vert \cdot \Vert _{D
}}})=W_{\beta }\otimes _{a}U_{[\beta ]},$ 
the lemma follows by Proposition~\ref{characterize_P}.
\end{proof}

\bigskip

\begin{lemma}
\label{direct_sum_G} Let $X$ be a Banach space and assume that $U,V \in 
\mathbb{G}(X).$ If $U\cap V=\{0\},$ then $U\oplus V\in \mathbb{G}(X).$
Moreover, $U\cap V\in \mathbb{G}(X)$ holds.
\end{lemma}

\begin{proof}
To prove the first statement assume that $U \cap V = \{0\}.$ Since $U,V \in 
\mathbb{G}(X)$ there exist $U^{\prime },V^{\prime }\in \mathbb{G}(X),$ such
that $X=U \oplus U^{\prime }= V \oplus V^{\prime }.$ Then $U=X \cap U = (V
\oplus V^{\prime }) \cap U = U \cap V^{\prime }$ and $V = X \cap V = (U
\oplus U^{\prime }) \cap V = V \cap U^{\prime }.$ In consequence, we can
write 
\begin{equation*}
U \oplus V \oplus (U^{\prime }\cap V^{\prime }) = (U \cap V^{\prime })
\oplus (V \cap U^{\prime }) \oplus (U^{\prime }\cap V^{\prime }) = (U \oplus
U^{\prime }) \cap ( V \oplus V^{\prime }) = X,
\end{equation*}
and the first statement follows. To prove the second one, observe that $X =
(U \cap V) \oplus (U \cap V^{\prime }) \oplus (V \cap U^{\prime }) \oplus
(U^{\prime }\cap V^{\prime }). $
\end{proof}

\bigskip

An important consequence of the above two lemmas is the following theorem.

\begin{theorem}
\label{closed_linear_subspace} Assume that
$(V_{\alpha},\|\cdot\|_{\alpha})$ is a normed space 
for each $\alpha \in D$ 
and let $\|\cdot\|_D$ be a norm 
on the tensor space $\mathbf{V}_D = \left.
_{a}\bigotimes_{\alpha \in D}V_{\alpha}\right.$ such that  
\eqref{tree_injective_norm} holds. Then for each
$\mathbf{v} \in \mathfrak{M}_{\mathfrak{r}}(\mathbf{V}_{D})$ we have 
$
\mathbf{Z}^{(D)}(\mathbf{v}) \in \mathbb{G}(\mathbf{V}_{D_{\Vert \cdot \Vert
_{D}}}),$
and hence $\mathfrak{M}_{\mathfrak{r}}(\mathbf{V}_{D})$ is an
immersed submanifold of $\mathbf{V}_{D_{\Vert \cdot \Vert _{D}}}.$
\end{theorem}

\begin{proof}
For each $\alpha \in D$ we have $W_{\alpha }^{\min }(\mathbf{v})\in 
\mathbb{G}(V_{\alpha _{\Vert \cdot \Vert _{\alpha }}})$ and $%
U_{D\setminus \{\alpha \}}^{\min }(\mathbf{v})\subset \left.
_{a}\bigotimes_{\delta \in D\setminus \{\alpha \}}\mathbf{V}_{\delta
_{\Vert \cdot \Vert _{\delta }}}\right. $ is a finite-dimensional subspace.
From Lemma~\ref{two_faces} we have $W_{\alpha }^{\min }(\mathbf{v})\otimes
_{a}U_{D\setminus \{\alpha \}}^{\min }(\mathbf{v})\in \mathbb{G}(\mathbf{V%
}_{D_{\Vert \cdot \Vert _{D}}})$ for all $\alpha \in D.$ Since $\left.
_{a}\bigotimes_{\alpha \in D}U_{\alpha }^{\min }(\mathbf{v})\right. \in 
\mathbb{G}(\mathbf{V}_{D_{\Vert \cdot \Vert _{D}}}),$ by Lemma~\ref%
{direct_sum_G}, we obtain that $\mathbf{Z}^{(D)}(\mathbf{v})\in \mathbb{G}(%
\mathbf{V}_{D_{\Vert \cdot \Vert _{D}}}).$
\end{proof}

\begin{example}
Let us recall the topological tensor spaces introduced in Example \ref{Bsp HNp}%
. Let $I_{\alpha}\subset \mathbb{R}$ $\left( \alpha \in D\right) $ and $1\leq
p<\infty.$ Let $\mathbf{I}:= 
\mathop{\mathchoice{\raise-0.22em\hbox{\huge
$\times$}} {\raise-0.05em\hbox{\Large $\times$}}{\hbox{\large
$\times$}}{\times}}_{\alpha\in D} I_{\alpha}$. Hence $L^p(\mathbf{I})$
is a tensor Banach space for all $\alpha \in T_D.$ In this example we denote
the usual norm of $L^p(\mathbf{I})$ by $\|\cdot\|_{0,p}.$
Since $\|\cdot\|_{0,p}$ is a reasonable crossnorm (see Example 4.72 in 
\cite{Hackbusch}), then \eqref{tree_injective_norm} holds. From Theorem \ref{closed_linear_subspace} we obtain that $\mathfrak{%
M}_{\mathfrak{r}}\left(\left._a \bigotimes_{\alpha \in D} L^p(I_{\alpha})\right.\right)$
is an immersed submanifold of $L^p(\mathbf{I}).$
\end{example}

\begin{example}
Now, we return to Example~\ref{example_BM1}. From Example 4.42 in \cite%
{Hackbusch} we know that the norm \newline
$\|\cdot\|_{(0,1),p}$ is a crossnorm on $H^{1,p}(I_1) \otimes_a
H^{1,p}(I_2), $ and hence it is not weaker than the injective norm. In
consequence, from Theorem~\ref{closed_linear_subspace}, we obtain that $%
\mathfrak{M}_{\mathfrak{r}}(H^{1,p}(I_1) \otimes_a H^{1,p}(I_2))$ is an
immersed submanifold in $H^{1,p}(I_1) \otimes_{\|\cdot\|_{(0,1),p}}
H^{1,p}(I_2).$
\end{example}

Since in a reflexive Banach space every closed linear subspace is proximinal
(see p. 61 in \cite{Floret}), we have the following corollary.

\begin{corollary}
\label{approximation_corollary} Assume that
$(V_{\alpha},\|\cdot\|_{\alpha})$ is a normed space 
for each $\alpha \in D$ 
and let $\|\cdot\|_D$ be a norm 
on the tensor space $\mathbf{V}_D = \left.
_{a}\bigotimes_{\alpha \in D}V_{\alpha}\right.$ such that 
\eqref{tree_injective_norm} holds and $\mathbf{V}_{D_{\|\cdot\|_D}} = \left.
_{\|\cdot\|_D}\bigotimes_{\alpha \in D}V_{\alpha}\right.$ is a reflexive Banach space. 
Then for any $\mathbf{v}\in \mathfrak{M}_{\mathfrak{r}}(\mathbf{V}_{D})$
and $\dot{\mathbf{u}}\in \mathbf{V}_{D_{\Vert \cdot \Vert _{D}}},$
there exists $\dot{\mathbf{v}}_{best}\in \mathbf{Z}^{(D)}(\mathbf{v})$ such
that 
\begin{equation*}
\Vert \dot{\mathbf{u}}-\dot{\mathbf{v}}_{best}\Vert =\min_{\dot{\mathbf{v}}%
\in \mathbf{Z}^{(D)}(\mathbf{v})}\Vert \dot{\mathbf{u}}-\dot{\mathbf{v}}%
\Vert .
\end{equation*}
\end{corollary}

\section{On the Dirac--Frenkel variational principle on tensor Banach spaces}
\label{Dirac_Frenkel}

\subsection{Model reduction in tensor Banach spaces}

In this section we consider the abstract ordinary differential equation in a
reflexive tensor Banach space $\mathbf{V}_{D_{\Vert \cdot \Vert _{D}}},$
given by 
\begin{align}
\dot{\mathbf{u}}(t)& =\mathbf{F}(t,\mathbf{u}(t)),\text{ for }t\geq 0,
\label{BODE1} \\
\mathbf{u}(0)& =\mathbf{u}_{0},  \label{BODE2}
\end{align}%
where we assume $\mathbf{u}_{0}\neq \mathbf{0}$ and $\mathbf{F}:[0,\infty
)\times \mathbf{V}_{D_{\Vert \cdot \Vert _{D}}}\longrightarrow \mathbf{V}%
_{D_{\Vert \cdot \Vert _{D}}}$ satisfying the usual conditions 
to have existence and uniqueness of solutions. As usual we assume that
$(V_{\alpha},\|\cdot\|_{\alpha})$ is a normed space 
for each $\alpha \in D$ 
and let $\|\cdot\|_D$ be a norm 
on the tensor space $\mathbf{V}_D = \left.
_{a}\bigotimes_{\alpha \in D}V_{\alpha}\right.$ such that \eqref{tree_injective_norm}
holds. We want to approximate $%
\mathbf{u}(t),$ for $t\in I:=(0,T )$ for some $T >0,$ by
a differentiable curve $t\mapsto \mathbf{v}_{r}(t)$ from $I$ to $\mathfrak{M}%
_{\mathfrak{r}}(\mathbf{V}_{D}),$ where $\mathfrak{r}\in \mathcal{AD}(\mathbf{V}_D)$
$(\mathfrak{r} \neq \mathbf{0}),$ such that $\mathbf{v}_{r}(0)=\mathbf{v}_{0}\in \mathfrak{M}%
_{\mathfrak{r}}(\mathbf{V}_{D})$ is an approximation of $\mathbf{u}_{0}.$%
\footnote{%
$\mathbf{v}_{0}$ can be chosen as the best approximation of $\mathbf{%
u}_{0}$ in $\mathfrak{M}_{\mathfrak{r}}(\mathbf{V}_D)$ because a best approximation exists \cite{FALHACK}.}

\bigskip

Our main goal is to construct a reduced order model of \eqref{BODE1}--\eqref%
{BODE2} over the Banach manifold $\mathfrak{M}_{\mathfrak{r}}(\mathbf{V}%
_{D}).$ Since $\mathbf{F}(t,\mathbf{v}_{r}(t))\in \mathbf{V}_{D_{\Vert
\cdot \Vert _{D}}},$ for each $t\in I,$ and $\mathbf{Z}^{(D)}(\mathbf{v}%
_{r}(t))$ is a closed linear subspace in $\mathbf{V}_{D_{\Vert \cdot \Vert
_{D}}},$ we have the existence of a $\dot{\mathbf{v}}_{r}(t)\in \mathbf{Z}%
^{(D)}(\mathbf{v}_{r}(t))$ such that 
\begin{equation*}
\Vert \dot{\mathbf{v}}_{r}(t)-\mathbf{F}(t,\mathbf{v}_{r}(t))\Vert
_{D}=\min_{\dot{\mathbf{v}}\in \mathbf{Z}^{(D)}(\mathbf{v}_{r}(t))}\Vert 
\dot{\mathbf{v}}-\mathbf{F}(t,\mathbf{v}_{r}(t))\Vert _{D}.
\end{equation*}
It is well known that, if $\mathbf{V}_{D_{\Vert \cdot \Vert _{D}}}$ is a
Hilbert space, then $\dot{\mathbf{v}}_{r}(t)=\mathcal{P}_{\mathbf{v}_{r}(t)}(%
\mathbf{F}(t,\mathbf{v}_{r}(t))),$ where 
\begin{equation*}
\mathcal{P}_{\mathbf{v}_{r}(t)}=\mathcal{P}_{\mathbf{Z}^{(D)}(\mathbf{v}%
_{r}(t))\oplus \mathbf{Z}^{(D)}(\mathbf{v}_{r}(t))^{\bot }}
\end{equation*}%
is called the \emph{metric projection.} It has the following important
property: $\dot{\mathbf{v}}_{r}(t)=\mathcal{P}_{\mathbf{v}_{r}(t)}(\mathbf{F}%
(t,\mathbf{v}_{r}(t)))$ if and only if
\begin{equation*}
\langle \dot{\mathbf{v}}_{r}(t)-\mathbf{F}(t,\mathbf{v}_{r}(t)),\dot{\mathbf{%
v}}\rangle _{D}=0\text{ for all }\dot{\mathbf{v}} \in \mathbf{Z}^{(D)}(%
\mathbf{v}_{r}(t)).
\end{equation*}

The concept of a metric projection can be extended to the Banach space
setting. To this end we recall the following definitions. A Banach space $X$
with norm $\Vert \cdot \Vert $ is said to be \emph{strictly convex} if $%
\Vert x+y\Vert /2<1$ for all $x,y\in X$ with $\Vert x\Vert =\Vert y\Vert =1$
and $x\neq y.$ It is \emph{uniformly convex} if $\lim_{n\rightarrow \infty
}\Vert x_{n}-y_{n}\Vert =0$ for any two sequences $\{x_{n}\}_{n\in \mathbb{N}%
}$ and $\{y_{n}\}_{n\in \mathbb{N}}$ such that $\Vert x_{n}\Vert =\Vert
y_{n}\Vert =1$ and $\lim_{n\rightarrow \infty }\Vert x_{n}+y_{n}\Vert /2=1.$
It is known that a uniformly convex Banach space is reflexive and strictly
convex. A Banach space $X$ is said to be \emph{smooth} if the limit 
\begin{equation*}
\lim_{t\rightarrow 0}\frac{\Vert x+ty\Vert -\Vert x\Vert }{t}
\end{equation*}%
exists for all $x,y\in \mathbb{S}_X=\{z\in X:\Vert z\Vert =1\}.$ Finally, a Banach
space $X$ is said to be \emph{uniformly smooth} if its modulus of smoothness 
\begin{equation*}
\rho (\tau )=\sup_{x,y\in \mathbb{S}_X}\left\{ \frac{\Vert x+\tau y\Vert +\Vert x-\tau
y\Vert }{2}-1\right\} ,\,\tau >0,
\end{equation*}%
satisfies the condition $\lim_{\tau \rightarrow 0}\rho (\tau )=0.$ In
uniformly smooth spaces, and only in such spaces, the norm is uniformly Fr%
\'{e}chet differentiable. A uniformly smooth Banach space is smooth. The
converse is true if the Banach space is finite-dimensional. It is known that
the space $L^{p}$ $(1<p<\infty )$ is a uniformly convex and uniformly smooth
Banach space.

\bigskip

Let $\langle \cdot ,\cdot \rangle :X\times X^{\ast }\longrightarrow \mathbb{R%
}$ denote the duality pairing, i.e.,%
\begin{equation*}
\langle x,f\rangle :=f(x).
\end{equation*}%
The normalised duality mapping $J:X\longrightarrow 2^{X^{\ast }}$ is defined
by 
\begin{equation*}
J(x):=\{f\in X^{\ast }:\langle x,f\rangle =\Vert x\Vert ^{2}=(\Vert f\Vert
^{\ast })^{2}\},
\end{equation*}%
and it has the following properties (see \cite%
{Alber}):

\begin{enumerate}
\item[(a)] If $X$ is smooth, the map $J$ is single-valued;

\item[(b)] if $X$ is smooth, then $J$ is norm--to--weak$^{\ast}$ continuous;

\item[(c)] if $X$ is uniformly smooth, then $J$ is uniformly norm--to--norm
continuous on each bounded subset of $X.$
\end{enumerate}

\begin{remark}
In a Hilbert space, the normalised duality mapping is the Riesz map. Notice that  after identifying $X$ with $X^*,$ 
it can be shown (see Proposition~4.8(i) in \cite{Cioranescu}) 
that the normalised duality mapping is the identity operator. 
\end{remark}

Assume that
$(V_{\alpha},\|\cdot\|_{\alpha})$ is a normed space 
for each $\alpha \in D.$ Let $\mathbf{V}_{D_{\Vert \cdot \Vert _{D}}}=\left. _{\Vert
\cdot \Vert _{D}}\bigotimes_{\alpha \in D}V_{\alpha}\right.$ be a reflexive and strictly convex
tensor Banach space such that \eqref{tree_injective_norm} holds. For $%
\mathbf{F}(t,\mathbf{v}_{r}(t))$ in $\mathbf{V}_{D_{\Vert \cdot \Vert _{D}}},
$ with a fixed $t\in I,$ it is known that the set 
\begin{equation*}
\left\{ \dot{\mathbf{v}}_{r}(t):\Vert \dot{\mathbf{v}}_{r}(t)-\mathbf{F}(t,%
\mathbf{v}_{r}(t))\Vert _{D}=\min_{\dot{\mathbf{v}}\in \mathbf{Z}^{(D)}(%
\mathbf{v}_{r}(t))}\Vert \dot{\mathbf{v}}-\mathbf{F}(t,\mathbf{v}%
_{r}(t))\Vert _{D}\right\}
\end{equation*}%
is always a singleton. Let $\mathcal{P}_{\mathbf{v}_{r}(t)}$ be the mapping
from $\mathbf{V}_{D_{\Vert \cdot \Vert _{D}}}$ onto $\mathbf{Z}^{(D)}(\mathbf{v%
}_{r}(t))$ defined by $\dot{\mathbf{v}}_{r}(t):=\mathcal{P}_{\mathbf{v}%
_{r}(t)}(\mathbf{F}(t,\mathbf{v}_{r}(t)))$ if and only if 
\begin{equation*}
\Vert \dot{\mathbf{v}}_{r}(t)-\mathbf{F}(t,\mathbf{v}_{r}(t))\Vert
_{D}=\min_{\dot{\mathbf{v}}\in \mathbf{Z}^{(D)}(\mathbf{v}_{r}(t))}\Vert 
\dot{\mathbf{v}}-\mathbf{F}(t,\mathbf{v}_{r}(t))\Vert _{D}.
\end{equation*}%
It is also called \emph{the metric projection.} The classical
characterisation of the metric projection together with Proposition 2.10 of \cite{Alber}
allows us to state the next result.

\begin{theorem}
\label{th_metric_projection} Assume that
$(V_{\alpha},\|\cdot\|_{\alpha})$ is a normed space 
for each $\alpha \in D.$  Let $\mathbf{V}_{D_{\Vert \cdot \Vert _{D}}}=\left. _{\Vert
\cdot \Vert _{D}}\bigotimes_{\alpha \in D}V_{\alpha}\right.$ be a reflexive and strictly convex
tensor Banach space such that \eqref{tree_injective_norm} holds. Then for each $t\in I$  the following statements are equivalent.
\begin{enumerate}
\item[(a)] $\dot{\mathbf{v}}_{r}(t)=\mathcal{P}_{\mathbf{v}_{r}(t)}(\mathbf{F}(t,\mathbf{%
v}_{r}(t))).$ 
\item[(b)]$
\langle \dot{\mathbf{v}}_{r}(t)-\dot{\mathbf{v}},J(\mathbf{F}(t,\mathbf{v}%
_{r}(t))-\dot{\mathbf{v}}_{r}(t))\rangle \geq 0\text{ for all }\dot{\mathbf{v%
}}\in \mathbf{Z}^{(D)}(\mathbf{v}_{r}(t)).$
\item[(c)] $\langle \dot{\mathbf{v}},J(\mathbf{F}(t,\mathbf{v}%
_{r}(t))-\dot{\mathbf{v}}_{r}(t))\rangle = 0\text{ for all }\dot{\mathbf{v%
}}\in \mathbf{Z}^{(D)}(\mathbf{v}_{r}(t)).$
\end{enumerate}
\end{theorem}

An alternative approach is the use of the so-called \emph{generalised
projection operator} (see also \cite{Alber}). To formulate this, we will use
the following framework. Let $\mathbf{V}%
_{D_{\Vert \cdot \Vert _{D}}}$ be a reflexive, strictly convex and smooth
tensor Banach space. Following \cite{KamiTaka}, we can define a function $%
\phi :\mathbf{V}_{D_{\Vert \cdot \Vert _{D}}}\times \mathbf{V}_{D_{\Vert
\cdot \Vert _{D}}}\longrightarrow \mathbb{R}$ by 
\begin{equation*}
\phi (\mathbf{u},\mathbf{v})=\Vert \mathbf{u}\Vert _{D}^{2}-2\langle \mathbf{%
u},J(\mathbf{v})\rangle +\Vert \mathbf{v}\Vert _{D}^{2},
\end{equation*}%
where $\langle \cdot ,\cdot \rangle $ denotes the duality pairing and $J$ is the
normalised duality mapping. It is known that the set 
\begin{equation*}
\left\{ \dot{\mathbf{v}}_{r}(t):\phi (\dot{\mathbf{v}}_{r}(t),\mathbf{F}(t,%
\mathbf{v}_{r}(t)))=\min_{\dot{\mathbf{v}}\in \mathbf{Z}^{(D)}(\mathbf{v}%
_{r}(t))}\phi (\dot{\mathbf{v}},\mathbf{F}(t,\mathbf{v}_{r}(t)))\right\}
\end{equation*}%
is always a singleton. It allows us to define a map $\Pi _{\mathbf{v}%
_{r}(t)}:\mathbf{V}_{D_{\Vert \cdot \Vert _{D}}}\longrightarrow \mathbf{Z}%
^{(D)}(\mathbf{v}_{r}(t)) $ by $\dot{\mathbf{v}}_{r}(t):=\Pi _{\mathbf{v}%
_{r}(t)}(\mathbf{F}(t,\mathbf{v}_{r}(t)))$ if and only if 
\begin{equation*}
\phi (\dot{\mathbf{v}}_{r}(t),\mathbf{F}(t,\mathbf{v}_{r}(t)))=\min_{\dot{%
\mathbf{v}}\in \mathbf{Z}^{(D)}(\mathbf{v}_{r}(t))}\phi (\dot{\mathbf{v}}%
,\mathbf{F}(t,\mathbf{v}_{r}(t))).
\end{equation*}%
The map $\Pi _{\mathbf{v}_{r}(t)}$ is called \emph{the generalised
projection.} It coincides with the metric projection when $\mathbf{V}%
_{D_{\Vert \cdot \Vert _{D}}}$ is a Hilbert space.

\begin{remark}
We point out that, in general, the operators $\mathcal{P}_{\mathbf{v}_r(t)} $
and $\Pi_{\mathbf{v}_r(t)}$ are nonlinear in Banach (not Hilbert) spaces.
\end{remark}

Again, a classical characterisation of the generalised projection gives us
the following theorem.

\begin{theorem}
\label{th_generalized_projection} Assume that
$(V_{\alpha},\|\cdot\|_{\alpha})$ is a normed space 
for each $\alpha \in D.$  Let $\mathbf{V}_{D_{\Vert \cdot \Vert _{D}}}=\left. _{\Vert
\cdot \Vert _{D}}\bigotimes_{\alpha \in D}V_{\alpha}\right.$ be a reflexive and strictly convex
tensor Banach space such that \eqref{tree_injective_norm} holds. Then for each $t\in I$ we
have 
\begin{equation*}
\dot{\mathbf{v}}_{r}(t)=\Pi _{\mathbf{v}_{r}(t)}(\mathbf{F}(t,\mathbf{v}%
_{r}(t)))
\end{equation*}%
if and only if 
\begin{equation*}
\langle \dot{\mathbf{v}}_{r}(t)-\dot{\mathbf{v}},J(\mathbf{F}(t,\mathbf{v}%
_{r}(t)))-J(\dot{\mathbf{v}}_{r}(t))\rangle \geq 0\text{ for all }\dot{%
\mathbf{v}}\in \mathbf{Z}^{(D)}(\mathbf{v}_{r}(t)).
\end{equation*}
\end{theorem}

\subsection{The time--dependent Hartree method}

Assume that $(V_{\alpha},\|\cdot\|_{\alpha})$ is a Banach space 
for each $\alpha \in D.$  Let $\mathbf{V}_{\|\cdot\|}=\left. _{\Vert
\cdot \Vert _{D}}\bigotimes_{\alpha \in D}V_{\alpha}\right.$ 
be a reflexive and strictly convex tensor Banach space such 
that \eqref{tree_injective_norm} holds. Let us consider 
in $\mathbf{V}_{\|\cdot\|}$ a flow generated by a densely
defined operator $A \in L(\mathbf{V}_{\|\cdot\|},\mathbf{V}_{\|\cdot\|}).$
More precisely, there exists a collection of bijective maps $\boldsymbol{%
\varphi}_t:\mathcal{D}(A) \longrightarrow \mathcal{D}(A),$ here $\mathcal{D}%
(A)$ denotes the domain of $A,$ satisfying

\begin{enumerate}
\item[(i)] $\boldsymbol{\varphi}_0 = \mathbf{id},$

\item[(ii)] $\boldsymbol{\varphi}_{t+s} = \boldsymbol{\varphi}_t \circ 
\boldsymbol{\varphi}_s,$ and

\item[(iii)] for $\mathbf{u}_0 \in \mathcal{D}(A),$ the map $t \mapsto 
\boldsymbol{\varphi}_t$ is differentiable as a curve in $\mathbf{V}%
_{\|\cdot\|},$ and $\mathbf{u}(t):= \boldsymbol{\varphi}_t(\mathbf{u}_0)$
satisfies 
\begin{align*}
\dot{\mathbf{u}} & = A \mathbf{u}, \\
\mathbf{u}(0) & = \mathbf{u}_0.
\end{align*}
\end{enumerate}

In this framework we want to study the approximation of a solution $\mathbf{u%
}(t)=\boldsymbol{\varphi }_{t}(\mathbf{u}_{0})\in \mathbf{V}_{\Vert \cdot
\Vert }$ by a curve $\mathbf{v}_{r}(t):=\lambda (t)\otimes
_{\alpha \in D}v_{\alpha}(t) $ in the Banach manifold $\mathfrak{M}_{(1,\ldots ,1)}(%
\mathbf{V}),$ also called in \cite{Lubish} the \emph{Hartree manifold.} The
time--dependent Hartree method consists in the use of the Dirac--Frenkel
variational principle on the Hartree manifold. More precisely, we want to
solve the following reduced order model: 
\begin{align*}
\dot{\mathbf{v}}_{r}(t)& =\mathcal{P}_{\mathbf{v}_{r}(t)}(A\mathbf{v}_{r}(t))%
\text{ for }t\in I, \\
\mathbf{v}_{r}(0)& =\mathbf{v}_{0},
\end{align*}%
with $\mathbf{v}_{0}=\lambda _{0}\otimes _{\alpha \in D}v_{0}^{(\alpha)}\in \mathfrak{M}%
_{(1,\ldots ,1)}(\mathbf{V})$ being an approximation of $\mathbf{u}_{0}.$ 
By using the characterisation of the metric
projection in a Banach space, for each $t>0$ we would like to find $\dot{%
\mathbf{v}}_{r}(t)\in \mathrm{T}_{\mathbf{v}_{r}(t)}\mathfrak{i}\left( 
\mathbb{T}_{\mathbf{v}_{r}(t)}(\mathfrak{M}_{(1,\ldots ,1)}(\mathbf{V}%
))\right) $ such that 
\begin{align}
\langle \dot{\mathbf{v}}, J(\dot{\mathbf{v}}_{r}(t)-A\mathbf{v}_{r}(t))\rangle & =0\text{ for all }\dot{\mathbf{v}}\in \mathrm{T}_{\mathbf{v}%
_{r}(t)}\mathfrak{i}\left( \mathbb{T}_{\mathbf{v}_{r}(t)}(\mathfrak{M}%
_{(1,\ldots ,1)}(\mathbf{V}))\right),  \label{eq:11}
\end{align}
\begin{align*}
\mathbf{v}_{r}(0)& =\mathbf{v}_{0}=\lambda _{0}\otimes _{\alpha \in D}v_{0}^{(\alpha)}.
\end{align*}%
A first result is the following Lemma.

\begin{lemma}
\label{previous_HF} Let $\mathbf{v}\in \mathcal{C}^{1}(I,\mathcal{U}(\mathbf{%
v}_{0})),$ where $\mathbf{v}(0)=\mathbf{v}_{0}\in \mathfrak{M}_{(1,\ldots
,1)}(\mathbf{V})$ and $(\mathcal{U}(\mathbf{v}_{0}),\Theta _{\mathbf{v}%
_{0}}) $ is a local chart at $\mathbf{v}_{0}$ in $\mathfrak{M}_{(1,\ldots
,1)}(\mathbf{V}).$ Assume that $\mathbf{v}$ is also a $\mathcal{C}^{1}$%
-morphism between the manifolds $I\subset \mathbb{R}$ and $\mathcal{U}(%
\mathbf{v}_{0})\subset \mathfrak{M}_{(1,\ldots ,1)}(\mathbf{V})$ such that $%
\mathbf{v}(t)=\lambda (t)\bigotimes_{\alpha \in D}v_{\alpha}(t)$ for some $\lambda \in 
\mathcal{C}^{1}(I,\mathbb{R})$ and $v_{\alpha}\in \mathcal{C}^{1}(I,V_{\alpha})$ for $%
\alpha \in D.$ Then 
\begin{equation}  \label{curve_derivative}
\dot{\mathbf{v}}(t)=\dot{\lambda}(t)\bigotimes_{\alpha \in D}v_{\alpha}(t)+\lambda
(t)\sum_{\alpha \in D}\dot{v}_{\alpha}(t)\otimes \bigotimes_{\substack{\beta \in D\\\beta \neq \alpha}}v_{\beta}(t)=\mathrm{T%
}_{\mathbf{v}(t)}\mathfrak{i}(\mathrm{T}_{t}\mathbf{v}(1)),\footnote{Observe that the derivative at $t$
of a map $\mathbf{v}:I \rightarrow \mathfrak{M}_{(1,\ldots
,1)}(\mathbf{V})$ considered as a morphism between manifolds is given by a 
linear map $\mathrm{T}_t \mathbf{v}:\mathbb{R} \rightarrow
\mathbb{T}_{\mathbf{v}(t)}(\mathfrak{M}_{(1,\ldots
,1)}(\mathbf{V}))$ which is characterised by the fact that $\mathrm{T}_t \mathbf{v}(\dot{\mu}) = \dot{\mu} \mathrm{T}_t \mathbf{v}(1)$ holds
for all $\dot{\mu} \in \mathbb{R}.$ It allows us to identify the linear map $\mathrm{T}_t \mathbf{v}$
with the vector $\mathrm{T}_t \mathbf{v}(1),$ that represents the derivative of the curve $\mathbf{v}(t)$ by using local coordinates which is usually written as
$\dot{\mathbf{v}}(t)$ by abuse of notation.
}
\end{equation}%
where $\dot{\lambda}(t)\in \mathbb{R}$ and $\dot{v}_{\alpha}(t) \in W_{\alpha}^{\min}(\mathbf{v}_0)$ for $t\in I$ and $\alpha \in D.$
Moreover, if we assume that for each $\alpha \in D,$
$V_{\alpha}$ is a Hilbert space and $v_{\alpha}(t)\in \mathbb{S}_{V_{\alpha}},$ i.e., $\Vert v_{\alpha}(t)\Vert
_{\alpha}=1$ for $t\in I,$ then $\dot{v}_{\alpha}(t)\in \mathbb{T}%
_{v_{\alpha}(t)}(\mathbb{S}_{V_{\alpha}})$ for $t\in I$ and $\alpha \in D.$
\end{lemma}

\begin{proof}
First of all, we recall that by the construction of $\mathcal{U}(\mathbf{v}%
_{0})$ it follows that $W_{\alpha}^{\min }(\mathbf{v}_{0})=W_{\alpha}^{\min }(\mathbf{v%
}(t))$ and that $U_{\alpha}^{\min }(\mathbf{v}_{0})=\mathrm{span}\{v_{0}^{(\alpha)}\}$
is linearly isomorphic to $U_{\alpha}^{\min }(\mathbf{v}(t))$ for all $t\in I$
and $\alpha \in D.$ Assume $\Theta _{\mathbf{v}_{0}}(\mathbf{v}%
(t))=(\lambda (t),L_{1}(t),\ldots ,L_{d}(t)),$ i.e., 
\begin{equation*}
\mathbf{v}(t):=\lambda (t)\bigotimes_{\alpha \in D}\left( id_{\alpha}+L_{\alpha}(t)\right)
(v_{0}^{(\alpha)}),
\end{equation*}%
where $\lambda \in \mathcal{C}^{1}(I,\mathbb{R}\setminus \{0\}),$ $L_{\alpha}\in 
\mathcal{C}^{1}(I,\mathcal{L}(U_{\alpha}^{\min }(\mathbf{v}_{0}),W_{\alpha}^{\min }(%
\mathbf{v}_{0})))$ and $(id_{\alpha}+L_{\alpha}(t))(v_{0}^{(\alpha)})\in U_{\alpha}^{\min }(%
\mathbf{v}(t))$ for $\alpha \in D.$ We point out that the linear map $%
\mathrm{T}_{t}\mathbf{v}:\mathbb{R}\rightarrow \mathbb{T}_{\mathbf{v}(t)}(%
\mathfrak{M}_{(1,\ldots ,1)}(\mathbf{V}))$ is characterised by 
\begin{equation}
\mathrm{T}_{t}\mathbf{v}(1)=(\Theta _{\mathbf{v}_{0}}\circ \mathbf{v}%
)^{\prime }(t)=(\dot{\lambda}(t),\dot{L}_{1}(t),\ldots ,\dot{L}_{d}(t)).
\label{tangent_derivative}
\end{equation}%
Since $L_{\alpha}\in \mathcal{C}^{1}(I,\mathcal{L}(U_{\alpha}^{\min }(\mathbf{v}%
_{0}),W_{\alpha}^{\min }(\mathbf{v}_{0})))$ then $\dot{L}_{\alpha}\in \mathcal{C}%
^{0}(I,\mathcal{L}(U_{\alpha}^{\min }(\mathbf{v}_{0}),W_{\alpha}^{\min }(\mathbf{v}%
_{0}))).$ Observe that $U_{\alpha}^{\min }(\mathbf{v}_{0})$ and $U_{\alpha}^{\min }(%
\mathbf{v}(t))$ have $W_{\alpha}^{\min }(\mathbf{v}_{0})$ as a common complement.
From Lemma~\ref{Char_Projections} we know that 
\begin{equation*}
P_{U_{\alpha}^{\min }(\mathbf{v}_{0})\oplus W_{\alpha}^{\min }(\mathbf{v}%
_{0})}|_{U_{\alpha}^{\min }(\mathbf{v}(t))}:U_{\alpha}^{\min }(\mathbf{v}%
(t))\longrightarrow U_{\alpha}^{\min }(\mathbf{v}_{0})
\end{equation*}%
is a linear isomorphism. We can write 
\begin{equation*}
L_{\alpha}(t)=L_{\alpha}(t)P_{U_{\alpha}^{\min }(\mathbf{v}_{0})\oplus W_{\alpha}^{\min }(%
\mathbf{v}_{0})}\text{ and }\dot{L}_{\alpha}(t)=\dot{L}_{\alpha}(t)P_{U_{\alpha}^{\min }(%
\mathbf{v}_{0})\oplus W_{\alpha}^{\min }(\mathbf{v}_{0})},
\end{equation*}%
and then in \eqref{tangent_derivative} we identify $\dot{L}_{\alpha}(t)\in 
\mathcal{L}(U_{\alpha}^{\min }(\mathbf{v}_{0}),W_{\alpha}^{\min }(\mathbf{v}_{0})))$
with 
\begin{equation*}
\dot{L}_{\alpha}(t)P_{U_{\alpha}^{\min }(\mathbf{v}_{0})\oplus W_{\alpha}^{\min }(\mathbf{v}%
_{0})}|_{U_{\alpha}^{\min }(\mathbf{v}(t))}\in \mathcal{L}(U_{\alpha}^{\min }(\mathbf{v%
}(t)),W_{\alpha}^{\min }(\mathbf{v}_{0}))).
\end{equation*}%
Introduce $v_{\alpha}(t):=(id_{\alpha}+L_{\alpha}(t))(v_{0}^{(\alpha)})$ for $\alpha \in D.$
Then 
\begin{equation*}
\dot{L}_{\alpha}(t)(v_{\alpha}(t))=\dot{L}_{\alpha}(t)P_{U_{\alpha}^{\min }(\mathbf{v}%
_{0})\oplus W_{\alpha}^{\min }(\mathbf{v}_{0})}|_{U_{\alpha}^{\min }(\mathbf{v}%
(t))}(v_{0}^{(\alpha)}+L_{\alpha}(t)(v_{0}^{(\alpha)}))=\dot{L}_{\alpha}(t)(v_{0}^{(\alpha)})
\end{equation*}%
holds for all $t\in I$ and $\alpha \in D.$ Hence 
\begin{equation}
\dot{v}_{\alpha}(t)=\dot{L}_{\alpha}(t)(v_{0}^{(\alpha)})=\dot{L}_{\alpha}(t)(v_{\alpha}(t))
\label{derivative}
\end{equation}%
holds for all $t\in I$ and $\alpha \in D.$ From Lemma \ref%
{characterization_tangent_map}(b) and \eqref{tangent_derivative}, we have 
\begin{equation*}
\mathrm{T}_{\mathbf{v}(t)}\mathfrak{i}(\mathrm{T}_{t}\mathbf{v}(1))=\dot{%
\lambda}(t)\bigotimes_{\alpha \in D}v_{\alpha}(t)+\lambda (t)\sum_{\alpha \in D}\dot{L}%
_{\alpha}(t)(v_{\alpha}(t))\otimes \bigotimes_{\beta\neq \alpha}v_{\beta}(t),
\end{equation*}%
and, by using \eqref{derivative} for $\mathbf{v}(t)=\lambda
(t)\bigotimes_{\alpha \in D}v_{\alpha}(t),$ we obtain \eqref{curve_derivative}.

To prove the second statement, recall that $U_{\alpha}^{\min }(\mathbf{v}(t))=%
\mathrm{span}\,\{v_{\alpha}(t)\}$ and $V_{\alpha}=U_{\alpha}^{\min }(\mathbf{v}(t))\oplus
W_{\alpha}^{\min }(\mathbf{v}_{0})$ for $\alpha \in D.$ Let
$\left\langle \cdot,\cdot\right\rangle _{\alpha}$ be the scalar product defined
on $V_{\alpha}$ $\left( \alpha \in D \right).$
Then we consider 
\begin{equation*}
W_{\alpha}^{\min }(\mathbf{v}_{0})=\mathrm{span}\,\{v_{\alpha}(t)\}^{\bot }=\{u_{\alpha}\in
V_{\alpha}:\langle u_{\alpha},v_{\alpha}(t)\rangle _{\alpha}=0\}\text{ for }\alpha \in D,
\end{equation*}%
and hence $\langle \dot{v}_{\alpha}(t)),v_{\alpha}(t)\rangle _{\alpha}=0$ holds for $\alpha\in D.$ From Remark~\ref{unit_sphere}, we have $(\dot{v}_{1}(t),\ldots ,%
\dot{v}_{d}(t))\in \mathcal{C}(I,%
\mathop{\mathchoice{\raise-0.22em\hbox{\huge
$\times$}} {\raise-0.05em\hbox{\Large $\times$}}{\hbox{\large
$\times$}}{\times}}_{\alpha \in D}\mathbb{T}_{v_{\alpha}(t)}(\mathbb{S}_{V_{\alpha}})),$
because $W_{\alpha}^{\min }(\mathbf{v}_{0})=\mathbb{T}_{v_{\alpha}(t)}(\mathbb{S}%
_{V_{\alpha}})$ for $\alpha \in D.$
\end{proof}

\bigskip

The next result, where we assume that $\lambda(t) = \lambda_0 = 1$
holds for all time $t,$ gives us the time dependent Hartree method 
on tensor Banach (not necessarily Hilbert) spaces (compare with
Theorem~3.1 in \cite{Lubish}).

\begin{theorem}[Time dependent Hartree method on tensor Banach spaces]
The solution $\mathbf{v}_{r}(t)=\bigotimes_{\alpha \in D}v_{\alpha}(t),$ with 
$(v_{1}(t),\ldots ,v_{d}(t))\in 
\mathop{\mathchoice{\raise-0.22em\hbox{\huge
$\times$}} {\raise-0.05em\hbox{\Large $\times$}}{\hbox{\large
$\times$}}{\times}}_{\alpha \in D} V_{\alpha}$, of 
\begin{align*}
\dot{\mathbf{v}}_{r}(t)& =\mathcal{P}_{\mathbf{v}_{r}(t)}(A\mathbf{v}_{r}(t))%
\text{ for }t\in I, \\
\mathbf{v}_{r}(0)& =\mathbf{v}_{0},
\end{align*}%
satisfies 
\begin{equation*}
\left\langle \dot{w}%
_{\alpha}\otimes (\bigotimes_{\substack{ \beta \in D \\ \beta \neq \alpha}}v_{\beta}(t)),
J(\dot{\mathbf{v}}_{r}(t)-A\mathbf{v}_{r}(t))
\right\rangle =0 \text{\quad for all }\dot{w}_{\alpha}\in V_{\alpha},\quad \alpha \in D.
\end{equation*}%
\end{theorem}

\begin{proof}
From Lemma~\ref{previous_HF} we have $\mathbb{T}_{\mathbf{v}_{r}(t)}\left( 
\mathfrak{M}_{(1,\ldots ,1)}(\mathbf{V})\right) =\mathbb{R}\times 
\mathop{\mathchoice{\raise-0.22em\hbox{\huge
$\times$}} {\raise-0.05em\hbox{\Large $\times$}}{\hbox{\large
$\times$}}{\times}}_{\alpha \in D}
W_{\alpha}^{\min}(\mathbf{v}_0),$
Thus, for each $\dot{\mathbf{w}} \in \mathrm{T}_{\mathbf{v}(t)}i\left( 
\mathbb{T}_{\mathbf{v}(t)}\left( \mathfrak{M}_{(1,\ldots ,1)}(\mathbf{V}%
)\right) \right) $ there exists $(\dot{\varpi},\dot{w}_{1},\ldots ,\dot{%
w}_{d})\in \mathbb{R}\times 
\mathop{\mathchoice{\raise-0.22em\hbox{\huge $\times$}}
{\raise-0.05em\hbox{\Large $\times$}}{\hbox{\large $\times$}}{\times}}%
_{\alpha \in D}W_{\alpha}^{\min}(\mathbf{v}_0),$ such that 
\begin{equation*}
\dot{\mathbf{w}}=\dot{\varpi}\bigotimes_{\alpha \in D}v_{\alpha}(t)+
\sum_{\alpha \in D}\dot{w}_{\alpha}\otimes (\bigotimes_{\substack{ \beta \in D \\ \beta \neq \alpha}}v_{\beta}(t)).
\end{equation*}%
Observe that \eqref{eq:11} holds if and only if 
\begin{equation*}
\left\langle \dot{\varpi}%
\bigotimes_{\alpha \in D}v_{\alpha}(t)+ \sum_{\alpha \in D}\dot{w}%
_{\alpha}\otimes (\bigotimes_{\substack{ \beta \in D \\ \beta \neq \alpha}}v_{\beta}(t)),
J(\dot{\mathbf{v}}_{r}(t)-A\mathbf{v}_{r}(t))
\right\rangle =0
\end{equation*}%
for all $(\dot{\varpi},\dot{w}_{1},\ldots ,\dot{w}_{d})\in \mathbb{R}%
\times 
\mathop{\mathchoice{\raise-0.22em\hbox{\huge $\times$}}
{\raise-0.05em\hbox{\Large $\times$}}{\hbox{\large $\times$}}{\times}}%
_{\alpha \in D} W_{\alpha}^{\min}(\mathbf{v_0}).$ In particular, for a fixed
$\alpha \in D$ take $\dot{w}_{\beta} = 0$ for all $\beta \neq \alpha$ and $\dot{\varpi}=0$ then
$$
\left\langle \dot{w}%
_{\alpha}\otimes (\bigotimes_{\substack{ \beta \in D \\ \beta \neq \alpha}}v_{\beta}(t)),
J(\dot{\mathbf{v}}_{r}(t)-A\mathbf{v}_{r}(t))
\right\rangle =0
$$
holds for all $\dot{w}%
_{\alpha} \in W_{\alpha}^{\min}(\mathbf{v_0}).$ By taking
$\dot{\varpi} =1$ and  $\dot{w}_{\beta} = 0$ for all $\beta \in D$
it holds
$$
\left\langle 
\bigotimes_{\alpha \in D}v_{\alpha}(t),
J(\dot{\mathbf{v}}_{r}(t)-A\mathbf{v}_{r}(t))
\right\rangle =0.
$$
Since $U_{\alpha}^{\min }(\mathbf{v}(t))=%
\mathrm{span}\,\{v_{\alpha}(t)\}$ and $V_{\alpha}=U_{\alpha}^{\min }(\mathbf{v}(t))\oplus
W_{\alpha}^{\min }(\mathbf{v}_{0})$ for $\alpha \in D$ the theorem follows.
\end{proof}

\bigskip

Let $\left\langle \cdot,\cdot\right\rangle _{\alpha}$ be a scalar product defined
on $V_{\alpha}$ $\left( \alpha \in D \right) $, i.e., $V_{\alpha}$ is a pre-Hilbert
space. Then $\mathbf{V}=\left. _{a}\bigotimes_{\alpha \in D}V_{\alpha}\right. $ is
again a pre-Hilbert space with a scalar product which is defined for
elementary tensors $\mathbf{v}=\bigotimes_{\alpha \in D}v^{(\alpha)}$ and $\mathbf{w}%
=\bigotimes_{\alpha \in D}w^{(\alpha)}$ by%
\begin{equation}
\left\langle \mathbf{v,w}\right\rangle =\left\langle
\bigotimes_{\alpha \in D}v^{(\alpha)},\bigotimes_{\alpha \in D}w^{(\alpha)}\right\rangle
:=\prod_{\alpha \in D}\left\langle v^{(\alpha)},w^{(\alpha)}\right\rangle _{\alpha}\qquad\text{%
for all }v^{(\alpha)},w^{(\alpha)}\in V_{\alpha}.
\label{(Skalarprodukt fur Elementarprodukte}
\end{equation}
This bilinear form has a unique extension $\left\langle \cdot,\cdot
\right\rangle :\mathbf{V}\times\mathbf{V}\rightarrow\mathbb{R}.$ One
verifies that $\left\langle \cdot,\cdot\right\rangle $ is a scalar product,
called the \emph{induced scalar product}. Let $\mathbf{V}$ be equipped with
the norm $\left\Vert \cdot\right\Vert $ corresponding to the induced scalar
product $\left\langle \cdot,\cdot\right\rangle .$ As usual, the Hilbert
tensor space $\mathbf{V}_{\left\Vert \cdot\right\Vert }=\left. _{\left\Vert
\cdot \right\Vert }\bigotimes_{\alpha \in D}V_{\alpha}\right. $ is the completion of $%
\mathbf{V}$ with respect to $\left\Vert \cdot\right\Vert $. Since the norm $%
\left\Vert \cdot\right\Vert $ is derived via \eqref{(Skalarprodukt fur
Elementarprodukte}, it is easy to see that $\left\Vert \cdot\right\Vert $
is a reasonable and even uniform crossnorm. Moreover,
without loss of generality, we can assume $\Vert v_{0}^{(\alpha)}\Vert
_{\alpha}=1$ for $\alpha \in D.$

Before stating the next result, we introduce for $\mathbf{v}_{r}(t)=\lambda
(t)\bigotimes_{\alpha \in D}v_{\alpha}(t)$ the following time dependent bilinear forms 
\begin{equation*}
\mathrm{a}_{\alpha}(t;\cdot ,\cdot ):V_{\alpha}\times V_{\alpha}\longrightarrow \mathbb{R},
\end{equation*}%
defined by 
\begin{equation*}
\mathrm{a}_{\alpha}(t;z_{\alpha},y_{\alpha}):=\left\langle A\Big( z_{\alpha}\otimes
\bigotimes_{\substack{ \beta \in D \\ \beta \neq \alpha}}v_{\beta}(t)\Big) ,\Big( y_{\alpha}\otimes \bigotimes_{\substack{ \beta \in D \\ \beta \neq \alpha}}v_{\beta}(t)\Big) \right\rangle
\end{equation*}%
for each $\alpha \in D.$ Now, we will show the next result (compare with
Theorem~3.1 in \cite{Lubish}).

\begin{theorem}[Time dependent Hartree method on tensor Hilbert spaces]
The solution $\mathbf{v}_{r}(t)=\lambda (t)\bigotimes_{\alpha \in D}v_{\alpha}(t)$ with 
$(v_{1}(t),\ldots ,v_{d}(t))\in 
\mathop{\mathchoice{\raise-0.22em\hbox{\huge
$\times$}} {\raise-0.05em\hbox{\Large $\times$}}{\hbox{\large
$\times$}}{\times}}_{\alpha \in D}\mathbb{S}_{V_{\alpha}}$, of 
\begin{align*}
\dot{\mathbf{v}}_{r}(t)& =\mathcal{P}_{\mathbf{v}_{r}(t)}(A\mathbf{v}_{r}(t))%
\text{ for }t\in I, \\
\mathbf{v}_{r}(0)& =\mathbf{v}_{0},
\end{align*}%
satisfies 
\begin{equation*}
\langle \dot{v}_{\alpha}(t),\dot{w}_{\alpha}\rangle _{\alpha}-\mathrm{a}_{\alpha}(t;v_{\alpha}(t),%
\dot{w}_{\alpha}) =0\text{\quad for all }\dot{w}_{\alpha}\in \mathbb{T}%
_{v_{\alpha}(t)}(\mathbb{S}_{V_{\alpha}}),\quad \alpha \in D,
\end{equation*}%
and 
\begin{equation*}
\lambda (t)=\lambda _{0}\exp \left( \int_{0}^{t}\left\langle A\left( \otimes
_{\alpha \in D}v_{\alpha}(s)\right) ,\otimes _{\alpha \in D}v_{\alpha}(s)\right\rangle ds\right)
.
\end{equation*}
\end{theorem}

\begin{proof}
From Lemma~\ref{previous_HF} we have $\mathbb{T}_{\mathbf{v}_{r}(t)}\left( 
\mathfrak{M}_{(1,\ldots ,1)}(\mathbf{V})\right) =\mathbb{R}\times 
\mathop{\mathchoice{\raise-0.22em\hbox{\huge
$\times$}} {\raise-0.05em\hbox{\Large $\times$}}{\hbox{\large
$\times$}}{\times}}_{\alpha \in D}\mathbb{T}_{v_{\alpha}(t)}(\mathbb{S}_{V_{\alpha}}),$
Thus, for each $\dot{\mathbf{w}} \in \mathrm{T}_{\mathbf{v}(t)}i\left( 
\mathbb{T}_{\mathbf{v}(t)}\left( \mathfrak{M}_{(1,\ldots ,1)}(\mathbf{V}%
)\right) \right) $ there exists $(\dot{\varpi},\dot{w}_{1},\ldots ,\dot{%
w}_{d})\in \mathbb{R}\times 
\mathop{\mathchoice{\raise-0.22em\hbox{\huge $\times$}}
{\raise-0.05em\hbox{\Large $\times$}}{\hbox{\large $\times$}}{\times}}%
_{\alpha \in D}\mathbb{T}_{v_{\alpha}(t)}(\mathbb{S}_{V_{\alpha}}),$ such that 
\begin{equation*}
\dot{\mathbf{w}}=\dot{\varpi}\bigotimes_{\alpha \in D}v_{\alpha}(t)+\lambda
(t)\sum_{\alpha \in D}\dot{w}_{\alpha}\otimes (\bigotimes_{\substack{ \beta \in D \\ \beta \neq \alpha}}v_{\beta}(t)).
\end{equation*}%
Then \eqref{eq:11} holds if and only if 
\begin{equation*}
\left\langle\dot{\varpi}%
\bigotimes_{\alpha \in D}v_{\alpha}(t)+\lambda (t)\sum_{\alpha \in D}\dot{w}%
_{\alpha}\otimes (\bigotimes_{\substack{ \beta \in D \\ \beta \neq \alpha}}v_{\beta}(t)), \dot{\mathbf{v}}_{r}(t)-A\mathbf{v}_{r}(t)\right\rangle =0
\end{equation*}%
for all $(\dot{\varpi},\dot{w}_{1},\ldots ,\dot{w}_{d})\in \mathbb{R}%
\times 
\mathop{\mathchoice{\raise-0.22em\hbox{\huge $\times$}}
{\raise-0.05em\hbox{\Large $\times$}}{\hbox{\large $\times$}}{\times}}%
_{\alpha \in D}\mathbb{T}_{v_{\alpha}(t)}(\mathbb{S}_{V_{\alpha}}).$ Then 
\begin{align*}
\dot{\lambda}(t)\dot{\varpi}+\lambda (t)^{2}\sum_{\alpha \in D}\left( \langle 
\dot{v}_{\alpha}(t),\dot{w}_{\alpha}\rangle _{\alpha}-\langle
A\bigotimes_{\mu \in D}v_{\mu}(t),\dot{w}_{\alpha}\otimes (\bigotimes_{\substack{ \beta \in D \\ \beta \neq \alpha}}v_{\beta}(t))\rangle \right)& \\
-\lambda (t)\dot{\varpi}\langle
A\bigotimes_{\alpha \in D}v_{\alpha}(t),\bigotimes_{\alpha \in D}v_{\alpha}(t)\rangle &=0,
\end{align*}%
i.e., 
\begin{equation}  \label{HF5}
\begin{array}{l}
\dot{\varpi}\left( \dot{\lambda}(t)-\lambda (t)\langle
A\bigotimes_{\alpha \in D}v_{\alpha}(t),\bigotimes_{\alpha \in D}v_{\alpha}(t)\rangle \right) \\ 
+\lambda (t)^{2}\sum_{\alpha \in D}\left( \langle \dot{v}_{\alpha}(t),\dot{w}%
_{\alpha}\rangle _{\alpha}-\langle A\bigotimes_{\mu \in D}v_{\mu}(t),\dot{w}%
_{\alpha}\otimes (\bigotimes_{\substack{ \beta \in D \\ \beta \neq \alpha}}v_{\beta}(t))\rangle \right) =0%
\end{array}%
\end{equation}%
holds for all $\dot{\varpi}\in \mathbb{R}$ and $(\dot{w}_{1},\ldots ,%
\dot{w}_{d})\in 
\mathop{\mathchoice{\raise-0.22em\hbox{\huge $\times$}}
{\raise-0.05em\hbox{\Large $\times$}}{\hbox{\large $\times$}}{\times}}%
_{\alpha \in D}\mathbb{T}_{v_{\alpha}(t)}(\mathbb{S}_{V_{\alpha}}).$ If $\lambda (t)$
solves the differential equation 
\begin{align*}
\dot{\lambda}(t)& =\left\langle A\left( \otimes _{\alpha \in D}v_{\alpha}(t)\right)
,\otimes _{\alpha \in D}v_{\alpha}(t)\right\rangle \lambda (t) \\
\lambda (0)& =\lambda _{0},
\end{align*}%
i.e., 
\begin{equation*}
\lambda (t)=\lambda _{0}\exp \left( \int_{0}^{t}\left\langle A\left( \otimes
_{\alpha \in D}v_{\alpha}(s)\right) ,\otimes _{\alpha \in D}v_{\alpha}(s)\right\rangle ds\right)
,
\end{equation*}%
then the first term of \eqref{HF5} is equal to $0.$ Therefore, from (\ref%
{HF5}) we obtain that for all $\alpha \in D$, 
\begin{equation*}
\langle \dot{v}_{\alpha}(t),\dot{w}_{\alpha}\rangle _{\alpha}-\langle
A\bigotimes_{\mu \in D}v_{\mu}(t),\dot{w}_{\alpha}\otimes (\bigotimes_{\substack{ \beta \in D \\ \beta \neq \alpha}}v_{\beta}(t))\rangle =0,
\end{equation*}%
that is, 
\begin{equation*}
\langle \dot{v}_{\alpha}(t),\dot{w}_{\alpha}\rangle _{\alpha}-\mathrm{a}_{\alpha}(t;v_{\alpha}(t),%
\dot{w}_{\alpha})=0
\end{equation*}%
holds for all $\dot{w}_{\alpha}\in \mathbb{T}_{v_{\alpha}(t)}(\mathbb{S}_{V_{\alpha}}),$
and the theorem follows.
\end{proof}

\subsection{Concluding Remarks}

We would point out that when we assume that $V_{\alpha} = V$ for all $\alpha \in D$ then 
the theory presented above covers the classical MCTDH approximation for molecules (see for example 
Section 1.9 in \cite{Hartree}). In fact the approximate wave function $\mathbf{v}_r(t)$ 
computed on $\mathfrak{M}_{(1,\ldots,1)}(\mathbf{V})$ 
does not conform the Pauli's exclusion principle. To take into account 
the antisymmetry of the wave function we need to use the
so-called multi--configuration time--dependent Hartree-Fock (MCTDHF) approximation. The MCTDHF 
is based on the use of the so-called Hartree-Fock manifold. This manifold is constructed
by using the existence of a projection $P_{\mathfrak{S}}$ from
$\mathbf{V}$ to the linear subspace of antisymmetric tensors 
of $\mathbf{V}$ (corresponding to fermions). Then the Hartree-Fock manifold is defined as
$$
\mathfrak{M}_{(1,\ldots,1)}^{\mathfrak{A}}(\mathbf{V}):=\left\{
P_{\mathfrak{S}}(v_1 \otimes \cdots \otimes v_d): v_{\alpha} \in V_{\alpha},\, \alpha \in D
\right\}.
$$ 
In a similar way, the use of a projection onto the linear subspace of symmetric tensors of $\mathbf{V}$ 
(corresponding to bosons) allows us to introduce a manifold, 
namely $\mathfrak{M}_{(1,\ldots,1)}^{\mathfrak{S}}(\mathbf{V}).$ 
The extension of the results given in this paper to $\mathfrak{M}_{(1,\ldots,1)}^{\mathfrak{A}}(\mathbf{V})$ and $\mathfrak{M}_{(1,\ldots,1)}^{\mathfrak{S}}(\mathbf{V})$ 
are part of a work in progress and will be published elsewhere.

\end{document}